\documentclass[11pt]{TPmod}

\usepackage{tikz}
\usepackage[T1]{fontenc}
\usepackage[all]{xy}

\usepackage[shortlabels]{enumitem}

\usepackage[parfill]{parskip}    
\usepackage{latexsym, amssymb, bbm}
\usepackage{amsthm}
\usepackage{xypic}
\usepackage{xcolor}
\usepackage{tikz, tikz-cd}
\usepackage{graphicx}
\usepackage{amsmath}
\usepackage{upgreek}

\begingroup
\makeatletter
   \@for\theoremstyle:=definition,remark,plain\do{%
     \expandafter\g@addto@macro\csname th@\theoremstyle\endcsname{%
        \addtolength\thm@preskip\parskip
     }%
   }
\endgroup

\newtheorem{remark}[thm]{Remark}
\newtheorem{hyp1}[mainthm]{Hypothesis}

\DeclareMathOperator{\CZ}{CZ}

\newcommand{\wt}[1]{\widetilde{#1}}

\newcommand{\norm}[1]{\left\| #1 \right\|}

\numberwithin{equation}{section}

\def\iff{\Leftrightarrow}

\def\C{\mathbb{C}}

\def\Q{\mathbb{Q}}
\def\R{\mathbb{R}}
\def\Z{\mathbb{Z}}

\def\CP{\mathbb{CP}}

\def\bD{\mathbb{D}}

\def\cA{\mathcal{A}}
\def\cC{\mathcal{C}}

\def\cF{\mathcal{F}}

\def\cJ{\mathcal{J}}
\def\cK{\mathcal{K}}
\def\cL{\mathcal{L}}

\def\cN{\mathcal{N}}
\def\cO{\mathcal{O}}
\def\cP{\mathcal{P}}
\def\cQ{\mathcal{Q}}

\def\a{\alpha}

\def\eps{\epsilon}

\def\s{\sigma}

\def\w{\omega}

\def\bfD{\mathbf{D}}

\def\bflambda{\boldsymbol{\lambda}}
\def\bfkappa{\boldsymbol{\kappa}}

\def\ginv{f}
\def\Nov{\Lambda}

\begin{document}
\title{Quantum cohomology as a deformation of symplectic cohomology}

\author{Matthew Strom Borman,  Nick Sheridan, 
 and Umut Varolgunes
}

\address{Matthew Strom Borman, Department of Mathematics, University of Illinois,  1409 W. Green Street, Urbana, IL 61801, U.S.A.}
\address{Nick Sheridan, School of Mathematics, University of Edinburgh, Edinburgh EH9 3FD, U.K.}
\address{Umut Varolgunes, School of Mathematics, University of Edinburgh, Edinburgh EH9 3FD, U.K.}

\begin{abstract}
We prove that under certain conditions, the quantum cohomology of a positively monotone compact symplectic manifold is a deformation of the symplectic cohomology of the complement of a simple crossings symplectic divisor. 
We also prove rigidity results for the skeleton of the divisor complement. 
\end{abstract}

\maketitle

\setlength{\parindent}{0pt}


\section{Introduction}

\subsection{Geometric setup}\label{subsec:setup}

Let $(M,\w)$ be a compact symplectic manifold satisfying the monotonicity condition: $$ 2 \kappa c_1(TM) =  [\w], \text{ for some }\kappa>0$$ in $H^2(M;\R)$.  

Let $D=\cup_{i=1}^N D_i \subset M$ be an SC symplectic divisor (in the sense of \cite[Section 2]{Tehrani2018a}) and set $X = M \setminus D$.\footnote{From now on, we systematically shorten SC symplectic divisor to SC divisor as we believe this will not cause confusion.}

We assume that there exist positive rational numbers $\lambda_1,\ldots ,\lambda_N$ called \textit{weights} such that $$2c_1(TM)=\sum_i \lambda_i \cdot \mathrm{PD}(D_i)\quad\text{  in  $H^2(M;\R)$.}$$ Note that the number of weights in the setup depends on the divisor. If $\mathrm{PD}(D_i)$ are linearly independent classes in $H^2(M;\R)$ (e.g., if $D$ is smooth), the weights are canonically determined. Otherwise, the choice of weights is extra data.

The classes $\mathrm{PD}(D_i)$ have canonical lifts to the relative cohomology group $H^2(M,X;\R)$ along the canonical map $$H^2(M,X;\R)\to H^2(M;\R),$$see Section \ref{ss=div-basic} for more details. Let us denote these classes by $\mathrm{PD}^{rel}(D_i)$ and note that they form a basis of $H^2(M,X;\R).$ We define the class \[ \bflambda := \sum_i \lambda_i\cdot \mathrm{PD}^{rel}(D_i) \in H^2(M,X;\R), \] which is a lift of $2c_1(TM)$ by construction. Consequently, $\kappa\bflambda$ is a lift of $[\omega]$.

Let us denote by $\Omega^*(M,X)$ the relative de Rham complex, which is by definition the cone of the restriction map $\Omega^*(M) \to \Omega^*(X)$. Note that there is a relative de Rham isomorphism $$H^*(\Omega^*(M,X))\to H^*(M,X;\R),$$ which in particular tells us that there exists a one-form $\theta \in \Omega^1(X)$ satisfying 
\[ \w|_X = d\theta,\qquad \text{and} \qquad \int_u \omega - \int_{\partial u} \theta = \kappa\bflambda \cdot u \qquad \text{for all $u \in H_2(M,X)$.}\]

Using that $\kappa\lambda_i>0$ for all $i$, we may arrange that $(X,\theta)$ is a finite type convex symplectic manifold. Moreover $c_1(TX)=0$, and a preferred homotopy class of trivializations of a power of the canonical bundle of $X$ is determined by the choice of weights $\lambda_i$ (see Section \ref{subsec:SH} for details).

\begin{example}\label{eg:proj}
Suppose that $M$ is a smooth complex projective variety, $D$ a simple normal crossings divisor, and $\bfD^{\bflambda} = \sum_i \lambda_i D_i$ is an effective ample $\mathbb{Q}$-divisor whose class in $\mathrm{CH}_*(M)_\Q$ is twice the anticanonical class: $$[\bfD^{\bflambda}]=-2K_M.$$ Let us also choose an arbitrary $\kappa>0$.

Choose a positive integer $k$ such that $k\bfD^{\bflambda}$ has integral coefficients, and let $\cL$ be the corresponding complex line bundle with section $s$. Ampleness implies that we can choose a positive Hermitian metric $\norm{\cdot}$ on $\cL$ with curvature $2$-form $F$. We define the symplectic form $\omega:=\frac{-i\kappa}{k}F$ on $M$. We can also define the primitive $\theta:=\frac{\kappa}{k}d^c\log\norm{s}$ of $\omega$ on $X$. Using $D$ as our SC divisor and $\lambda_i$ as the weights, this puts us in the geometric set-up described above.
Note that in this case $X$ is an affine variety.
\end{example}

The setup that we described thus far is among the most studied in symplectic geometry. Now, we introduce an hypothesis which is less common, but which is very crucial for our results.

\begin{hyp1}\label{hyp:gr}
We have $\lambda_i\leq 2$ for all $i=1,\ldots,  N$. 
\end{hyp1}

\begin{rmk} 
Recalling that $[\bfD^{\bflambda}] = -2K_M$, we note that the extreme case of Hypothesis \ref{hyp:gr}, namely $\lambda_i = 2$ for all $i$, corresponds to $(M,D)$ being log Calabi--Yau. 
If we in addition assume that each irreducible component of $D$ is ample, then Hypothesis \ref{hyp:gr} implies that $(M,D)$ is either log Calabi--Yau or log general type. \end{rmk}

\begin{example}\label{eg:cpn}
Consider the setup of Example \ref{eg:proj}. Let us take $M=\mathbb{C}\mathbb{P}^n$, $D$ a simple normal crossings divisor of degree $d$. 
Then we may choose weights $\lambda_i$ such that Hypothesis \ref{hyp:gr} holds if and only if $d \ge n+1$. 
Note that $(M,D)$ is log Calabi--Yau if $d=n+1$, and log general type if $d>n+1$.
To see one direction of the implication, assume that $D = \cup_{i=1}^N D_i$ with $D_i$ smooth of degree $d_i$ and Hypothesis \ref{hyp:gr} holds. Then we have
$$2(n+1) = \sum_i \lambda_i d_i \le \sum_i 2d_i = 2d.$$
Note that in the log Calabi--Yau case $d=n+1$, all weights $\lambda_i$ must be equal to $2$. 
\end{example}

\subsection{Quantum cohomology is a deformation of symplectic cohomology}\label{ss-defintro}

We fix, once and for all, a commutative ring $\Bbbk$. Let $A \subset \Q$ be the subgroup generated by the weights $\lambda_i$, and set $\Nov = \Bbbk[A]$ to be the group algebra of $A$.\footnote{Explicitly, $\Nov$ is the $\Bbbk$-algebra of $\Bbbk$-linear combinations of the symbols $e^a$ where $a \in A$, and $e^a \cdot e^b = e^{a+b}$.}  
We define a $\Q$-grading on $\Nov$ by putting $e^a$ in degree $i(e^a) = a$. Let $a_0>0$ be a generator of $A$, and define $q:=e^{a_0}$. Hence, we have an isomorphism of algebras $\Nov\cong \Bbbk[q,q^{-1}]$.

Throughout the paper, we will consider various filtrations associated to filtration maps (see Section \ref{ss-filtration-maps} for a review of this notion). 
We will abuse notation by using the same symbol for the filtration map and the associated filtration. 
In the first instance of this abuse of notation, we introduce the filtration $\cQ_{\ge \bullet}$ on $\Nov$ associated to the filtration map $\cQ:\Nov \to \Z$ induced by $\cQ(q^a) = a$. 
Thus $\cQ_{\ge p} \Nov$ consists of all linear combinations of monomials $q^a$ with $a \ge p$.

We define the graded $\Nov$-module $QH^*(M;\Nov) := H^*(M;\Bbbk) \otimes_\Bbbk \Nov$, equipped with the tensor product grading.\footnote{Our main results do not concern the quantum cup product on $QH^*(M;\Nov)$, but it plays a role in some of the conjectures in Section \ref{s-conj}.} We are concerned with the following idealized and vague conjecture:

\begin{conj}\label{conj:def}
Under certain hypotheses:
\begin{enumerate}[(a)]
\item \label{it:a} $QH^*(M;\Nov)$ is the cohomology of a natural deformation of the symplectic cochain complex $SC^*(X;\Bbbk)$ over $\Nov$;
\item \label{it:b} Furthermore, the associated spectral sequence converges: $E_1 = SH^*(X;\Bbbk) \otimes_\Bbbk \Nov \Rightarrow QH^*(M;\Nov)$.
\end{enumerate}
\end{conj}

We will prove a modified version of Conjecture \ref{conj:def} in the set-up described in Section \ref{subsec:setup}. Notably, for the analogue of part (b) we will need Hypothesis \ref{hyp:gr}.

\begin{remark}
Conjecture \ref{conj:def} part (b) is not true in general.  For example, if we take $M = \CP^n$ and $D$ a hyperplane, then $X=M \setminus D = \C^n$ has vanishing symplectic cohomology. 
But we can't have a spectral sequence with vanishing $E_1$ page, converging to the non-vanishing cohomology of $\CP^n$! 
Note that Hypothesis \ref{hyp:gr} is not satisfied in this case by Example \ref{eg:cpn}. More generally it is not satisfied for $D$ a union of $N \le n$ hyperplanes; and $X= \C^{n+1-N} \times (\C^*)^{N-1}$ still has vanishing symplectic cohomology in these cases.
\end{remark}

Note that Conjecture \ref{conj:def} \ref{it:a} is a statement about the chain complex $SC^*(X;\Bbbk)$, which depends on various auxiliary data which we have not included in the notation. 
Given such a choice, we consider the chain complex \begin{equation}\label{eqdef-SCNov}SC^*_\Nov:=(SC^*(X;\Bbbk)\otimes_\Bbbk \Nov,d\otimes id_\Nov) \end{equation} with the tensor product grading.\footnote{In general this will be a $\mathbb{Q}$-grading.} 
It admits a $\cQ$-filtration induced by the filtration map $\cQ(\gamma \otimes r) = \cQ(r)$. 
In the modified version of Conjecture \ref{conj:def} \ref{it:a} that we prove, we will need to replace $SC_\Nov$ with an `equivalent' filtered complex $\wt{SC}_\Nov$:

\begin{main}[a modified version of Conjecture \ref{conj:def} \ref{it:a}]
\label{thm:aa}
Assume that we are in the setup described in Section \ref{subsec:setup}. 
Then there exists a choice of the auxiliary data needed to define $SC^*(X;\Bbbk)$, and a filtered cochain complex of $\cQ_{\ge 0}\Nov$-modules, $\wt{SC}_\Nov := \left(\widetilde{SC}^*_\Nov,\wt{d},\wt{\cQ}_{\ge \bullet}\right)$, with the following properties:
\begin{enumerate}
\item \label{it:1a} $\left(\wt{SC}^*_\Nov,\wt{d},\wt{\cQ}_{\ge \bullet}\right)$ is filtered quasi-isomorphic to $\left(SC^*_\Nov,d \otimes id_\Nov, \cQ_{\ge \bullet}\right)$ (see Section \ref{ss-thmaproof} for the precise meaning of this statement).
\item \label{it:2a} There exists a second $\cQ_{\ge 0}\Nov$-linear differential $\partial$ on $\wt{SC}^*_\Nov$, such that $\partial - \wt{d}$ strictly increases the $\wt{\cQ}$-filtration.  
We call $\partial$ the deformed differential.
\item \label{it:3a} We have $H^*\left(\wt{SC}^*_\Nov,\partial\right) \cong QH^*(M;\Nov)$. 
\end{enumerate}
\end{main}

By considering the spectral sequence associated to the deformed filtered complex $\left(\wt{SC}^*_\Nov,\partial,\wt{\cQ}_{\ge \bullet}\right)$, we then obtain:

\begin{main}[Conjecture \ref{conj:def} \ref{it:b}]\label{thm:specseqa}
Assume now that Hypothesis \ref{hyp:gr} holds. 
Then the spectral sequence associated to the filtered complex $(\wt{SC}_\Nov,\partial,\wt{\cQ}_{\ge \bullet})$ converges, and has $$E_1^{j,k}=SH^{k+j(1+a_0)}(X;\Bbbk) \otimes_\Bbbk \Bbbk \cdot q^{-j}\Rightarrow QH^{j+k}(M;\Nov),$$ where $j\in \mathbb{Z}$ and $k\in \mathbb{Q}$.
\end{main}

\begin{rmk}
Because our Floer complexes are $\Q$-graded, our spectral sequence $(E_i^{j,k},d_i^{j,k})$ has $i,j \in \Z$ and $k \in \Q$, rather than the usual $k \in \Z$. 
All the standard theory of spectral sequences goes through in this slightly more general context. 
Indeed, one can think of such a spectral sequence as a collection of ordinary spectral sequences indexed by $\{c \in \Q:0 \le c < 1\}$, by setting $E(c)_i^{j,k} = E_i^{j,k-c}$.
\end{rmk}

Let us note an immediate corollary:

\begin{cor}
Under Hypothesis \ref{hyp:gr}, the affine variety $X$ from Example \ref{eg:proj} has non-vanishing symplectic cohomology. In particular, it is not flexible.
\end{cor}

\begin{rmk}
We expect that analogues of Theorems \ref{thm:aa} and \ref{thm:specseqa} hold also under the assumption that $M$ is Calabi--Yau, i.e., $c_1(TM) = 0$. Indeed,  Yuhan Sun has recently proved very closely related results \cite{Sun2021}. 
In this case, the key notion is `index boundedness',  as used by McLean in \cite{McLean2020}, together with certain lower bounds on the indices of the one-periodic orbits `on the divisor'. We refer the reader to \cite{Sun2021} for more details.
\end{rmk}

\subsection{Rigidity results}\label{ss-rigid-sk}

By applying the same techniques as those used to prove Theorems \ref{thm:aa} and \ref{thm:specseqa}, 
we will prove a rigidity result for certain subsets of $M$.

The main tool used to prove the result is a version of the relative symplectic cohomology developed by the third author in \cite{Varolgunes2018} (with which we assume some familiarity). Slightly modifying the construction there, for any compact $K\subset M$, we can define a $\mathbb{Q}$-graded $\Nov$-module $SH_M^*(K;\Nov)$. 
The definition of this invariant involves choosing acceleration data to define a Floer $1$-ray, and then the chain-level invariant is defined to be not the telescope but a degreewise-completed telescope. 
More details are given in Section \ref{ss-rel-SH}.\footnote{
The construction that we give here can be generalized to all symplectic manifolds with the property $$ 2 c_1(TM) = \eta [\w] \quad\text{ for some }\eta\in \R,$$ and subgroup $B\subset\R$ which contains $\omega(\pi_2(M))$. 
Namely, we define the filtered graded algebra $\Bbbk[B]$ where $i(e^b)=\eta b$ and the filtration level of $e^b$ is $b$. We then define the Novikov-type algebra $$\Lambda_{B,\eta}:= \widehat{\Bbbk[B]},$$ where the completion is degreewise. Our $\Lambda$ in this paper is nothing but $\Lambda_{\kappa A,\kappa^{-1}}$, whereas the Novikov field used in e.g. \cite{Tonkonog2020} is  $\Lambda_{\R,0}.$ 
The construction produces a $\Z+\eta B $-graded $\Nov_{B,\eta}$-module $SH_M^*(K;\Nov_{B,\eta})$. When $c_1(M)=0$, and taking into account only the contractible orbits, the invariant that is denoted by $SH_M^*(K;\Nov)$ in \cite{Tonkonog2020} is a special case of this construction as well. It would have been called $SH_M^*(K;\Nov_{\mathbb{R},0})$ in our notation here, and a capped orbit $(\gamma,u)$ here would be interpreted as $T^{\cA(\gamma,u)}\gamma$ in that paper's notation. Let us also note that $\eta<0$ requires using virtual techniques, which forces us to make certain assumptions on $\Bbbk$.
}

Following \cite{Tonkonog2020}, we say that $K$ is \emph{$SH$-invisible} if $SH_M^*(K;\Nov)= 0$, and \emph{$SH$-visible} otherwise. 
One can prove that $SH$-visible subsets are not stably displaceable (see Theorem \ref{analog_thesis2}).\footnote{Stably displaceable means $K\times Z\subset M\times T^*S^1$ is displaceable from itself by a Hamiltonian diffeomorphism, where $Z$ is the zero section of $T^*S^1$.} For example, PSS isomorphisms imply that $QH^*(M;\Nov)\cong SH_M(M;\Nov)$, so $M$ is $SH$-visible; and as a result $M$ is not stably displaceable (as a subset of itself). 
Moreover, there are restriction maps $SH_M^*(K';\Nov)\to SH_M^*(K;\Nov)$, whenever $K'$ contains $K$. By a unitality argument, it follows that any compact subset of an $SH$-invisible subset is $SH$-invisible (see Theorem \ref{analog_thesis3}).

We say that $K$ is \emph{nearly $SH$-visible} if any compact domain that contains $K$ in its interior is $SH$-visible. 
As straightforward consequences of the previous paragraph, one can show that $SH$-visible subsets are nearly $SH$-visible,\footnote{We do not have examples of nearly $SH$-visible subsets which are not $SH$-visible.} and nearly $SH$-visible subsets are not stably displaceable.

We say that $K$ is \emph{$SH$-full} if for any compact $K'$ contained in $M\setminus K$, $SH_M^*(K';\Nov)= 0$. 
$SH$-full subsets are nearly $SH$-visible, as a consequence of the Mayer--Vietoris property of relative symplectic cohomology \cite{Varolgunes2018}. 
One can prove that an $SH$-full subset cannot be displaced from a nearly $SH$-visible subset by a symplectomorphism. It is also the case that $SH$-full subsets are not stably displaceable from themselves (see \cite[Corollary 1.9]{Tonkonog2020}). By a closed--open map plus unitality argument (see \cite[proof of Theorem 1.2 (6)]{Tonkonog2020}), it can be shown that Floer-theoretically essential (over $\Bbbk$) monotone Lagrangians cannot be displaced from $SH$-full subsets by a symplectomorphism.

Now we state our result, which will need some extra hypotheses beyond those already mentioned in Section \ref{subsec:setup}.
First of all, we assume that $D$ is an orthogonal SC divisor. 
Then there exist Hamiltonian circle actions rotating about the $D_i$, and commuting on the overlaps, by \cite{McL12T}; we assume that $\theta$ is `adapted' to such a system of commuting Hamiltonians, in an appropriate sense. We make these notions precise in Section \ref{s-basics-of-SC} below. 
We remark that the data we need is weaker than what McLean calls a `standard tubular neighbourhood' in \cite{McLean2020}.

Let $Z$ be the Liouville vector field of $(X,\theta)$. 
We define the continuous function $\rho^0: X \to \R$, so that the Liouville flow starting at $x$ is defined precisely for time $t <-\log(\rho^0(x))$. 
Note that $\mathbb{L} = \{\rho^0 = 0\}$ is the Lagrangian skeleton of $(X,\theta)$. 
We extend the function to $\rho^0:M \to \R$ by setting $\rho^0|_D = 1$.

\begin{defn}\label{def-crit}
We define 
$$\tilde{\s}_{crit} := 1-\frac{2}{\max_i \lambda_i},\qquad \s_{crit} := \max(0,\tilde{\s}_{crit}),$$
and set
$$K_{crit} := \{\rho^0 \le \s_{crit}\} \subset M.$$
Note that $\s_{crit} = 0$, and hence $K_{crit} = \mathbb{L}$, if and only if Hypothesis \ref{hyp:gr} is satisfied.
\end{defn}

Equivalently, $K_{crit}$ is the image of the Liouville flow for time $\log(\s_{crit})$.

\begin{main}\label{thm-superheavy}
The subset $K_{crit} \subset M$ is $SH$-full. 
In particular, if Hypothesis \ref{hyp:gr} is satisfied, then $\mathbb{L}$ is $SH$-full.
\end{main}

For example, this means that when Hypothesis \ref{hyp:gr} is satisfied, $\mathbb{L}$ cannot be displaced from any Floer-theoretically essential (over $\Bbbk$) monotone Lagrangian.

\begin{rmk}
It is possible for a compact subset to be $SH$-full for one choice of $\Bbbk$ but not for another. We did not make a big deal about this as our result is uniform for all ground fields. We expect this to play a real role in the context of Conjecture \ref{conj:ss-nohyp}. We also refer the reader to Remark 1.8 of \cite{Tonkonog2020} for another weakening of the notion of $SH$-fullness.
\end{rmk}

\begin{rmk}
An analogue of Theorem \ref{thm-superheavy}, in the case that $M$ is Calabi--Yau, was proved in \cite{Tonkonog2020}.
\end{rmk}

\subsection{Floer theory conventions}

We give a quick outline of our conventions for Hamiltonian Floer theory on $M$, for the purposes of giving an overview of the proofs of our main results in the following section (see Section \ref{s-HF-conventions} for more details). 
Let $A' \subset \Z$ be the image of $2c_1(TM): \pi_2(M) \to \Z$, and set $\Nov' := \Bbbk[A']$. 
Note $A' \subset A$, so $\Nov' \subset \Nov$. 

A `cap' for an orbit $\gamma:S^1 \to M$ of a Hamiltonian $H: S^1 \times M \to \R$ is an equivalence class of discs $u$ bounding $\gamma$, where two discs  are considered equivalent if they have the same symplectic area. 
One can associate an index $i(\gamma,u)$ and action $\cA(\gamma,u)$ to a capped orbit $(\gamma,u)$ of a non-degenerate Hamiltonian. 
The `mixed index' 
$$i_{mix}(\gamma) = i(\gamma,u) - \kappa^{-1} \cA(\gamma,u)$$
is independent of the cap $u$. 

We define $CF^i(M,H)$ to be the free $\Z$-graded $\Bbbk$-module generated by capped orbits $(\gamma,u)$ of $H$ satisfying $i(\gamma,u) = i$. 
This becomes a graded $\Nov'$-module, where $e^a \cdot (\gamma,u) = (\gamma,u \# B)$ where $2c_1(TM)(B) = a$. 
It also admits an action filtration, associated to the filtration map induced by $\cA(\gamma,u)$. 
We define $CF^*(M,H;\Nov) := CF^*(M,H) \otimes_{\Nov'} \Nov$. 
It has a $\Bbbk$-basis of `fractional caps': a fractional cap for $\gamma$ is a formal expression $ u + a$, where $u$ is a cap for $\gamma$ and $a \in A$, and we declare $u+a \sim u' + a'$ iff $a-a' \in A'$ and $(\gamma,u') = e^{a-a'} \cdot (\gamma,u)$. 

The Floer differential increases degree by $1$, and respects the action filtration (i.e., it does not decrease action). 
The PSS isomorphism identifies $HF^*(M,H;\Nov) \cong QH^*(M;\Nov)$.
If $H_1 \le H_2$ pointwise, then there exists a continuation map $CF^*(M,H_1;\Nov) \to CF^*(M,H_2;\Nov)$ which respects action filtrations.

We now explain our conventions for relative symplectic cohomology. 
Given $K \subset M$ compact, a choice of acceleration data $(H_\tau,J_\tau)$ is the data required to define a Floer 1-ray
$$
\mathcal{C}(H_\tau,J_\tau):= CF^*(M, H_1; \Nov)\to CF^*(M, H_2; \Nov)\to\ldots 
$$
consisting of Floer cohomology groups and continuation maps, where the monotone sequence of Hamiltonians $H_1 \le H_2 \le \ldots$ converges to $0$ on $K$ and $+\infty$ outside of $K$. 
We consider the telescope complex $tel(\cC)$, which is constructed so that 
$$H^*(tel(\cC)) = \varinjlim_i HF^*(M,H_i;\Nov) = QH^*(M;\Nov).$$ 
We define $\widehat{tel}(\cC)$ to be the degreewise completion of $tel(\cC)$ with respect to the action filtration, and $SH^*_M(K;\Nov) := H^*(\widehat{tel}(\cC))$.

\subsection{Outline of proofs}\label{ss-proofs}

In this section we give an extended overview of the proofs of our main results, trying to convey the main ideas while avoiding technicalities. 
We assume that we are in the geometric setup described in Section \ref{subsec:setup}, with the additional properties and data explained in Sections \ref{ss-defintro} and \ref{ss-rigid-sk}. 

We will construct a function $\rho:M \to \R$ which is a smoothing of $\rho^0$ (really, a family of smoothings $\rho^R$ for $R>0$ sufficiently small) with the following properties: 
\begin{itemize}
\item it will be continuous on $M$, and smooth on the complement of $\mathbb{L}$; 
\item $\rho|_{\mathbb{L}} = 0$ and $\rho|_D \approx 1$;\footnote{If $D$ is smooth then we can arrange that $\rho|_D = 1$; if $D$ is normal crossings then $\rho|_D$ will be equal to $1$ away from a neighbourhood of the singularities of $D$, where an error is introduced by `rounding corners'.}
\item it will satisfy $Z(\rho) = \rho$ on $X \setminus \mathbb{L}$. 
\end{itemize}
It also has the property that $K_\s := \{\rho \le \s\}$ is a Liouville subdomain of $X$ for any $\s \in (0,1)$. 
Because $Z(\rho) = \rho$, $K_\s \to \mathbb{L}$ as $\s \to 0$. 

\subsubsection{Theorem \ref{thm:aa}}

We choose $\s \in (\s_{crit},1)$, and construct acceleration data $(H_\tau,J_\tau)$ for $K_\s \subset M$ as follows. 
Fix $0<\ell_1< \ell_2<\ldots$ such that the Reeb flow on $Y=\partial K_\s$ has no $\ell_n$-periodic orbits for all $n$, and $\ell_n \to \infty$ as $n \to \infty$. 
We choose an increasing family of smooth functions $h_n : \R \to \R$, approximating the piecewise-linear functions $\max(0,\ell_n(\rho - \s))$ with increasing accuracy as $n \to \infty$, and being linear with slope $\ell_n$ for $\rho \ge \s$ (see Figure \ref{fig:B}).
We consider acceleration data $(H_\tau,J_\tau)$ for $K_\s\subset M$ such that $J_\tau$ is of contact type near $\partial K_\s$ and $H_n$ is equal to a carefully chosen perturbation of $h_n \circ \rho$.
The $1$-periodic orbits of Hamiltonians $H_n$ then fall into two groups (1) $SH$-type: contained in $K_\s$ and (2) $D$-type: outside of $K_\s$. We also make sure that the $SH$-type orbits that are not ``Reeb type'' are constant. 

We now consider the Floer $1$-ray  
$$
\mathcal{C}(H_\tau,J_\tau):= CF^*(M, H_1; \Nov)\to CF^*(M, H_2; \Nov)\to\ldots 
$$
associated to our choice of acceleration data. 
We decompose the associated telescope complex as a direct sum of the $SH$-type generators and the $D$-type generators:
$$tel(\cC) = tel(\cC)_{SH} \oplus tel(\cC)_D.$$
This is a direct sum as $\Nov$-modules, not as cochain complexes: the differential, which we denote by $\partial$, mixes up the factors. 

By restricting the acceleration data to $K_\s$, we also obtain a Floer $1$-ray of $\Bbbk$-cochain complexes
$$
\mathcal{C}_{SH}(H_\tau,J_\tau):= CF^*(K_\s, H_1|_{K_\s}; \Bbbk)\to CF^*(K_\s, H_2|_{K_\s}; \Bbbk)\to\ldots 
$$
and we set
$$SC^*(X;\Bbbk) := tel(\cC_{SH}).$$ 
We denote the differential by $d$. Strictly speaking, this is the cochain complex defining the symplectic cohomology of the Liouville domain $K_\s$ \`a la Viterbo \cite{Vit}. Our notation is justified by the fact that in \cite[Section 4]{McL12T}, McLean shows that $H^*(SC^*(X;\Bbbk))$ only depends on $X$.

We associate a canonical fractional cap $u_{in}$ to each $SH$-type orbit $\gamma$, by setting $u_{in} := u - u \cdot \bflambda$ for an arbitrary cap $u$ (one easily checks that $u_{in}$ is independent of $u$). 
There is then an isomorphism of $\Nov$-modules (recall Equation \eqref{eqdef-SCNov})
\begin{align} 
\label{eq:SCtoCF} SC^*_\Nov & \xrightarrow{\sim} tel(\cC)_{SH} \\
\nonumber \gamma \otimes q^a & \mapsto q^a \cdot (\gamma,u_{in}).
\end{align}
However this is not a chain map: indeed, the matrix component $\partial_{SH,SH}$ need not even be a differential.

\begin{prop}[= Proposition \ref{prop-pos-inta}]\label{prop-pos-int}
For any Floer solution $u$ that contributes to $\mathcal{C}(H_\tau,J_\tau)$ with both ends asymptotic to $SH$-type orbits, we have $u\cdot\bflambda\geq 0.$ In case of equality, $u$ is contained in $K_\s$.
\end{prop}

One could think of Proposition \ref{prop-pos-int} as a manifestation of positivity of intersection of Floer trajectories with the components of the divisor $D$ (c.f. \cite[Lemma 4.2]{Tonkonog}), although we actually prove it using an argument related to Abouzaid--Seidel's `integrated maximum principle' \cite[Lemma 7.2]{Abouzaid2007}.

The consequence of Proposition \ref{prop-pos-int} is that $d \otimes id_\Nov - \partial_{SH,SH}$ strictly increases the $\cQ$-filtration. 
Using PSS isomorphisms, we also see that the homology of $tel(\mathcal{C})$ is isomorphic to $QH^*(M;\Nov).$ 
Thus we are some way towards proving Theorem \ref{thm:aa}, but we are troubled by the existence of $D$-type orbits.
The following proposition is the most important ingredient in the proof of Theorem \ref{thm:aa}, as it allows us to `throw out' the $D$-type orbits.

\begin{prop}\label{prop:imix-div}
There exists $\delta>0$ such that
$$i_{mix}({\gamma})\geq \kappa^{-1}\delta \ell_n$$
for any $D$-type orbit $\gamma$ of $H_n$. 
\end{prop}
\begin{proof}[Sketch of proof when $D$ is smooth]
The Hamiltonian $H_n$ is approximately equal to $\ell_n\left(\rho - \s\right)$ near $D$. 
When $D$ is smooth we have $\rho = r/\kappa\lambda$, where $r$ is the moment map for a Hamiltonian circle action rotating a neighbourhood of $D$ about $D$ with unit speed. 
In particular, the Hamiltonian flow of $H_n$ approximately rotates around $D$ at speed $\ell_n/\kappa \lambda$, and the $D$-type orbits are approximately constant. (This is in contrast to the Hamiltonians used, for example, in \cite{Tonkonog}, which are approximately constant near $D$, and which have non-constant $D$-type orbits linking $D$.)  

We compute the mixed index with respect to the approximately constant cap, which is called $u_{out}$ in the body of the paper. 
As the Hamiltonian flow of $H_n$ rotates around $D$ at speed $\ell_n/\kappa \lambda$, we have $i(\gamma,u_{out}) \approx 2\ell_n/\kappa \lambda$.  
On the other hand we have $H_n \approx h_n(1) \approx \ell_n(1-\s)$ along $D$, and $\omega(u_{out}) \approx 0$, so $\cA(\gamma,u_{out}) \approx \ell_n(1-\s)$. 
Combining we have
\begin{align*}
i_{mix}(\gamma) &= i(\gamma,u_{out}) - \kappa^{-1} \cA(\gamma,u_{out}) \\
&\approx \frac{2\ell_n}{\kappa \lambda} - \kappa^{-1}\ell_n(1-\s) \\
&\ge \kappa^{-1}\ell_n (\s - \s_{crit}),
\end{align*}
which gives the desired result, as we chose $\s > \s_{crit}$.
\end{proof}

Our first thought, in trying to `throw out' the $D$-type orbits, might be to consider the submodule of $tel(\cC)$ spanned by orbits satisfying $i_{mix}(\gamma) < \kappa^{-1} \delta \ell_n$, as that is contained in $tel(\cC)_{SH}$ by Proposition \ref{prop:imix-div}. 
However this does not behave well with respect to the differential: it is neither subcomplex, quotient complex, nor subquotient. 
Instead, we consider a family of subquotient complexes $(SC^{(p)}_\Nov,\partial_p)$ of $tel(\cC)$, indexed by $p \in \R$, spanned by generators $(\gamma,u)$ satisfying 
$$ i(\gamma,u) < p \le \frac{\cA(\gamma,u) + \delta\ell_n}{\kappa}.$$
(Note that these are contained in $tel(\cC)_{SH}$ by Proposition \ref{prop:imix-div}, which is identified with $SC_\Nov$ by \eqref{eq:SCtoCF}.)

To see that this is a subquotient of $tel(\cC)$, we first observe that the differential clearly increases the quantity $\cF(\gamma,u) = \frac{\cA(\gamma,u) + \delta\ell_n}{\kappa}$: it increases action, and increases $n$ and hence $\ell_n$ by the definition of the telescope complex. 
Therefore it defines a filtration map, so $\cF_{\ge p} tel(\cC)$ is a subcomplex. 
On the other hand, the degree truncation $\sigma_{<p} C^\bullet := \bigoplus_{i<p} C^i$ is always a quotient complex of any cochain complex. 
Thus $(SC^{(p)}_\Nov,d \otimes id_\Nov) = \sigma_{<p} \cF_{\ge p} SC_\Nov$ is a subquotient of $SC_\Nov$, whose generators are all of $SH$-type by Proposition \ref{prop:imix-div}.

\begin{prop}\label{prop-qiso}
For any $p\in \mathbb{R}$, both $\mathcal{F}_{\geq p}tel(\mathcal{C})\subset tel(\mathcal{C})$ and $\cF_{\ge p} tel(\cC_{SH}) \subset tel(\cC_{SH})$ are quasi-isomorphic subcomplexes.
\end{prop}
\begin{proof}[Sketch of proof]
We may identify $\cF_{\ge p} tel(\cC)$ as the telescope complex of the 1-ray of Floer groups $\cA_{\ge \kappa p - \delta\ell_n} CF^*(M,H_n,\Nov)$. 
The key point is that $\kappa p - \delta\ell_n \to - \infty$ as $n \to \infty$, and the action filtration is exhaustive, so the direct limit `eventually catches everything' (see Appendix \ref{ss-cofinal}). 
The argument for $\cF_{\ge p} tel(\cC_{SH}) \subset tel(\cC_{SH})$ is identical.
\end{proof}

Because $H^j(\sigma_{<p}C^\bullet) = H^j(C^\bullet)$ for $j<p-1$, we have
$$ H^j(\sigma_{<p}\cF_{\ge p} tel(\cC),\partial) = H^j(M;\Nov) \qquad \text{for $j < p-1$.}$$
If we were willing to weaken the statement in Theorem \ref{thm:aa}, and only achieve the isomorphism of item \eqref{it:3a} up to degree $p-1$, we would now be done: we could simply take $\wt{SC}_\Nov = SC^{(p)}_\Nov$, with $\wt{\cQ}$ equal to the filtration induced by $\cQ$. 
However, to get the corresponding statement in all degrees, we observe that there are natural maps $SC^{(p)}_\Nov \to SC^{(q)}_\Nov$ for all $p \ge q$, induced by the inclusion $\cF_{\ge p} \subset \cF_{\ge q}$ and the projection $\sigma_{<p} \twoheadrightarrow \sigma_{<q}$. 
We define $(\widetilde{SC}_\Nov,\partial)$ to be the homotopy inverse limit of the inverse system of chain complexes $(SC^{(p)}_\Nov,\partial_p)$, and $\wt{\cQ}$ the filtration induced by the $\cQ$-filtration on $SC_\Nov$. 
The result is that
$$H^*(\widetilde{SC}_\Nov,\partial) = \varprojlim_p H^*(SC^{(p)}_\Nov,\partial_p) = QH^*(M;\Nov)$$
as desired. (We remark that this step requires us to check that $\varprojlim^1 H^*(SC^{(p)}_\Nov,\partial_p) = 0$; indeed the inverse system is easily seen to satisfy the Mittag-Leffler property.)
This completes the sketch proof of Theorem \ref{thm:aa}.

\subsubsection{Theorem \ref{thm:specseqa}}\label{sss-C}

In order to prove Theorem \ref{thm:specseqa}, it suffices to prove that the $\wt{\cQ}$-filtration is bounded below and exhaustive, by the `Classical Convergence Theorem' \cite[Theorem 5.5.1]{Weibel}. 
The $\cQ$-filtration on each $SC^{(p)}_\Nov$ is exhaustive by definition, but the $\wt{\cQ}$-filtration on $\wt{SC}_\Nov$ is not exhaustive, due to the direct product taken in the construction. 
Nevertheless one can show that the inclusion $\cup_q \wt{\cQ}_{\ge q} \wt{SC}_\Nov \subset \wt{SC}_\Nov$ is a quasi-isomorphism, and the $\wt{\cQ}$-filtration on this quasi-isomorphic subcomplex is exhaustive by construction. 

Thus the main thing to prove, in order to apply the Classical Convergence Theorem, is that the $\wt{\cQ}$-filtration is bounded below. 
The key ingredient is the following:

\begin{prop}\label{prop:imix-pos}
Suppose that Hypothesis \ref{hyp:gr} is satisfied. 
Then for any $SH$-type orbit $\gamma$, we have $i(\gamma,u_{in}) \ge 0$.
\end{prop}
\begin{proof}[Sketch of proof when $D$ is smooth]
Note that the result is trivial for constant $SH$-type orbits, as $i(\gamma,u_{in})$ is equal to a Morse index which is non-negative. 
For a Reeb-type orbit $\gamma$, we define $u_{out}$ be the small cap passing through $D$. 
Then the orbit $\gamma$ winds $\nu = u_{out} \cdot D$ times around $D$, so $i(\gamma,u_{out}) \approx 2\nu$. 
Thus we have 
$$i(\gamma,u_{in}) = i(\gamma,u_{out}) - \lambda u_{out} \cdot D = (2-\lambda) \nu \ge 0,$$
as required.
\end{proof}

We now show that the $\wt{\cQ}$-filtration is bounded below. 
To be precise, we need to show that for any $i$ there exists $q(i)$ such that $\wt{\cQ}_{\ge q(i)} \widetilde{SC}^i_\Nov = 0$.\footnote{The terminology is counterintuitive as our filtrations are decreasing, whereas the standard conventions for spectral sequences are for the filtrations to be increasing.} 
Indeed, we observe that for $i(\gamma \otimes e^a) = i$ fixed, we have
\[ a_0 \wt{\cQ}(\gamma \otimes e^a) = a = i(\gamma \otimes e^a) - i(\gamma,u_{in}) \le i\]
by Proposition \ref{prop:imix-pos}; thus we may take $q(i) = i/a_0$.

The following result is an immediate consequence of Theorem \ref{thm-superheavy} and the Mayer--Vietoris property of relative symplectic cohomology \cite{Varolgunes2018}.
However it also admits a simple direct proof using Proposition \ref{prop:imix-pos}, which we feel is illuminating, so we give it here.

\begin{prop}\label{prop-heavy}
Suppose Hypothesis \ref{hyp:gr} is satisfied. 
Then the restriction map $$SH_M(M;\Nov)\to SH_M(K_\s;\Nov) $$ is an isomorphism for all $\s\in (0,1)$. In particular, $K_\s\subset M$ is $SH$-visible for all $\s\in (0,1)$ and $\mathbb{L}$ is weakly $SH$-visible, hence not stably displaceable from itself.
\end{prop}
\begin{proof}
Note that we have $i(\gamma,u_{in}) \ge 0$ for any $SH$-type orbit, by Proposition \ref{prop:imix-pos}. 
We also have $\cA(\gamma,u_{in}) = h(\rho) - \rho \cdot h'(\rho) \le 0$, where $\rho = \rho(\gamma)$, by the well-known formula \cite[Section 1.2]{Vit}.\footnote{Note that our conventions are different from Viterbo's.} 
It follows that $i_{mix}(\gamma) \ge 0$. This inequality is satisfied for $D$-type orbits as well (recall Proposition \ref{prop:imix-div}), and therefore it is satisfied for all relevant one periodic orbits.

Now if we fix the index $i(\gamma,u) = i$, then the inequality $i_{mix}(\gamma) \ge 0$ yields an upper bound on the action: $\cA(\gamma,u) \le \kappa \cdot i$. 
Therefore the degreewise completion of the telescope complex has no effect: $$\widehat{tel}(\cC(H_\tau,J_\tau)) = tel(\cC(H_\tau,J_\tau)).$$ 
It follows that $SH^*_M(M;\Nov) \to SH^*_M(K_\s;\Nov)$ is an isomorphism as required.
\end{proof}

\subsubsection{Theorem \ref{thm-superheavy}}\label{sss-D}

In order to prove Theorem \ref{thm-superheavy}, we need to consider the dependence of our constructions on the `smoothing parameter' $R>0$, so we include it in the notation. 
The proof starts with the same strategy that was used in the proof of \cite[Theorem 1.24]{Tonkonog2020}. For $R$ sufficiently small and $\s$ sufficiently close to $1$, $M \setminus K^R_\s$ is stably displaceable (this follows from an $h$-principle as popularized by McLean in \cite{McLean2020}). Therefore, $SH_M\left(\overline{M\setminus K^R_\s};\Nov\right)=0$ for such $R,\s$. We then prove that there exists a continuous function $\s_{crit}^D(R)$, with $\s_{crit}^D(0) = \s_{crit}$, such that the following holds:

\begin{prop}[Proposition \ref{prop:index_bd_collar_invt}]
	\label{prop:index_bd_collar_invti}
	Let $\s_{crit}(R)<\s_1<\s_2<1$. Then, there exists an isomorphism
	$$
	SH^*_M\left(\overline{M\setminus K^R_{\s_1}};\Nov\right)\cong SH^*_M\left(\overline{M\setminus K^R_{\s_2}};\Nov\right).
	$$
\end{prop}

In particular, $SH_M\left(\overline{M \setminus K^R_\s}\right) = 0$ for all  $\s \in( \s_{crit}^D(R),1)$; as the compact sets $\left\{\overline{M  \setminus K^R_\s}\right\}_{R>0,\s>\s_{crit}^D(R)}$ exhaust $M \setminus K_{crit}$, this implies that $K_{crit}$ is SH-full. 

The proof of Proposition \ref{prop:index_bd_collar_invti} uses the `contact Fukaya trick' of \cite{Tonkonog2020}. 
This allows us to set up acceleration data $(H_\tau,J_\tau)$ for $\overline{M\setminus K_{\s_2}}$ and $(\tilde{H}_\tau,\tilde{J}_\tau)$ for $\overline{M\setminus K_{\s_1}}$, so that there is an isomorphism of Floer 1-rays $\cC(H_\tau,J_\tau) \cong \cC(\tilde{H}_\tau,\tilde{J}_\tau)$, which however need not respect action filtrations. 
The key to proving the Proposition, then, is to show that the action filtrations on the corresponding telescope complexes are topologically equivalent. 
The reason why this last step worked in \cite{Tonkonog2020} was the index-boundedness property (also popularized in \cite{McLean2020}). 
In our setting we need estimates on the mixed index, which have a different nature.

\subsection{Conjectures}\label{s-conj}

\subsubsection{Filtration on $QH^*(M;\Nov)$}

Note that, as an immediate corollary of Theorem \ref{thm:aa} \eqref{it:3a}, there exists a filtration $\wt{\cQ}_{\ge \bullet}$ on $QH^*(M;\Nov)$ induced by the $\wt{\cQ}$-filtration on $\left(\wt{SC}_\Nov,\partial\right)$. 
(In general this is different from the `obvious' filtration on $QH^*(M;\Nov)$, i.e., the one with filtration map $\alpha \otimes r \mapsto \cQ(r)$ for $\alpha \in H^*(M;\Bbbk) $, $r \in \Nov$.) 
We give a conjectural description of this $\wt{\cQ}$-filtration. 
Consider the function $f: M \to \R$ defined by $$f(x) = \sum_{k: x \in D_k} \lambda_k-2,$$ 
and set $M^{j} := \{f < j\}$.

\begin{conj}\label{conj:filt}
We have
$$ \wt{\cQ}_{\ge j} H^i(M;\Nov) \supset \ker (H^i(M;\Nov) \to H^i(M^{ja_0-i};\Nov)).$$
When Hypothesis \ref{hyp:gr} holds, this inclusion is an equality.
\end{conj}

We first observe that the Conjecture is consistent with the fact that
$$ q \cdot \wt{\cQ}_{\ge j}H^i(M;\Nov) = \wt{\cQ}_{\ge j+1} H^{i+a_0}(M;\Nov).$$
It is motivated by this together with the natural expectation that the isomorphism of Theorem \ref{thm:aa} \eqref{it:3a} sends
$$\mathrm{PD}(C) \mapsto \left[e^{\sum_{i \in I} \lambda_i} \cdot \mathrm{PSS}_{log}(C) + \text{(higher-order terms)}\right],$$
where $C$ is a cycle contained in $D_I$, and $\mathrm{PSS}_{log}$ is the log PSS map of \cite{Ganatra2020}. 
Thus we expect $\wt{\cQ}(\mathrm{PD}(C)) \ge \sum_{i \in I} \lambda_i/a_0$. 

\begin{rmk}
The filtration in Conjecture \ref{conj:filt} exhibits intriguing parallels with the weight filtration in Hodge theory, c.f. \cite{ElZein2002,Harder2019}.
\end{rmk}

\subsubsection{Analogue of Theorem \ref{thm:specseqa} in the absence of Hypothesis \ref{hyp:gr}}\label{sss-nohyp}

Let us consider the spectral sequence associated to the filtered complex $(\wt{SC}_\Nov,\partial, \wt{\cQ}_{\ge \bullet})$ of Theorem \ref{thm:aa}. 
If Hypothesis \ref{hyp:gr} holds, then it converges to $QH^*(M;\Nov)$ by Theorem \ref{thm:specseqa}; but it is also interesting to study the spectral sequence when this Hypothesis does not hold. 

As we saw  in Section \ref{sss-C}, the reason Hypothesis \ref{hyp:gr} is necessary for Theorem \ref{thm:specseqa} to hold is that it guarantees the $\wt{\cQ}$-filtration on $\wt{SC}_\Nov$ is bounded below, and in particular complete. 
Let us denote by $(\overline{SC}_\Nov,\partial)$ the completion of $(\wt{SC}_\Nov,\partial)$ with respect to the $\wt{\cQ}$-filtration. 
Note that taking the completion does not change the spectral sequence.

We give a conjectural description of $H^*(\overline{SC}_\Nov,\partial)$, based on suggestions made to us independently by Pomerleano and Seidel. 
For each $i \in I$, define $QH^*(M;\Nov)_i$ to be the $0$-generalized eigenspace of the operator $\mathrm{PD}(D_i)\star(-)$ on $QH^*(M;\Nov)$, where $\star$ denotes the quantum cup product. 
I.e., it is the subspace of $\alpha \in QH^*(M;\Nov)$ such that $\mathrm{PD}(D_i)^{\star k} \star \alpha = 0$ for some $k$.
We then define
$$QH^*(M;\Nov)_{crit} := \bigcap_{i: \lambda_i>2} QH^*(M;\Nov)_i.$$

\begin{conj}\label{conj:ss-nohyp}
We have $H^*(\overline{SC}_\Nov,\partial) \cong QH^*(M;\Nov)_{crit}.$
Furthermore, the resulting spectral sequence converges to $QH^*(M;\Nov)_{crit}$.
\end{conj}

As evidence for the conjecture, we use Conjecture \ref{conj:filt} to argue that whenever $\lambda_i>2$, the degree-$0$ class $c - e^{-2}\mathrm{PD}(D_i)$ is invertible in the $\wt{\cQ}$-completed quantum cohomology, for any $c \neq 0$. 
Indeed its inverse is
$$\left(c - e^{-2}\mathrm{PD}(D_i)\right)^{-1} = c^{-1} \cdot \sum_{j =0}^\infty \left(c^{-1} e^{-2} \mathrm{PD}(D_i) \right)^{\star j},$$
which converges because $\wt{\cQ}(e^{-2} \mathrm{PD}(D_i)) \ge (\lambda_i - 2)/a_0>0$.
Therefore, any $c$-generalized eigenvector of $e^{-2}\mathrm{PD}(D_i) \star (-)$ dies in the $\wt{\cQ}$-completion:
$$ \left(c - e^{-2}\mathrm{PD}(D_i)\right)^{\star k} \star \alpha = 0 \qquad \Rightarrow \qquad \alpha = 0,$$
by multiplying on the left by the inverse. 

Assuming that the $\Bbbk$-linear endomorphisms $e^{-2}\mathrm{PD}(D_i) \star (-)$ admit Jordan normal forms, the above argument suggests that only the $0$-generalized eigenspaces can `survive'. 
This gives some evidence for Conjecture \ref{conj:ss-nohyp} in the case that $\Bbbk$ is an algebraically closed field. 
It is reasonable to believe that one can bootstrap from there to the case of a general commutative ring $\Bbbk$. For the rest of this section we will assume that $\Bbbk$ is an algebraically closed field.

\begin{remark} We strongly expect that $H^*(\overline{SC}_\Nov,\partial)$ is nothing but the relative symplectic cohomology of the skeleton of $X$. There is an intriguing contrast between Conjecture \ref{conj:ss-nohyp} and Ritter's work \cite{Ritter2014}: precisely, let us consider the case that $D$ is smooth and $\lambda>2$, and let $\cN$ be the total space of the inverse of the normal bundle to $D$. Then Conjecture \ref{conj:ss-nohyp} (together with the above expectation) says that $QH^*(M)_{crit}$, which is the $0$-generalized eigenspace of $QH^*(M)$, `lives on the skeleton of $X$'; whereas Ritter shows that $SH^*(\cN)$ is the quotient of $QH^*(\cN)$ by its $0$-generalized eigenspace. 
Note that we can obtain $\cN$ from the Liouville completion $\widehat{X}$ of $X$ by replacing a neighbourhood of the skeleton with a copy of $D$ (more precisely, the symplectic cut of $\widehat{X}$ along the hypersurface $\{\rho = 1\}$ is $M \coprod \cN$).
\end{remark}

\begin{remark}
In light of Venkatesh's quantitative generalization of Ritter's results \cite{Venkatesh2021}, we expect that considering Liouville domain neighborhoods $V$ of the skeleton of varying sizes (vaguely speaking, `in the directions of the components of the divisors'), one might observe that additional simultaneous generalized eigenspaces start contributing to $SH_M^*(V;\Nov)$. It might be possible to interpret Theorem \ref{thm-superheavy} as the other end of this size dependence: if the size of $V$ is large enough in all directions (e.g., if it contains $K_{crit}$), then all simultaneous generalized eigenspaces contribute to $SH_M^*(V;\Nov).$
\end{remark}

Further evidence for Conjecture \ref{conj:ss-nohyp} is provided in \cite{EliashbergPolterovich}, in the case $M = \CP^1 \times \CP^1$, where $D$ is a $(1,1)$ hypersurface: indeed the conjecture is confirmed in this case. 
We discuss further examples in Sections \ref{sss-quadric} and \ref{sss-fanohyp} below.

We now recall a variation on the definition of relative symplectic cohomology from \cite[Remark 1.8]{Tonkonog2020}. 
The relative symplectic cohomology $SH^*_M(K;\Nov)$ is a module over $SH^*_M(M;\Nov) = QH^*(M;\Nov)$, via the restriction map. 
For any idempotent $a \in QH^0(M;\Nov)$, we define the `$a$-relative symplectic cohomology of $K$' to be $a \cdot SH^*_M(K;\Nov)$. 
We define corresponding properties of subsets of $M$: $a$-$SH$-visible, $a$-$SH$-full, etc.

\begin{lem}\label{lem:idemp}
The subspace $QH^*(M;\Nov)_{crit} \subset QH^*(M;\Nov)$ is an ideal which is generated by an idempotent $a$.
\end{lem}
\begin{proof}
We first observe that for any even element $\alpha$ in a supercommutative Frobenius algebra, the decomposition into generalized eigenspaces of $\alpha \star (-)$ is orthogonal (with respect to the pairing and the algebra structure), and hence the generalized eigenspaces are ideals generated by idempotents.
It follows for each $i$, the subspace $QH^*(M;\Nov)_i$ is an ideal generated by an idempotent; so the intersection is an ideal generated by the product of these idempotents.
\end{proof}

\begin{conj}\label{conj:idemp}
Under the same hypotheses as for Theorem \ref{thm-superheavy}  (without assuming Hypothesis \ref{hyp:gr}), the skeleton $\mathbb{L}$ is $a$-$SH$-full, where $a$ is the idempotent from Lemma \ref{lem:idemp}. 
\end{conj}

Conjecture \ref{conj:idemp} implies, for example, that $\mathbb{L}$ must intersect every $a$-Floer-theoretically essential (over $\Bbbk$) monotone Lagrangian, where the latter condition means that $CO(a \otimes_\Nov \Bbbk) \in HF^0(L;\Bbbk)$ is non-zero. (Here we have used the algebra homomorphism $\Nov \to \Bbbk$, which sends $q \mapsto 1$, to define an idempotent $a \otimes_\Nov \Bbbk \in QH^0(M;\Bbbk)$).

\subsubsection{Maurer--Cartan element}\label{sss-MC}

For the purpose of this section, we assume that $\Bbbk$ is a field of characteristic zero, and we assume that Hypothesis \ref{hyp:gr} holds.
 
Recall that the symplectic cochain complex $SC^*(X;\Bbbk)$ carries an $L_\infty$ structure \cite{FabertSalchow}. 
This consists of a sequence of operations $\ell^k: SC^*(X;\Bbbk)^{\otimes k} \to SC^*(X;\Bbbk)$ of degree $3-2k$, satisfying the $L_\infty$
 relations; and $\ell^1 = d$ is the standard differential. 
We extend these linearly to make $SC^*_\Nov$ into an $L_\infty$ algebra. 
We recall that a \emph{Maurer--Cartan element} for the $L_\infty$ algebra $(SC^*_\Nov,\ell^k)$ is an element $\beta \in \cQ_{\ge 1}SC^2_\Nov$, satisfying the Maurer--Cartan equation:
$$ \sum_k \frac{\ell^k(\beta,\ldots,\beta)}{k!} = 0.$$
We remark that this is in fact a finite sum, because the terms live in successively higher levels of the $\cQ$-filtration, which Hypothesis \ref{hyp:gr} ensures is bounded below (see Section \ref{sss-D}).

A Maurer--Cartan element $\beta$ can be used to deform the $L_\infty$ structure to get a new one $\ell^k_{\beta}$ on $SC_\Nov$ (see e.g. \cite[Section 4]{Getzler2009}). In particular, the resulting operation $\ell^1_\beta$ defines a new differential on $SC_\Nov$.

\begin{conj}\label{conj:MC}
There exists a Maurer--Cartan element $\beta \in SC^2_\Nov$ such that in the statement of Theorem \ref{thm:aa}, we may take $\widetilde{SC}_\Nov = SC_\Nov$ and $\partial  = \ell^1_\beta$. 
\end{conj}

\begin{rmk} Cieliebak and Latschev have outlined ideas closely related to Conjecture \ref{conj:MC} (but in a more general context) in talks as far back as 2014.
\end{rmk}

\begin{rmk}
Moreover, one expects that Floer-theoretic operations on quantum cohomology of $M$ (such as the quantum cup product) are deformations of the corresponding operations on symplectic cohomology of $X$ by $\beta$, c.f. \cite{Fabert}.
\end{rmk}

\begin{rmk}
In the proof of Theorem \ref{thm:aa} presented in this paper, we need to replace $SC_\Nov$ with $\wt{SC}_\Nov$. 
Conjecture \ref{conj:MC} suggests an alternative proof, in which no such replacement is necessary. 
The cost is that the construction is significantly more elaborate, relying on the $L_\infty$ structure and a version of the homotopy transfer theorem, which makes it harder to see the key geometric ideas, which are the same in both proofs.
\end{rmk}

\begin{rmk}\label{rmk:partialcomp}
It is natural to envision generalizations of our results, as well as of Conjecture \ref{conj:MC}, where $M$ is allowed to be only a partial compactification of $X$; and furthermore, where some of the weights $\lambda_i$ are allowed to be equal to $0$. 
We present several examples in Section \ref{s-eg} below which illustrate such a generalization. 
For example, Remark \ref{rmk:RP2} gives evidence for this generalized conjecture in the case $M=T^*\mathbb{RP}^2$, with $D \subset M$ a smooth divisor equipped with weight $\lambda=0$; the generalized conjecture in this case says that $SC^*(M;\Bbbk)$ is a `deformation' of $SC^*(X;\Bbbk)$ (note that there is no need for a Novikov ring in the definition of symplectic cohomology of $M$, as it is exact). 
We put scare quotes around `deformation' because when the weights are $0$, the extra terms in the deformed differential may simply preserve the $\cQ$-filtration, rather than strictly increasing it; so there is no sense in which they are `small'.
To make a useful version of the conjecture one would need an alternative to the $\cQ$-filtration, which \emph{is} strictly increased by the extra terms; it would probably be defined in terms of the grading.
\end{rmk}
  
Note that the projection of $\beta$ to $\mathrm{Gr}_1 SC^2_\Nov$ is $d$-closed, and hence defines a class $[\beta_1] \in \mathrm{Gr}_1 SH^2(X;\Nov)$. 
It is immediate from Conjecture \ref{conj:MC} that the differential on the $E_1$ page of the spectral sequence is given by $[[\beta_1],-]$, where $[-,-]$ denotes the Lie bracket on $SH^*(X;\Bbbk)$.

We now explain how our conjectures connect with work of Tonkonog \cite{Tonkonog}. 
Tonkonog considers the following setup: $\bar{M}$ is a compact Fano variety equipped with its monotone K\"ahler form, $\bar{D} \subset \bar{M}$ a simple normal crossings anticanonical divisor, $X=\bar{M} \setminus \bar{D}$, and $M = \bar{M} \setminus \cup_{i=1}^J \bar{D}_i$ is a partial compactification of $X$, with compactifying divisor $D=M \cap \bar{D}$. 
Tonkonog defines a class $\mathcal{BS} \in SH^0(X;\Bbbk)$ by counting pseudoholomorphic `caps' in $M$, such that the following holds:

\begin{thm}[Theorem 6.5 in \cite{Tonkonog}]\label{thm:Tonk}
For any exact closed Lagrangian $L \subset X$ equipped with a $\Bbbk^*$-local system $\xi$, we have $\mathcal{CO}(\mathcal{BS}) = \mathfrak{m}^0_{(L,\xi)}$, where $\mathcal{CO}:SH^*(X;\Nov) \to H^*(L;\Nov)$ is the closed--open map, and $\mathfrak{m}^0_{(L,\xi)} \in H^2(L;\Nov)$ is the disc potential.
\end{thm}

This fits into the generalized geometric setup alluded to in Remark \ref{rmk:partialcomp} (we are in the log Calabi--Yau setting, and we equip each component of $D$ with its canonical weight $2$). 
It connects with our conjectures as follows:

\begin{conj}
We have $\mathcal{BS} = [\beta_1]$.
\end{conj}

In many settings, we can tightly constrain the class $\beta$ using grading considerations. 
For each $i$ we can define a cocycle $B_i \in SC^{2-\lambda_j}(X;\Bbbk)$ by `counting caps passing through $D_i$', following \cite{Tonkonog} or \cite{LogPSS}. 
We define 
$$ B:= \sum_i e^{\lambda_i} \cdot B_i \in SC^2(X;\Nov).$$

\begin{conj}\label{conj:split}
Suppose we are in the log Calabi--Yau case: i.e., $\lambda_i = 2$ for all $i$, and furthermore that the minimal Chern number of $M$ is $\ge 2$. 
Then we have $\beta = B$. 
\end{conj}

\begin{rmk}
If the minimal Chern number of $M$ is $1$, then we conjecture that $\beta = B +e^2 \cdot B_0$, where $B_0 \in SC^0(X;\Bbbk)$ is a multiple of the unit, and counts certain holomorphic spheres in $M$ of Chern number $1$. 
Note that the additional term $B_0$ is irrelevant for the purposes of Conjecture \ref{conj:MC}, as $\ell^1_B = \ell^1_{B+B_0}$ using the fact that $B_0$ is a multiple of the unit.
\end{rmk}

As evidence for the Conjecture, we first observe that $\mathrm{Gr}_1^\cQ SC^2_\Nov$ is generated by the classes $qB_i$, together with the unit $q \cdot 1$; and argue that the coefficient of the unit in $\beta$ must count certain Chern-number-1 spheres. 
We further observe that $\cQ_{\ge 2} SC^2_\Nov = 0$. 
This follows as we have $a_0 = 2$, so any generator $\gamma \otimes q^j$ of $SC^2_\Nov$ with $j \ge 2$ must have $i(\gamma) \le -2$; however $i(\gamma) \ge 0$ by Proposition \ref{prop:imix-pos}. 

\begin{rmk}
Based on \cite[Lemma 6.4]{SheridanVersality}, we also expect Conjecture \ref{conj:split} to hold under either of the following hypotheses:
\begin{itemize}
\item $D$ is smooth and Hypothesis \ref{hyp:gr} is satisfied.
\item $M$ is a projective variety, $D$ a complex divisor, and $c_1(TM)$ lies in the interior of the cone $Amp'(M,D) \subset H^2(M;\R)$ defined in \cite[Definition 3.26]{SheridanVersality}.
\end{itemize}
\end{rmk}

In settings where Conjecture \ref{conj:split} holds, the Maurer--Cartan element $\beta$ is determined up to gauge equivalence by the cohomology classes $[B_i]$. 
Furthermore, the components of $\beta$ get `turned on' one by one as the corresponding divisors get added compactifying $X$.

\color{black}

\subsubsection{Mirror symmetry in the log Calabi--Yau case}\label{sss-MS}

Let us consider the log Calabi--Yau case, where $X=M\setminus D$ and $X$ is equipped with its preferred Liouville structure and trivialization of canonical bundle. 
In this case we have $a_0 = 2$, so $\Nov = \Bbbk[q,q^{-1}]$ where $i(q) = 2$.

Assume that $Y$ is a mirror scheme to $X$ over $\Bbbk$, which is smooth. 
Even though we choose to leave what this means vague, we will assume that it implies
\begin{equation}
\label{eq:SHmirr}
SH^i(X;\Bbbk) \simeq \bigoplus_{p+q = i} H^q(Y,\Lambda^p TY),
\end{equation}
and in particular
$$SH^0(X;\Bbbk)\simeq H^0(Y,\mathcal{O}_Y).$$ 
Therefore, the classes $B_i \in SH^0(X;\Bbbk)$ are mirror to functions $w_i \in H^0(Y,\mathcal{O}_Y)$. 
We set $W:= \sum_i w_i$. This sum includes the constant term $w_0$, which may be non-zero in the case that the minimal Chern number of $M$ is $1$.

Now let $Y_\Nov$ denote the base change of $Y$ to $\Nov$, and $W_\Nov = qW$ be a function on $Y_\Nov$.

\begin{conj}\label{conj:MS}
The Landau--Ginzburg model $(Y_\Nov,W_\Nov)$ is mirror to $M$.
\end{conj}

\begin{rmk}
In fact, Conjecture \ref{conj:MS} should extend beyond the log Calabi--Yau case we consider here. 
However, it becomes difficult (and confusing) to interpret the mirror in terms of the language of classical algebraic geometry: the polyvector fields on $Y_\Nov$ are given a non-standard grading, and in general $W_\Nov$ may be a polyvector field rather than a function. 
In contrast, in the log Calabi--Yau case one can give a transparent interpretation of Conjecture \ref{conj:MS} in terms of the classical algebraic geometry of the Landau--Ginzburg model $(Y,W)$ defined over $\Bbbk$, which we now do. 
(We discuss the non-log-Calabi--Yau case in Remark \ref{rmk:MS-gen} at the end of this section.)
\end{rmk}

We consider the Koszul complex associated to the section $dW$ of $T^*Y$:
$$K(dW) := \left\{\ldots \to \Lambda^{p+1}(TY)\xrightarrow{dW} \Lambda^{p}(TY)\to\ldots\to TY \xrightarrow{dW} \mathcal{O}_Y\right\}.$$ 
This is a complex of vector bundles over $Y$. 
When the critical locus $Z:=\mathrm{Crit}(W)$ is isolated, $K(dW)$ is a resolution of $\mathcal{O}_Z$, and therefore its hypercohomology gives the algebra of functions on the critical locus: $\mathbb{H}^*(K(dW)) \cong \cO(Z)$ (the hypercohomology is concentrated in degree $*=0$). 
In general, we define $\cO(Z^h):= \mathbb{H}^*(K(dW))$, because this hypercohomology is, essentially by definition, the graded algebra of functions on the `derived critical locus of $W$' (see e.g. \cite{Vezzosi}). 

Conjecture \ref{conj:MS} implies, among other things, that we have an isomorphism of graded $\Nov$-algebras
\begin{equation}\label{eq:QHmirr}
\cO(Z^h) \otimes_\Bbbk \Nov \cong QH^*(M;\Nov).
 \end{equation}
We expect that the mirror to the spectral sequence of Theorem \ref{thm:specseqa} on the RHS, is the hypercohomology spectral sequence on the LHS, in a sense we now make clear.

We recall the construction of the hypercohomology spectral sequence 
$$^I \! E_1^{p,q} = H^q(\Lambda^{-p} TY) \Rightarrow \cO(Z^h),$$
following \cite[Section 5.7]{Weibel}. 
We take a Cartan--Eilenberg resolution $\cC^{p,q}$ of $K(dW)$, and consider the resulting bicomplex $C^{p,q} = \Gamma(\cC^{p,q})$. 
We define a filtration map on this complex by $\cQ(c) = p$ for $c \in C^{p,q}$ (i.e., we have $\cQ(c)=-p$ for $c$ a section of $\Lambda^p TY$). 
The resulting $\cQ$-filtration induces the spectral sequence with $E_1$ page as above. 
The differential on the $E_1$ page is given by contracting with $dW$.

We now consider the bicomplex $C^{p,q} \otimes_\Bbbk \Nov$, and equip it with the filtration map $\cQ(c \otimes r) = \cQ(c)+\cQ(r)$. 
We conjecture that the resulting filtered complex is filtered quasi-isomorphic to $(\wt{SC}_\Nov,\partial,\wt{\cQ}_{\ge \bullet})$, and in particular the corresponding spectral sequence is isomorphic to the one from Theorem \ref{thm:specseqa}. 
As evidence, we compute that the spectral sequence has
\begin{align*}
E_1^{j,k} &= \bigoplus_{p+q=3j+k} H^q(\Lambda^p TY) \otimes_\Bbbk \Bbbk \cdot q^{-j+p} \\
&\cong SH^{3j+k}(X;\Bbbk) \otimes_\Bbbk \Bbbk \cdot q^{-j+p},
\end{align*}
which is clearly isomorphic to the $E_1$ page of the spectral sequence from Theorem \ref{thm:specseqa}.

\begin{rmk}
The attentive reader may notice the presence of an extra `$p$' in the exponent of $q$, compared with the $E_1$ page from Theorem \ref{thm:specseqa}. 
This is because the isomorphism of $E_1$ pages
\begin{align*}
SH \otimes_\Bbbk \Lambda &= \bigoplus_{q,p}H^q(\Lambda^p TY) \otimes_\Bbbk \Lambda \qquad \text{sends}\\
SH \otimes_\Bbbk \Bbbk & \mapsto \bigoplus_{q,p} H^q(\Lambda^p TY) \otimes_\Bbbk \Bbbk \cdot q^{-p}.
\end{align*}
This reflects the fact that $\cQ(c) =-p$ for $c \in \Lambda^p TY$. 
\end{rmk}

We now explain how this fits with the picture from the previous section. 
The isomorphism \eqref{eq:SHmirr} is expected to respect the natural graded Lie algebra structures on both sides (among other things), where the Lie bracket on the polyvector field cohomology is given by the Schouten--Nijenhuis bracket. 
The differential on the $E_1$ page of the symplectic spectral sequence is given by $[B,-]$. 
The differential on the $E_1$ page of the hypercohomology spectral sequence is given by contraction with $qdW$, which coincides with $[qW,-]$ (as one can see from the definition of the Schouten--Nijenhuis bracket); thus the two differentials match. 

More precisely, we expect that the isomorphism of Lie algebras \eqref{eq:SHmirr} can be refined to a quasi-isomorphism of $L_\infty$ algebras, and the Maurer--Cartan element $\beta$ matches with the Maurer--Cartan element $qW$ up to gauge equivalence. 
This would yield a chain-level quasi-isomorphism underlying \eqref{eq:QHmirr}, which would imply the isomorphism of spectral sequences discussed above.

Note that when $Y$ is affine, there is no need to take a Cartan--Eilenberg resolution: we may take $C^{p,0} = \Gamma(\Lambda^{-p} TY)$ and $C^{p,q} = 0$ for $q \neq 0$, with differential given by contracting with $dW$, and the bicomplex is simply a complex. 
In particular the hypercohomology spectral sequence degenerates at $E_2$.
This leads us to make the following: 

\begin{conj}
If $X$ in addition (to the conditions from the first paragraph of this section) admits a homological Lagrangian section and $SH^0(X;\Bbbk)$ is a smooth algebra, then the spectral sequence of Theorem \ref{thm:specseqa} degenerates at $E_2$ page. 
\end{conj}

Under these assumptions on $X$ one can take $Y$ to be the smooth affine scheme $Spec(SH^0(X;\Bbbk))$ (see \cite{Pomerleano2021}), which would satisfy \eqref{eq:SHmirr}, which is our reason to make this conjecture.

For example, the conjecture holds in the toric Fano examples (see Section \ref{sss-toric-fano}), essentially by the argument given above. This degeneration also follows from the fact that one can construct $SC^*(X;\Bbbk)$ with zero differential in this case!

\begin{rmk}\label{rmk:MS-gen}  
We now discuss the non-log-Calabi--Yau case of Conjecture \ref{conj:MS}, which will appear in several examples in Section \ref{s-eg} below. 
There are three complicating factors: 
\begin{enumerate}
\item The mirror to $X$ will in general be a Landau--Ginzburg model $(Y,w)$, rather than simply a variety $Y$;
\item \label{it:grad}The algebra of polyvector fields on $Y$ must be equipped with a non-standard grading;
\item \label{it:poly} \emph{a priori}, $\beta$ will be mirror to a gauge equivalence class of Maurer--Cartan elements for the differential graded Lie algebra of polyvector fields on $(Y,w)$, rather than simply a function $W$ on $Y$. 
\end{enumerate}

Issue \eqref{it:grad} is already present if one wants to talk about the mirror of $T^*S^1$ with a non-standard trivialization of its canonical bundle and then consider the correspondence between compactifications and deformations. In this case one cannot use a traditional SYZ approach as the zero section of $T^*S^1$ does not even have vanishing Maslov class with respect to such a trivialization. It seems that in order to develop some general geometric intuition in the non-log Calabi-Yau cases, it would be helpful to use the language of derived algebraic geometry but we do not feel comfortable enough to do this at this point. 

Concerning issue \eqref{it:poly}, we actually expect that $\beta$ should be mirror to a function in broad generality, although it is not clear how to prove this. 
In some cases it follows from grading considerations, as in Conjecture \ref{conj:split} and the ensuing remarks. 
\end{rmk}

\begin{rmk}\label{rmk:logfano}
Even though we avoid a general discussion, we do use our expectations in the log Fano case in some examples in Section \ref{s-eg} below. 
Here is our starting ansatz in these examples: start with a log Calabi--Yau pair $(M,D' )$, where $$D' = \bigcup_{i=1}^{N+J} D'_i,\qquad \text{and set} \qquad D = \bigcup_{i=1}^N D'_i.$$ Suppose that $X'=M \setminus D'$ is mirror to $Y$ as at the start of this section. 
This means that we could choose all weights $\lambda'_i = 2$; we assume, however, that there exists a valid choice of weights with $\lambda'_i > 0$ for all $1 \le i \le N$, and $\lambda'_i=0$ for $N+1 \le i \le N+J$. 
We equip $X'$ with the trivialization of its canonical bundle corresponding to these weights, and equip the algebra $SH^0(X';\Bbbk)$ with its induced grading. We posit that this is the graded algebra of functions on the mirror of $X'$ (with the alternative trivialization), which we regard as a `graded scheme'. We set $X = M \setminus D$, and posit that the mirror to $X$ is $(Y,w)$ where $w = \sum_{i=N+1}^{N+J} w_i$. 
We furthermore posit that the Maurer--Cartan element $\beta$ corresponding to $X \subset M$ is mirror to $W_\Nov = \sum_{i=1}^N e^{\lambda_i} w_i$, and therefore that the mirror to $M$ is $(Y_\Nov,w+W_\Nov)$.
\end{rmk}

\color{black}

\subsection{Examples}\label{s-eg}

\subsubsection{Fano toric varieties}\label{sss-toric-fano}

Let $\Delta\subset \mathbb{R}^n$ be a Fano Delzant polytope. This means that it is a Delzant polytope equal to the intersection of half-spaces (with no redundancy) $$\nu_i(x)+2\kappa\geq 0,\text{ }i=1,\ldots,m$$ for $\kappa>0$ and $\nu_i\in(\mathbb{Z}^n)^\vee$ primitive. Using the symplectic boundary reduction construction (one of the many options), we construct a symplectic manifold $(M_\Delta,\omega)$ with a Hamiltonian $T^n$ action and moment map $$\pi: M_\Delta\to  \mathbb{R}^n.$$ The image of the moment map is by construction $\Delta$. Finally, note that $M_\Delta$ satisfies the monotonicity condition $ 2 \kappa c_1(TM_\Delta) =  [\w]$.

We define the toric SC divisor $D_\Delta$ as the preimage of the boundary of $\Delta$ under the moment map. Note that $D_\Delta=\bigcup_{i=1}^m D_i $ is automatically an orthogonal SC divisor. We define $X_\Delta=M_\Delta\setminus D_\Delta.$ Again by construction $X_\Delta$ is a product $int(\Delta)\times (\mathbb{R}^n)^\vee/ (\mathbb{Z}^n)^\vee.$ Denoting the coordinates on $\mathbb{R}^n$ by $x_1,\ldots,x_n$ and the circle valued coordinates on  $(\mathbb{R}^n)^\vee/ (\mathbb{Z}^n)^\vee$ by $\phi_1,\ldots,\phi_n$ we have $$\omega|_X=\sum dx_id\phi_i.$$

We note the short exact sequence $$\xymatrix{
  0 \ar[r] & H^1(X_\Delta;\R)   \ar[r]^-{f} & H^2(M_\Delta,X_\Delta;\R)  \ar[r]^-{g} & H^2(M_\Delta;\R)  \ar[r] & 0. }$$

A choice of weights is (as always) equivalent to the choice of a rational class \[ \bflambda\in H^2(M_\Delta,X_\Delta;\R)\cong \mathbb{R}^m, \] which is sent to $2c_1(TM_\Delta)$ by $g$ and which has positive coordinates. We have a preferred lift given by $$\bflambda^{can}=(2,\ldots,2).$$

Let us also use the natural isomorphism $H^1(X_\Delta;\R)\cong \mathbb{R}^n$. The map $f$ is easily computed to be $$x\mapsto \nu_i(x).$$ Hence, the set of all possible positive weights is the image of the rational points in the interior of $\frac{1}{\kappa}\Delta$ under the map $\mathbb{R}^n\to   \mathbb{R}^m$ given by $$(x_1,\ldots,x_n)\mapsto(\nu_1(x)+2,\ldots,\nu_f(x)+2).$$

We see that the only weight that satisfies Hypothesis \ref{hyp:gr} is the canonical weight, which corresponds to $0\in \frac{1}{\kappa}\Delta$.

Now let us outline how Theorems \ref{thm:aa} and \ref{thm:specseqa} work in this context, assuming the conjectural results of Section \ref{sss-MC}.
We can arrange that
$$SC^*(X_\Delta;\Bbbk) \cong \Bbbk[z_1^{\pm 1},\ldots,z_n^{\pm 1},\partial/\partial z_1,\ldots,\partial/\partial z_n]$$
where the variables $z_i$ are commuting and have degree $0$, and the variables $\partial/\partial z_i$ are anticommuting and have degree $1$ (where the degrees are induced by $\bflambda^{can}$). 
We can also arrange that the $L_\infty$ structure is trivial, with the exception of the Lie bracket $\ell^2$, which coincides with the Schouten--Nijenhuis bracket. 
We can compute, for instance via Theorem \ref{thm:Tonk} and Cho--Oh's computation of the disc potential of toric Fano varieties \cite{ChoOh}, that $\beta = qW$, where 
$$W = \sum_i z^{\nu_i}.$$
Now Conjecture \ref{conj:MC} says that in the statement of Theorem \ref{thm:aa}, we can take
\begin{align*}
\wt{SC}_\Nov = SC_\Nov &= \Nov [z_1^{\pm 1},\ldots,z_n^{\pm 1},\partial/\partial z_1,\ldots,\partial/\partial z_n] ,\qquad \text{with}\\
\partial & = [qW,-].
\end{align*}
As explained in Section \ref{sss-MS}, this is the Koszul complex for $dW$, tensored with $\Nov$. 
One can show that $W$ has isolated singularities, so the cohomology of the Koszul complex is
$$\cO(Z) = \frac{\Bbbk [z_1^{\pm 1},\ldots,z_n^{\pm 1}]}{\left\langle \frac{\partial W}{\partial z_1},\ldots,\frac{\partial W}{\partial z_n} \right\rangle} = Jac(W).$$
Thus, assuming Conjecture \ref{conj:MC}, Theorem \ref{thm:aa} gives
$$QH^*(M_\Delta;\Nov) \cong H^*(SC_\Nov,[qW,-]) \cong Jac(W) \otimes_\Bbbk \Nov,$$
which is the familiar statement of closed-string mirror symmetry for toric Fano varieties, c.f. \cite{Batyrev}. 
Note that the spectral sequence of Theorem \ref{thm:specseqa} has $E_0 = E_1 = SC_\Nov$, $E_2 = Jac(W)\otimes_\Bbbk \Nov$, and degenerates at $E_2$ because the differential on $SC_\Nov$ vanishes (or alternatively, because $Jac(W)$ is concentrated in even degree).

Now let us outline how Theorem \ref{thm-superheavy} works in this context.  
For each $i=1,\ldots,m$, we have a Hamiltonian circle action with moment map $\nu_i \circ \pi$, which rotates around $D_i$, and these actions commute on the overlaps. 
It follows that they define a system of commuting Hamiltonians for $D_\Delta$, in the sense of Section \ref{s-basics-of-SC}. 
For any $p\in int(\Delta)$ we define the corresponding weights $\bflambda_p:=f(\frac{p}{\kappa})+\bflambda^{can}$ and  
primitive (of $\omega|_{X_\Delta}$) $$\theta_p=\sum (x_i - p_i)d\phi_i.$$ The relative de Rham class of $(\omega,\theta_p)$ is easily seen to be $f(p)+\kappa\bflambda^{can}=\kappa\bflambda_p$. The Liouville vector field corresponding to $\theta_p$ is $Z_p = \sum (x_i - p_i) \partial/\partial x_i$. It follows that $\theta_p$ is adapted to the system of commuting Hamiltonians in the sense of Section \ref{s-basics-of-SC}. The skeleton $\mathbb{L}_p$ for $\theta_p$ is nothing but the Lagrangian torus above $p$. The corresponding subset $K_{crit,p}$ is easily computed to be $\pi^{-1}(\tilde{K}_{crit,p})$, where $\tilde{K}_{crit,p} \subset \Delta$ is the smallest rescaling of $\Delta$, centred at $p$, which contains the origin. 
In particular, $K_{crit,p}$ coincides with $\mathbb{L}_p$ if and only if $\bflambda_p = \bflambda^{can}$, if and only if Hypothesis \ref{hyp:gr} is satisfied.

Our Theorem \ref{thm-superheavy} says that the monotone torus fiber $\mathbb{L}_0$ is $SH$-full. 
It follows that it is not stably displaceable. This result can also be obtained using Lagrangian Floer theory, using the fact that the disc potential always has a critical point in this case. Our result says nothing about the skeleta $\mathbb{L}_p$ for $p\neq 0$. Indeed it is known that for $n\leq 3$ all of these non-monotone fibers are displaceable by probes \cite[Corollary 3.9 and Proposition 4.7]{McDuff}.

The fact that $\mathbb{L}_0$ is $SH$-full also implies that it intersects every Floer theoretically essential (over some commutative ring) monotone Lagrangian. This result also follows from the fact that $\mathbb{L}_0$, equipped with appropriate local systems, split-generates each component of the monotone Fukaya category over an arbitrary field \cite[Corollary 1.3.1]{EvansLekili}.

\subsubsection{Skeleta in $S^2$}

Let us move on to a non-toric example. Consider $S^2$ with a symplectic structure $\omega$ such that $[\omega]=4\kappa \mathrm{PD}(\mathrm{pt}).$ Let $D$ be the union of $N$ distinct points $p_1,\ldots ,p_N\in S^2$. Consider weights $\lambda_1,\ldots,\lambda_N>0,$ which needs to satisfy $$\lambda_1+\ldots+\lambda_N=4.$$

Let $\theta$ be a primitive of $\omega$ on $S^2\setminus D$ compatible with the weights and with some choice of local moment maps for the circle actions rotating about the $p_i$. Let $\mathbb{L}$ be the induced skeleton. The complement $S^2-\mathbb{L}$ is a disjoint union of open disks $U_i$, $i=1,\ldots N$, one for each point $p_1,\ldots p_N$. $\mathbb{L}$ itself is the union of all critical points, homoclinical and heteroclinical orbits, and periodic orbits of the Liouville vector field by the Poincar\'e--Bendixson theorem. 
It is elementary to compute (using the compatibility with weights) that the symplectic area of $U_i$ is equal to $\kappa \lambda_i$. 
If we restrict the function $\rho:M \to \R$ to the disc $U_i$, then it extends continuously to $0$ along the boundary of the closed disk, it is equal to $1$ at $p_i$, and it generates a Hamiltonian circle action rotating $U_i$ about $p_i$. 

Hypothesis \ref{hyp:gr} is satisfied if and only if no weight is bigger than $2$, which means no disc $U_i$ has area more than half the area of $S^2$. 
In this case the subset $K_{crit}$ coincides with the skeleton $\mathbb{L}$. 
Otherwise, we have $\lambda_i>2$ for some $i$, and $K_{crit}$ is the union of $\mathbb{L}$ with a collar around the boundary of $U_i$, so that the rest of $U_i$ has area equal to half the area of $S^2$. 
Theorem \ref{thm-superheavy} says that $K_{crit}$ is $SH$-full. 
This implies that it is not stably displaceable, and furthermore that no two such subsets can be disjoint  from each other. 
It is easy to see explicitly that it is necessary to add the collar to $K_{crit}$ in order for these results to hold.

\subsubsection{The case $M=S^2$, $D=$ a point}

Let $M=S^2$, and $D$ be a single point. 
We start by sketching how Theorem \ref{thm:aa} works in this case. 
It is possible to take simpler models for $SC^*(X;\Bbbk)$ and $\widetilde{SC}_\Nov$ than those which appear in the actual proof of the Theorem. 

We take a model for $SC^*(X;\Bbbk)$ which is isomorphic to $\Bbbk[z,z\theta]$ where $z$ is a commutative variable of degree $-2$, and $\theta$ is anticommutative of degree $1$. 
The generator $1$ corresponds to the unique constant orbit, $z^j$ to the fundamental cycle of the Reeb orbit going $j$ times around $D$, and $z^j\theta$ to the point class of the same Reeb orbit. 
The differential $d$ sends $z^j \mapsto 0$ and $z^j\theta \mapsto z^{j-1}$. 
In particular the cohomology vanishes: symplectic cohomology of the disc is zero. 

We have $\Nov = \C[q]$ with $i(q) = 4$. 
We take $\widetilde{SC}_\Nov = SC_\Nov$, and consider the deformed differential $\partial$, where $\partial - d$ sends $z^j \mapsto 0$ and $z^j \theta \mapsto q z^{j+1}$. 
The cohomology of this differential is free of rank 2 over $\Nov$, with a basis given by $1$ and $qz$. 
In particular it is isomorphic to $QH^*(M;\Nov)$, in accordance with Theorem \ref{thm:aa}: the class $1$ corresponds to $1 \in QH^0(M;\Nov)$, and the class $qz$ corresponds to $\mathrm{PD}(\mathrm{pt}) \in QH^2(M;\Nov)$.

Theorem \ref{thm:specseqa} does not apply in this case, because Hypothesis \ref{hyp:gr} is not satisfied: we have $\lambda = 4 > 2$. 
And indeed the conclusion of the Theorem fails, because we can't have a spectral sequence with $E_1$ page vanishing, converging to $QH^*(M;\Nov) \neq 0$. 
The reason the proof of Theorem \ref{thm:specseqa} does not run is that the $\cQ$-filtration on $SC_\Nov$ is not degreewise complete. 
For example, the classes $q^kz^{2k}$ all have degree $0$, but their $\cQ$-values go to $+\infty$. The convergence theorems for spectral sequences all require completeness, and indeed it could not be otherwise: taking the completion does not change the spectral sequence associated to a filtered complex, by inspection of the construction. 
It is easy to verify that the degreewise completion of $(SC_\Nov,\partial)$ is acyclic: for example,
$$1 = \partial\left( \sum_{j=0}^\infty (-qz^2)^j \cdot z\theta\right).$$
This confirms Conjecture \ref{conj:ss-nohyp} in this case, as $QH^*(S^2;\Nov)_{crit} = 0$.

Theorem \ref{thm-superheavy} simply says in this case that a disc occupying half the area of the sphere is $SH$-full, c.f. \cite[Section 1.2.2]{VarolgunesThesis}.

We now offer another perspective on this computation following Remark \ref{rmk:logfano}, which will be useful in the next two sections. First, we take $M=S^2$ and $D' = D'_1 \cup D'_2$ to be an anticanonical divisor on $S^2$, where $D'_1$ and $D'_2$ are distinct points. 
If we equip each point with weight $\lambda'_i = 2$, then this is a special case of Section \ref{sss-toric-fano}: we see $QH^*(M;\Nov')$ as a deformation of $SC^*(X';\Bbbk)$ where $X' = M \setminus D'$. 
In this case $(SC_{\Nov'},\partial')$ is quasi-isomorphic to the complex $$\Nov'[x,x^{-1},\partial_x],$$ where $\partial_x^2=0$, the generator $q'$ of the Novikov ring is in degree $2$, $x$ is in degree $0$, and $\partial_x$ is in degree $1$; the differential $\partial$ is $\Nov'[x,x^{-1}]$-linear and sends
$$ 1 \mapsto 0, \qquad \partial_x\mapsto q'(1-x^{-2}).$$ As expected, this chain complex is degree-wise complete with respect to the $\cQ$-filtration and we obtain $$QH^*(S^2;\Nov') \cong \Nov[x,x^{-1}]/\left\langle x^2-1\right\rangle.$$

Now we consider the case that $\lambda_1 = 0$, $\lambda_2=4$. 
Following the recipe of Remark \ref{rmk:logfano}, if $X = M \setminus D'_2$ then $SC^*(X;\Bbbk)$ should be quasi-isomorphic to $\Bbbk[x,x^{-1},\partial_x],$ where $\partial_x^2=0$, $x$ is in degree $2$, and $\partial_x$ is in degree $-1$; the differential $d$ is $\Bbbk[x,x^{-1}]$-linear and sends
$$ 1 \mapsto 0, \qquad \partial_x\mapsto 1.$$ 
As expected, this chain complex is acyclic. 
The chain complex $(SC_\Nov,\partial)$ is quasi-isomorphic to $\Nov[x,x^{-1},\partial_x]$, with $x$ and $\partial_x$ graded as before, and the generator $q$ of $\Nov$ in degree $4$; the differential $\partial$ sends
$$1 \mapsto 0,\qquad \partial_x \mapsto 1-qx^{-2}.$$ 
Note that as expected, we have an isomorphism of chain complexes
\begin{align*}
(SC_\Nov,\partial) \otimes_\Nov \Nov' &\cong (SC_{\Nov'},\partial') \qquad \text{via the algebra map sending}\\
x & \mapsto q' x,\\
x\partial_x &\mapsto x \partial_x.
\end{align*}
We learned nothing new so far but we believe that this exercise might help unraveling the more complicated examples in the next two sections below.

\subsubsection{The quadric in $\CP^2$}\label{sss-quadric}

Consider $\mathbb{C}\mathbb{P}^2$ with its Fubini-Study symplectic form, and $D$ a smooth quadric with its canonical weight $3$, which does not satisfy Hypothesis \ref{hyp:gr}. $\mathbb{L}$ in this case is the monotone Lagrangian $\mathbb{R}\mathbb{P}^2$, which is known to be stably non-displaceable. On the other hand $\mathbb{R}\mathbb{P}^2$ can be displaced from the Chekanov torus (see \cite{Wu2015}), hence it is not $SH$-full for a general $\Bbbk$.
As was pointed out to us by Leonid Polterovich, it is also known that $\mathbb{R}\mathbb{P}^2$ is $[\mathbb{C}\mathbb{P}^2]$-superheavy over $\mathbb{Z}/2$, see \cite[Example 4.12]{EntovICM}.

Let us now test Theorem \ref{thm:aa} and Conjecture \ref{conj:ss-nohyp} in this case, using the mirror picture outlined in Remark \ref{rmk:logfano}. 
The expectation, following \cite[Section 5.2]{Auroux2007}, is as follows. 

Consider the graded ring
$$R := \Bbbk[x,y,z]/(z(xy-1)-1),\qquad \text{where $|x|=-1$, $|y|=1$, $|z|=0$,}$$
and consider elements $w_1 = y^2z$ and $w_2 = x$ of $R$. 
We set $Y = Spec(R)$. Then $X$ should be mirror to the Landau--Ginzburg model $(Y,w_1)$ while $M$ should be mirror to $(Y_\Nov,w_1+qw_2)$, where $|q|=3$.

We expect $(SC^*(X;\Bbbk),d)$ to be quasi-isomorphic to
\begin{equation}\label{eq:RP2mir}
\left(\bigoplus_p \Lambda^p TY , [w_1,-]\right),
\end{equation}
while $(SC_\Nov,\partial)$ should be filtered quasi-isomorphic to 
$$\left(\bigoplus_p \Lambda^p TY \otimes_\Bbbk \Nov, [w_1 + qw_2,-]\right),$$
with the filtration map given by $\cQ(c \otimes q^a) = -p + a $ for $c \in \Lambda^p TY$. 
We can compute the cohomology of this complex: it comes out as the Jacobian ring of $w_1+qw_2$, which is 
\begin{align*}
\Nov[x,y,z]/(q-y^3z^2,2yz-xy^2z^2,z(xy-1)-1) &\cong \Nov[x,y,z]/(z-1,x-2q^{-1}y^2,y^3-q) \\
&\cong \Nov[y](y^3-q)\\
& \cong QH^*(\CP^2;\Nov).
\end{align*}
This agrees with Theorem \ref{thm:aa} in this case.

Now we turn to Conjecture \ref{conj:ss-nohyp}. We consider two cases:

\textbf{Case 1: $2$ is invertible in $\Bbbk$.} We easily deduce that $1-q(x/2)^{3}$ is nullhomologous; it is also clearly invertible in the $\cQ$-completion. This implies that the cohomology vanishes after $\cQ$-completion.

\textbf{Case 2: $2=0$ in $\Bbbk$.} In this case the Jacobian ring is $\Nov[x,y,z]/(z-1,x,y^3-q) = \Nov[y]/(y^3-q)$. 
It is easy to see that $\cQ$-completion does not change the cohomology.

Both cases are in agreement with Conjecture \ref{conj:ss-nohyp}: if $2$ is invertible, then $\mathrm{PD}(D) \star(-)$ is invertible, so $QH^*(M;\Nov)_{crit} = 0$. On the other hand $[D]$ is $2$-divisible, so if $2=0$, then $QH^*(M;\Nov)_{crit} = QH^*(M;\Nov)$. 

This leads us to conjecture that $\mathbb{RP}^2 \subset \mathbb{CP}^2$ is $SH$-full if the characteristic of $\Bbbk$ is $2$, but not otherwise (Entov's result that $\mathbb{R}\mathbb{P}^2$ is $[\mathbb{C}\mathbb{P}^2]$-superheavy over $\Z/2$ can be considered as further evidence for this conjecture). 
This would imply that $\mathbb{RP}^2$ is non-stably displaceable (which is known), and intersects any monotone Lagrangian which is Floer-theoretically essential over a field of characteristic $2$ (note that this does \emph{not} include the Chekanov torus, as can easily be seen from the superpotential computed in \cite{Auroux2007}). 

\begin{rmk}\label{rmk:RP2}
We sketch some evidence for the mirror symmetry statement \eqref{eq:RP2mir}, in the case that $\mathrm{char}(\Bbbk)=2$. Note that the completion of $X$ is symplectomorphic to $T^*\mathbb{RP}^2$, so $SH^*(X;\Bbbk) \cong H^*(\cL \mathbb{RP}^2;\Bbbk)$ by Viterbo's theorem \cite{ViterboICM,ViterboII,Abbondandolo2006}. We can compute
$$H^*(\cL \mathbb{RP}^2;\Bbbk) \cong H^*(\mathbb{RP}^2;\Bbbk) \oplus \bigoplus_{k >1} H^*(S(T\mathbb{RP}^2);\Bbbk)[k]$$
by \cite{Ziller}, where the first factor comes from the manifold $\mathbb{RP}^2$ of constant loops, and the subsequent factors come from the manifolds $S(T\mathbb{RP}^2)$ of `length-$k$' geodesics. Of course $H^*(\mathbb{RP}^2;\Bbbk) \cong \Bbbk[y]/y^3$ with $|y|=1$, while $H^*(S(T\mathbb{RP}^2);\Bbbk)$ has rank 1 in degrees $0,1,2,3$. On the other hand, one may compute that 
\begin{align*}
H^*\left(\bigoplus_p \Lambda^p TY , [w_1,-]\right) &\cong \Bbbk[x,y,v]/(y^3,y^2x,y^2v)\\
&= \Bbbk[y]/y^3 \oplus \bigoplus_{k>1} \langle x^k,x^ky,x^{k-1} \cdot v, x^{k-1} y \cdot v\rangle,
\end{align*}
where $v=x \partial_x - y \partial_y$ is an anticommuting variable. We identify $\Bbbk[y]/y^3$ as corresponding to the constant loops, and the subsequent factors as corresponding to the length-$k$ geodesics. The degrees match up (we observe that $|v|=1$). We remark that $x=w_2$ is the basic loop around $D$, which corresponds to the family of length-$1$ geodesics, so it makes sense that multiplying by $x$ takes us to the next $k$-value. 
\end{rmk}

\subsubsection{Fano hypersurfaces}\label{sss-fanohyp}

We consider some examples motivated by \cite{SheridanFano}. 
They follow a similar philosophy to Remark \ref{rmk:logfano}, but are a bit different as they are obtained by partially compactifying an affine variety which is of log general type, rather than being log Calabi--Yau.

Let $M = M_{n,a}$ be a smooth hypersurface of degree $a \le n+1$ in $\CP^{n+1}$, and $D = D_{n,a,i}$ a union of $i \le n+2$ generic hyperplanes. 
This fits into the setup of Section \ref{subsec:setup}, and we may take the weights all to be equal to $\lambda = \frac{2(n+2-a)}{i}$. 
In particular, Hypothesis \ref{hyp:gr} is satisfied if and only if $n+2-a \le i$. 
This corresponds to the variety $X_{n,a,i} = M_{n,a} \setminus D_{n,a,i}$ being log Calabi--Yau (in the case of equality) or log general type (otherwise). 
Hypothesis \ref{hyp:gr} is not satisfied precisely when $X_{n,a,i}$ is log Fano.

We conjecture that the mirror to $X_{n,a,i}$ is the Landau--Ginzburg model $(Y_{n,a,i},W_{n,a,i})$, where 
\begin{align*}
Y_{n,a,i} &= [\Bbbk^{n+2}/G_{n,a,i}] \qquad \text{is a stack, where}\\
G_{n,a,i} &= \ker \left(\Z_a^{n+1} \xrightarrow{\sum} \Z_a\right),\qquad \text{and}\\
W_{n,a,i} &= -z_1\ldots z_{n+2} + \sum_{j=i+1}^{n+2} z_j^a,\qquad \text{and furthermore that}\\
\beta_{n,a,i} &= q \cdot \sum_{j=1}^i z_j^a.
\end{align*}
Here we assume that $\Bbbk$ contains all $a$th roots of unity. The group $G_{n,a,i}$ acts torically, preserving $W_{n,a,i}$. 
The variables $z_j$ have degree $(2-\lambda)/a$ for $j\le i$ and degree $2/a$ for $j > i$, and $q$ has degree $\lambda$.

Now let us drop the $n,a,i$ from the notation. 
By taking $M$ to be a Fermat hypersurface, we obtain a natural action of the dual group $G^*$ on $M$, respecting $D$. 
Restricting to the invariant pieces of the relevant group actions, mirror symmetry predicts that 
$$SH^*(X;\Bbbk)^{G^*} \cong H^*\left( \Bbbk[z_1,\ldots,z_{n+2},\partial/\partial z_1,\ldots,\partial/\partial z_{n+2}],[W,-]\right)^{G},$$
and in fact that there is an underlying quasi-isomorphism of $L_\infty$ algebras. 
In accordance with Conjecture \ref{conj:MC}, this gives us
$$H^*\left(\wt{SC}_\Nov,\partial\right)^{G^*} \cong H^*\left(\Nov[z_1,\ldots,z_{n+2},\partial/\partial z_1,\ldots,\partial/\partial z_{n+2}],[W+\beta,-]\right)^{G},$$
and hence
$$QH^*(M;\Nov)^{G^*} \cong Jac(W+\beta)^G.$$
The Jacobian ring has relations
\begin{align*}
\frac{z_1 \ldots z_{n+2}}{z_j} &= q z_j^{a-1} \qquad \text{for $j \le i$}\\
\frac{z_1 \ldots z_{n+2}}{z_j} &= z_j^{a-1} \qquad \text{for $j > i$.}
\end{align*}
Multiplying them together we get that
$$(z_1 \ldots z_{n+2})^{n+1} = q^i (z_1 \ldots z_{n+2})^{a-1}.$$
This allows us to compute that
$$ Jac(W+\beta)^G = \Nov[H]/\left(H^{n+1} - q^i H^{a-1}\right),$$
where $H = z_1\ldots z_{n+2}$.

The class $H$ corresponds to the hyperplane class (except for the case $n+2-a = 1$, when it corresponds to the hyperplane class plus $a! q^{2i}$). 
One can check that this is the correct answer for $QH^*(M;\Nov)^{G^*}$, see \cite{Givental1996,Jinzenji1997}. 
This is in agreement with Theorem \ref{thm:aa}. 

We can also check Conjecture \ref{conj:ss-nohyp} in this case. 
We can factor the defining relation in the Jacobian ring as:
$$H^{n+1} - q^i H^{a-1} = H^{a-1} \prod_{\zeta^{n+2-a} = 1} \left(H - \zeta q^{\frac{i}{n+2-a}}\right).$$
Note that we have $H=z_1 \ldots z_{n+2} = qz_1^a$, from the first relation in the Jacobian ring. 
Thus $\cQ(H)= 1$. 
On the other hand, $\cQ(q^{i/(n+2-a)}) = i/(n+2-a)$. 
Therefore, precisely when Hypothesis \ref{hyp:gr} is not satisfied, the factors $(H - \zeta q^{i/(n+2-a)})$ become invertible in the $\cQ$-completion, as argued in Section \ref{sss-nohyp}. 

The result is that the $\cQ$-completion gives $\Nov[H]/H^{a-1}$, which corresponds to the zero generalized eigenspace (note that Hypothesis \ref{hyp:gr} is satisfied for all $i \ge 1$ in the anomalous case $n+2-a=1$, when this corresponds to the $-a!q$ generalized eigenspace.)

\subsection{Outline}

In Section \ref{s-basics-of-SC} we examine the structure of our symplectic manifold in a neighbourhood of the divisor $D$. 
In particular, we introduce the notion of a `system of commuting Hamiltonians near $D$', and say what it means for a Liouville one-form to be `adapted' to such a system. This completes the statement of the results in Section \ref{ss-rigid-sk}, where these notions were used without being defined.

In Section \ref{s-HF-conventions} we establish our conventions for Hamiltonian Floer theory and relative symplectic cohomology in $M$, and explain how they are related to symplectic cohomology of the exact symplectic manifold $X$. In particular, we establish that the map \eqref{eq:SCtoCF} respects index and action; and we prove the `positivity of intersection'-type result which is used to prove Proposition \ref{prop-pos-int}.

In Section \ref{s-H} we construct the functions $\rho^R$ which are smoothings of $\rho^0$. 
We consider degenerate Hamiltonians of the form $h \circ \rho^R$, explain how to perturb them to obtain non-degenerate time-dependent Hamiltonians $H$, and give estimates for the index and action of their orbits.

In Section \ref{s-proofs} we prove our main results.

\paragraph{Acknowledgments:} N.S. is supported by a Royal Society University Research Fellowship, ERC Starting Grant 850713 -- HMS, and the Simons Collaboration on Homological Mirror Symmetry. U.V was partially supported by ERC Starting Grant 850713 -- HMS and is currently supported by TUBITAK CoCirculation2 grant number 121C034. We are grateful to Yoel Groman, Ana Rita Pires, Paul Seidel, Ivan Smith, and Dmitry Tonkonog for helpful conversations; and to Dan Pomerleano for explaining Conjecture \ref{conj:ss-nohyp} to us.

\paragraph{History of work:} This paper started with M.S.B. and N.S. trying to prove Conjecture \ref{conj:MC}. A solution was announced in 2015, but never appeared. The project languished, until U.V. joined the collaboration in 2019 and pushed it to completion in its current form. M.S.B. and N.S. apologize for the long delay between announcement and appearance of the work.


\section{Symplectic divisors} \label{s-basics-of-SC}

\subsection{Basics}\label{ss=div-basic}

We recall some notions from \cite[Section 2.1]{Tehrani2018a}. 
Let $(M, \w)$ be a $2n$-dimensional closed symplectic manifold and let $D = \cup_{i=1}^N D_i$ be a symplectic divisor in $(M, \w)$. 
This means that for each $i$, $D_i \subset M$ is a connected smooth closed submanifold with real codimension two and for each subset $I \subset [N]$ the intersection $\bigcap_{i \in I} D_i$ is transverse and
$$D_I := \bigcap_{i \in I} D_i \subset M$$
is a symplectic submanifold.
Since the $D_i$ intersect transversally, for each $I \subset [N]$ there is an isomorphism of vector bundles 
\begin{equation}\label{e:Transverse}
	N_{M} D_I \xrightarrow{\sim} \bigoplus_{i \in I} N_M D_i |_{D_I}
\end{equation}
over $D_I$, induced by the inclusions $TD_I \subset TD_{i}|_{D_I}$.
Recall the normal bundle $N_M D$ for any symplectic submanifold $D \subset (M, \w)$ has a \emph{symplectic orientation} induced
by the symplectic orientations of $TD$ and $TM$.  
\begin{defn}
	A symplectic divisor $D \subset (M, \w)$ is
	\begin{enumerate}
	\item[(i)] a \emph{simple crossings (SC) divisor} if
	\eqref{e:Transverse} is an isomorphism of oriented vector bundles, where each normal bundle is
	given its symplectic orientation, for all $I$.
	\item[(ii)] \emph{orthogonal} if for all $i \not= j$ and $x \in D_i \cap D_j$ the
	$\w$-normal bundle $(T_xD_i)^\w \subset T_xM$ is contained in $T_xD_j$.
	\end{enumerate}
\end{defn}

\begin{rmk}\label{rmk:Dorth}
In \cite[Section 5]{McL12T} McLean proved that any SC divisor $D \subset (M, \w)$ can be smoothly isotoped in the space of SC divisors to an orthogonal SC divisor $D' \subset (M,\w)$; and that $X' = M \setminus D'$ is convex deformation equivalent to $X = M \setminus D$. This implies that $SH^*(X;\Bbbk) \cong SH^*(X';\Bbbk)$, by \cite[Lemma 4.11]{McL12T}. 
These results mean that it suffices to prove Theorems \ref{thm:aa} and \ref{thm:specseqa} under the assumption that $D$ is orthogonal.
\end{rmk}

Setting $X = M \setminus D$, by Lefschetz duality (e.g. Proposition 7.2 of \cite{Dold}) we have 
\begin{equation}\label{e:RelHom}
	H_2(M, X) \cong \Z^{N} \quad\mbox{where}\quad A \mapsto (A\cdot D_i)_{i=1}^N.
\end{equation}
The inverse is given by mapping the $i$th basis vector to a disk $u_i: (\bD, \partial \bD) \to (M, X)$ that is disjoint from the other $D_j$ and 
with intersection number $u_i \cdot D_i = 1$. 
The dual basis vectors of $H^2(M,X) \cong \Z^N$ are what we called $\mathrm{PD}^{rel}(D_i)$ in Section \ref{subsec:setup}.

Assume that  
\[ \bfkappa= \sum_{i} \kappa_i \mbox{PD}^{rel}[D_i] \in H^2(M,X;\R)\]
is a lift of $[\w]$ under the map $H^2(M, X; \R) \to H^2(M; \R)$ with $\kappa_i\in\mathbb{R}$. 
\begin{rmk}
In the setup from Section \ref{subsec:setup}, $\bfkappa$ will be $\kappa\bflambda$.
\end{rmk}

Now consider a de Rham representative $(\w,\theta)$ for $\bfkappa$ consisting of the symplectic form $\w$ together with a one-form $\theta \in \Omega^1(X)$ satisfying $d\theta = \w|_X$, and 
$$
	\kappa_i = \int_{u_i} \w - \int_{\partial u_i} \theta.
$$

Following McLean \cite{McL12T, McLean2016} we call $\kappa_i$ 
the \emph{wrapping numbers for $D$ with respect to $\theta$}, though we use the opposite sign convention than in \cite{McL12T}.

\subsection{Systems of commuting Hamiltonians}\label{ss:scH}

\begin{defn}
Let $D=\cup_i D_i$ be an SC divisor in a closed symplectic manifold $(M,\omega)$, and $R>0$. A \emph{system of commuting Hamiltonians (scH) near $D$, of radius $R$,} is a collection of open neighborhoods $UD_i\supset D_i$ and proper smooth functions $r_i: UD_i\to [0,R)$, for each $i$, with the following properties.
For each $i$, 
\begin{itemize}
\item $r_i$ generates an $\mathbb{R}/\mathbb{Z}$ action on $UD_i$, and $r_i^{-1}(0)=D_i$.
\item The fixed point set of the $\mathbb{R}/\mathbb{Z}$ action on $UD_i$ is $D_i$. 
\item The $\mathbb{R}/\mathbb{Z}$ action on $UD_i\setminus D_i$ is free.
\end{itemize} 

For all pairs $i,j$,
\begin{itemize}
\item $UD_i\cap UD_j$ is invariant under the $\mathbb{R}/\mathbb{Z}$ action generated by $r_i$.
\item The Hamiltonians $r_i$ and $r_j$ Poisson commute on $UD_i\cap UD_j$.
\end{itemize}
We will denote a scH near $D$ of radius $R$ with the notation $\{r_i:UD_i\to[0,R)\}.$
\end{defn}

Note that for any scH of radius $R$, we can `shrink' it to an scH of radius $R'<R$ by replacing $UD_i$ with $\{r_i<R'\}$ for each $i$.

\begin{prop}
Let $D$ be an SC divisor in a closed symplectic manifold $(M,\omega)$. If $D$ admits a scH, then it is orthogonal.
\end{prop}
\begin{proof}
Assume that $\{r_i:UD_i\to[0,R)\}$ is a scH near $D$. We need to show that  for all $i \not= j$ and $x \in D_i \cap D_j$ the symplectic orthogonal $(T_xD_i)^\w \subset T_xM$ is contained in $T_xD_j$. 

We consider the action of $S:=\mathbb{R}/\mathbb{Z}$ on $T_x M$ induced by $r_i$. The action on $T_xD_i\subset T_xM$ is trivial, since $D_i$ is fixed pointwise under the action of $S$. The action on $UD_i\cap UD_j$ leaves $\{r_j=0\}\cap UD_i\cap UD_j$ invariant by the Poisson commutativity property. Therefore $T_xD_j$ is an invariant subspace of $T_xM$ under the $S$ action. Finally, since the action of $S$ on $T_xM$ preserves the symplectic pairing, $(T_xD_i)^\w\subset T_xM$ is also an invariant subpace. 

Note that the action of $S$ on $(T_xD_i)^\w$ cannot be trivial by the Bochner linearization theorem, as $x$ does not have a neighhorhood on which $S$ acts trivially. Now we finish the proof with the following claim:
\begin{itemize}
\item  Assume that $V$ is a finite dimensional symplectic representation of $S$, which is the direct sum of two representations $W\oplus E$, where $W$ is the trivial representation on a symplectic codimension $2$ subspace, $E$ is not the trivial representation, and $E$ and $W$ are symplectically orthogonal. Let $W'$ be another codimension $2$ symplectic subspace of $V$ which is invariant under the action of $S$. Then if $W'$ is transverse to $W$, it has to contain $E$.
\end{itemize}
The proof of this statement is as follows. There exists $w+e\in W'$ with $e \neq 0$, as $W'$ is transverse to $W$. For any $\theta \in S$ we have $\theta \cdot (w+e) \in W'$, hence $\theta \cdot (w+e) - (w+e) = \theta \cdot e - e \in W'$. We may choose $\theta$ so that $\theta \cdot e \neq e$, so $W'\cap E\neq\{0\}$. 
This implies that $E \subset W'$ as required.
\end{proof}

\begin{defn} Let $D=\cup_{i=1}^N D_i$ be an SC divisor in a closed symplectic manifold $(M,\omega)$ and let $\{r_i:UD_i\to[0,R)\}$ be a scH near $D$. For all $I\subset [N],$ define $UD_I := \cap_{i \in I} UD_i$. We obtain a $(\mathbb{R}/\mathbb{Z})^I$ action on $UD_I$ with a moment map $$r_I: UD_I\to [0,R)^I$$ whose components are given by $r_i$, for $i\in I$.\end{defn}

\begin{prop}\label{prop-scH-exists}
Let $D$ be an orthogonal SC divisor in a closed symplectic manifold $(M,\omega)$. Then $D$ admits a scH.
\end{prop}
\begin{proof}
This is an immediate consequence of  \cite[Lemma 5.14]{McL12T}, where for each $i$, we use the well-defined radial coordinate of the symplectic disk bundle over $D_i$ in the statement as our $r_i$ (the domain is the symplectic disk bundle of course). It is trivial to see that this gives a scH near $D$. 
\end{proof}

\begin{rmk} It is natural to ask whether all systems of commuting Hamiltonians come from standard tubular neighborhoods in the sense of McLean. Even if this is the case,  the extra choice of a standard tubular neighborhood on top of a system of commuting Hamiltonians is not needed for our constructions and arguments. \end{rmk}

\subsection{Adapted Liouville one-forms}

\begin{defn}\label{def-adapted}
Let $D$ be an SC divisor in a closed symplectic manifold $(M,\omega)$ and let $\{r_i:UD_i\to[0,R)\}$ be an admissible scH near $D$. We call a one-form $\theta \in \Omega^1(M\setminus D)$ satisfying $d\theta = \w|_{M\setminus D}$ and with wrapping numbers $\kappa_i>0$ \emph{adapted} to $\{r_i:UD_i\to[0,R)\}$ if the Liouville vector field $Z$ of $\theta$ satisfies $$Z(r_i) = r_i - \kappa_i$$ over $UD_i \setminus D$, for all $i$.
\end{defn}

\begin{prop}Let $D$ be an orthogonal SC divisor in a closed symplectic manifold $(M,\omega)$. Assume that $$[\omega]=\sum_{i} \kappa_i \cdot \mathrm{PD}(D_i)\quad\text{ in $H_2(M,\mathbb{R})$,}$$ with $\kappa_i>0$. Then there exists $\{r_i:UD_i\to[0,R)\}$ a scH near $D$ for which there exists an adapted $\theta \in \Omega^1(M\setminus D)$ with wrapping numbers $\kappa_i$.
\end{prop}

\begin{proof}
We use a scH as in the proof of Proposition \ref{prop-scH-exists}. Then, a one-form $\theta$ on $M\setminus D$ produced by \cite[Lemma 5.17]{McL12T} is adapted in the sense of Definition \ref{def-adapted}, as we show below. Note that by the relative de Rham isomorphism, there is a primitive $\theta'$ defined on $M\setminus D$ such that the relative cohomology class in $H^2(M,M\setminus D)$ defined by $(\omega,\theta')$ is equal to $\sum \kappa_i \cdot \mathrm{PD}(D_i)$, which is why we can use McLean's lemma.

Using McLean's notation for the moment, on the fibers of the projections $\pi_I: UD_I \to D_I$ we have
\begin{equation}\label{e:NiceFormTheta}
		\theta|_{F_I^*} = \sum_{i \in I} (r_i - \kappa_i)\, d\phi_i
\end{equation}
where $F_I^* \cong \prod_{i \in I} (\bD_{R} \setminus 0)$ is the product of punctured disks. 
Using \eqref{e:NiceFormTheta}, we have
	$$
		Z(r_i) = \theta(X_{r_i}) = \theta(\partial_{\phi_i}) =r_i-\kappa_i,
	$$ 
as required. 
\end{proof}

\begin{remark}
Again one could ask whether every Liouville one-form adapted to a system of commuting Hamiltonians is adapted to some compatible standard tubular neighborhood in the sense of McLean. Whatever the answer might be, the flexibility that we achieved in these two sections already shows itself in the toric examples of Section \ref{sss-toric-fano}.
\end{remark}

\subsection{Admissibility}

\begin{defn}
Let $D=\cup_{i=1}^N D_i$ be an SC divisor in a closed symplectic manifold $(M,\omega)$, and let $\{r_i:UD_i\to[0,R)\}$ be a scH near $D$. Given $I \subset [N]$, a \emph{standard chart} $(U,\phi)$ in $UD_I$ is an $(\mathbb{R}/\mathbb{Z})^I$-invariant open subset $U\subset UD_I$ and a $(\mathbb{R}/\mathbb{Z})^I$-equivariant symplectic embedding $$\phi: U\to \mathbb{C}^I\times \mathbb{C}^{n-|I|},$$  where we use the action of $(\mathbb{R}/\mathbb{Z})^I$ on $\mathbb{C}^I\times \mathbb{C}^{n-|I|}$ given by $$\theta\cdot((z_i)_{i\in I}, w)=((e^{2\pi i\theta_i}z_i)_{i\in I},w)\text{ for all } \theta\in(\mathbb{R}/\mathbb{Z})^I\text{ and }((z_i)_{i\in I},w)\in\mathbb{C}^I\times \mathbb{C}^{n-|I|} .$$
\end{defn}

\begin{lem}\label{lem-SC-st-neigh}
Let $D=\cup_{i=1}^N D_i$ be an SC divisor in a closed symplectic manifold $(M,\omega)$ and let $\{r_i:UD_i\to[0,R)\}$ be a scH near $D$. For every $I\subset [N]$ and $x\in D_I$, there exists a standard chart $(U,\phi)$ in $UD_I$ containing $x$.
\end{lem}
\begin{proof}
This immediately follows from the equivariant Darboux theorem \cite[Theorem 22.1]{Guillemin1984}.
\end{proof}

We now choose an arbitrary Riemannian metric on $M$, and let $\mathrm{inj}(M)$ be the injectivity radius with respect to this metric. 
We call a standard chart $(U,\phi)$ in $UD_I$ \emph{admissible} if $U$ is contractible and has diameter $<\mathrm{inj}(M)/2$. 
The significance of admissibility for us is that it guarantees uniqueness of caps:

\begin{lem}\label{lem:adm-uniq}
If $\gamma:S^1\to M$ is a loop contained in some admissible standard chart, then there exists a disc bounding $\gamma$, whose image is contained inside an admissible chart. Moreover, such a disc is independent of the choice of admissible chart containing $\gamma$, up to homotopy rel. boundary in $M$.
\end{lem}
\begin{proof}
The existence is clear, as admissible standard charts are contractible. The uniqueness follows as the union of two admissible standard charts containing $\gamma$ has diameter $< \mathrm{inj}(M)$, hence is contained in a ball of radius $<\mathrm{inj}(M)$. As the ball is contractible, the caps in the two charts are homotopic rel. boundary in $M$.
\end{proof}

\begin{defn}
Let $D=\cup_{i=1}^N D_i$ be an SC divisor in a closed symplectic manifold $(M,\omega)$. We call a scH near $D$ \emph{admissible} if for every $I\subset [N]$ and $y\in UD_I$, there exists an admissible standard chart $(U,\phi)$ in $UD_I$ with $y \in U$.
\end{defn}

\begin{lem}
Let $D$ be an SC divisor in a closed symplectic manifold $(M,\omega)$, and $\{r_i:UD_i \to [0,R)\}$ a scH near $D$. 
Then any sufficiently small shrinking of the scH is admissible.
\end{lem}
\begin{proof}
First note that any standard chart around $x \in D_I$ can be shrunk so that it is admissible. 
Therefore we have a neighbourhood of $D_I$ given by the union of all admissible standard charts. 
By shrinking the scH sufficiently, we may ensure that $UD_I$ is contained in the neighbourhood, for all $I$.
\end{proof}

\begin{rmk}\label{rmk-weakly-admissible}
In Section \ref{ss-ham-review}, we will define a cap for a loop $\gamma:S^1 \to M$ to be an equivalence class of discs $u$ bounding $\gamma$ under the equivalence relation $u_1\sim u_2$ if $\int u_1^*\omega= \int u_2^*\omega$. Therefore, we could get away with the following weaker notion of admissibility for the purposes of the present paper. 
We call a standard chart weakly admissible if it is simply connected. Assume that we have a loop $\gamma$ inside $UD_I$ that is the orbit of a point under the action of a one dimensional subgroup $S$ of $(\R/\Z)^I$. We claim that the symplectic area of a cap of $\gamma$ that is contained inside a weakly admissible standard chart $U$ (assuming such charts exist) only depends on $\gamma$, i.e. it is independent of $U$ and the cap chosen inside of $U$. The reason is because we can then compute the symplectic area by transporting everything into $\mathbb{C}^I\times \mathbb{C}^{n-|I|}$ and see that it is equal to $l(0)-l(p)$, where $l:\mathbb{R}^I\to \mathbb{R}$ is a function whose pre-composition with $r_I$ generates the action of $S$ and $p$ is the point of $\mathbb{R}_{\geq 0}^I$ above which $\gamma$ lives. Hence, for such $\gamma$ existence of a weakly admissible standard chart determines uniquely an equivalence class of caps. This would be enough for our purposes.
\end{rmk}

\section{Quantum, Hamiltonian Floer, and symplectic cohomology}\label{s-HF-conventions}

\subsection{Quantum and Hamiltonian Floer cohomology}\label{ss-ham-review}

In this section, $(M, \w)$ will be a closed symplectic manifold such that $
	2\kappa c_1(TM) = [\w]$ on $\pi_{2}(M)$ for some $\kappa > 0
$.

Let $A'$ be the subgroup $\{2c_1(TM)(B): B \in \pi_2(M)\} \subset \Z$ and set $\Nov' = \Bbbk[A']$, graded by $i(e^a) = a$.

Let $\gamma:S^1 \to M$ be a nullhomotopic loop in $M$. 
A \emph{cap} for $\gamma$ is an equivalence class of disks $u:\mathbb{D} \to M$ bounding $\gamma$, 
where $u \sim u'$ if and only if the Chern number of the spherical class $[u-u']$ vanishes: $c_1(TM)(u-u') = 0$. 
The set of caps for $\gamma$ is a torsor for $A'$, which acts via 
$$a \cdot (\gamma, u) = (\gamma, u\# C) \qquad \text{where} \qquad 2c_1(TM)(C) = a.$$

Given a non-degenerate Hamiltonian $F: S^1 \times M \to \R$, let $\cP_F$ denote the set of contractible one-periodic orbits of $F$, and let $\tilde{\cP}_F$ be the set of orbits equipped with a cap.
Elements $\tilde{\gamma} = (\gamma, u) \in \tilde{\cP}_F$ have a $\Z$-grading 
and an action
$$
	i(\gamma, u) = \CZ(\gamma, u) + \frac{\dim(M)}{2} \quad\mbox{and}\quad
	\cA_F(\gamma, u) := \int_{S^1} F(t, \gamma(t))\,dt + \int_\bD u^*\w, 
$$
and these are compatible with the action of $A'$ in that
$$
	i(a \cdot (\gamma, u)) = i(\gamma, u) + a
	\qquad\text{and}\qquad
	\cA(a \cdot (\gamma, u)) = \cA(\gamma, u) + \kappa a\,.
$$
Note that the `mixed index'
$$
	i_{mix}(\gamma) := i(\gamma,u) - \kappa^{-1}\cA(\gamma,u)
$$
is independent of the cap $u$.

Define $CF^*(M, F)$ to be the free $\Z$-graded $\Bbbk$-module generated by $\tilde{\cP}_F$.
It is naturally a graded $\Nov'$-module, via $e^a \cdot (\gamma,u) := a \cdot (\gamma,u)$. It also admits a Floer differential after the choice of a generic $S^1$-family of $\omega$-compatible almost complex structures (which we suppress from the notation). The differential is $\Nov'$-linear, increases the grading by $1$, does not decrease action, and squares to zero.

One can also define continuation maps $CF(M, F_0) \to CF(M, F_1)$ in the standard way by choosing a smooth function $\mathcal{F}: \R_s\times S^1 \times M \to \R$, which is equal to $F_0$ for $s\ll 0$ and to $F_1$ for $s \gg 0$, as well as an $\R\times S^1$ dependent family of $\omega$-compatible almost complex structures, which together satisfy a regularity condition. Continuation maps are $\Nov'$-linear chain maps. If the continuation maps are defined using monotone Floer data, which means $\frac{\partial \mathcal{F}}{\partial s}\geq 0$, then the continuation map $CF(M, F_0) \to CF(M, F_1)$ does not decrease action.

\begin{rmk}
We would like to stress that the discussion of Hamiltonian Floer theory that we gave here is slightly simpler than the general theory due to our positive monotonicity assumption. In particular, we did not need to complete $\Nov$ or our Hamiltonian Floer groups, which is necessary in general for the potential infinite sums to make sense. For details, we refer the reader to \cite{Salamon1999}. Apart from the ones that we have explicitly stated above, our conventions for Hamiltonian Floer theory agree with (1), (2), (3) and (5) of Section 3.1 in \cite{Varolgunes2018}.
\end{rmk}

Let $A \subset \Q$ be a subgroup such that $A' \subset A$. 
Let $\Nov = \Bbbk[A]$, with the same grading convention $i(e^a) = a \in \Q$; then we have an inclusion $\Nov' \subset \Nov$.  
(Eventually we will take $A$ and $\Nov$ to be as defined in the beginning of Section \ref{ss-defintro} but we choose to be more general for a while.)

Let us define the $\Nov$-cochain complex $$CF^*(M,F; \Nov):=CF^*(M, F)\otimes_{\Nov'}\Nov.$$
We denote the cohomology of this cochain complex by $HF^*(M,F; \Nov) := H^*(CF^*(M,F;\Nov),\partial)$. 
There exists a natural PSS chain map:
$$
	C^*(M; \Bbbk)\otimes_\Bbbk \Nov \to CF^*(M, F;\Nov),
$$ which is known to be a quasi-isomorphism \cite{PSS}. The PSS map is well-defined up to chain homotopy and compatible with chain level continuation maps up to chain homotopy.

We now introduce the notion of `fractional caps' of orbits.
A fractional cap for $\gamma$ is a formal expression $u+a$, where $u$ is a cap for $\gamma$ and $a \in \R$, and we declare $u+a \sim u'+a'$ if and only if $a-a' \in A'$ and $u' = (a-a') \cdot u$. 
There is a well-defined index and action associated to a fractional cap:
$$
	i(\gamma,u+a) := i(\gamma,u) + a,\qquad \cA(\gamma,u+a) := \cA(\gamma,u) + \kappa a.
$$
There is a natural bijection between the $\Bbbk$-basis $(\gamma,u) \otimes e^a$ of $CF^*(M,F;\Nov)$, and the set of fractionally-capped orbits $(\gamma,u+a)$ with $a \in A$.



\subsection{Relative symplectic cohomology}\label{ss-rel-SH}

Let $M, \omega ,\kappa, \Nov$ be as in Section \ref{ss-ham-review}. We now define relative symplectic cohomology for compact subsets of $M$ over $\Nov$, referring to \cite{Varolgunes2018} for the details. As briefly mentioned in the introduction (see Section \ref{ss-rigid-sk}, especially the footnote on pages 4-5), the construction below is slightly different than the one in \cite{Varolgunes2018}. Namely, here we use capped orbits (in particular we only consider contractible orbits) and keep track of the caps rather than weighting Floer solutions using a formal variable.

Let $K\subset M$ be compact. We call the following data a choice of acceleration data for $K$:
\begin{itemize}
\item $H_1\leq H_2\leq\ldots$ a monotone sequence of non-degenerate one-periodic Hamiltonians $H_i:  S^1\times M\to \mathbb{R}$ cofinal among functions satisfying $H\mid_{S^1\times K}<0$. In other words, for every $(t,x)\in  S^1\times M$,
$$
H_i(t,x)\xrightarrow[i\to+\infty]{}\begin{cases}
0,& x\in K,\\
+\infty,& x\notin K.
\end{cases}
$$
\item A monotone homotopy of Hamiltonians $H_{i,i+1}:[i,i+1]\times S^1\times M\to\mathbb{R}$, for all $i$, which is equal to $H_i$ and $H_{i+1}$ at the corresponding end points.
\item A $\mathbb{R}_{\geq 1}\times S^1$-family of $\omega$-compatible almost complex structures.
\end{itemize}

We denote the acceleration data as a single family of time-dependent Hamiltonians and almost complex structures $(H_\tau, J_\tau)$, $\tau\in \mathbb{R}_{\geq 1}$. We also fix an non-decreasing surjective smooth map $(-\infty,\infty)\to [0,1]$. Given a $[i,i+1]$-dependent family of Hamiltonians and almost complex structures, we use this map to write down a Floer equation for maps from $\mathbb{R}\times S^1$ to $M$. Let us call the resulting $\mathbb{R}\times S^1$-family of Hamiltonians and almost complex structures the associated Floer data.

We require the acceleration data $(H_\tau, J_\tau)$ to satisfy the following two assumptions:

\begin{enumerate}
\item For each $i\in\mathbb{N}$, $(H_i,J_i)$ is regular.
\item For each $i\in\mathbb{N}$, the Floer data associated to  $(H_\tau, J_\tau)_{\tau\in [i,i+1]}$ is regular.
\end{enumerate}

Given acceleration data $(H_\tau, J_\tau)$,
Hamiltonian Floer theory provides a $1$-ray of  Floer $\Nov$-cochain complexes, called a {\it Floer 1-ray}: \begin{align*}
\mathcal{C}(H_\tau,J_\tau):= CF^*(M, H_1; \Nov)\to CF^*(M, H_2; \Nov)\to\ldots 
\end{align*}

The horizontal arrows are Floer continuation maps defined using the monotone homotopies appearing in the acceleration data. Recall that a cylinder $u$ contributing to a Floer differential or a continuation map has non-negative topological energy
\begin{equation}
\label{eq:E_top}
E_{top}(u)=\int_{S^1}\gamma_{out}^* H_{out}\,dt-\int_{S^1}\gamma_{in}^* H_{in}\,dt+\int_{\R \times S^1} u^*\omega \geq 0,
\end{equation}
where $\gamma_{out}$, $\gamma_{in}$ are the asymptotic orbits of $u$, and $H_{out}$, $H_{in}$ are the Hamiltonians at the corresponding ends. (For Floer differentials, $H_{out}=H_{in}=H_i$ and for continuation maps, $H_{out}=H_{i+1}$, $H_{in}=H_i$ for some $i$.)

\begin{rmk}We also note that the inequality in (\ref{eq:E_top}) comes from the more general inequality
\begin{align}\label{ineqTopE}
E_{top}(u)\geq \int_{\mathbb{R}\times S^1}\left(\frac{\partial H}{\partial s}\right)(u(s,t),s,t)dsdt,
\end{align} where $u$ is a solution of the Floer equation for an arbitrary $H: \mathbb{R}\times S^1\times M\to \mathbb{R}$ which is $s$-independent at the ends.
\end{rmk}

From now on, we will use the terminology introduced in Section \ref{sscomplete} freely. We apologetically ask the reader to take a look at it before moving further. Using the grading and action considerations from Section \ref{ss-ham-review}, $\mathcal{C}(H_\tau,J_\tau)$ becomes a 1-ray in $FiltCh_\Nov$. We define the $\Nov$-cochain complexes $tel(\mathcal{C}(H_\tau,J_\tau))$ and $\widehat{tel}(\mathcal{C}(H_\tau,J_\tau))$ as in Section \ref{sscomplete}. We can now repeat Section 3.3.2 of \cite{Varolgunes2018} in this set-up.

\begin{prop}\label{prop:SHrelwelldef}
	For two different choices of acceleration data for $K$, $(H_\tau, J_\tau)$ and $(H_\tau', J_\tau')$,
there is a canonical isomorphism
$$H^*\left(\widehat{tel}(\mathcal{C}(H_\tau, J_\tau))\right) \cong H^*\left(\widehat{tel}(\mathcal{C}(H_\tau', J'_\tau))\right)$$
 of $\mathbb{Q}$-graded $\Nov$-modules.
\end{prop}

Hence, we define \begin{align*} 
SH_M^*(K; \Nov):=H^*\left(\widehat{tel}(\mathcal{C}(H_s,J))\right).
\end{align*}

\begin{prop}
There are canonical restriction maps of $\mathbb{Q}$-graded $\Nov$-modules for $K\subset K'$:
$$SH_M^*(K'; \Nov)\to SH_M^*(K; \Nov).
$$
\qed
\end{prop}

We finally list the three properties we will need of relative symplectic cohomology. Here is the first one.

\begin{thm}
	\label{analog_thesis1}
	Assume that $tel(\mathcal{C}(H_\tau,J_\tau))$ is degreewise complete. Then $SH_M^*(K; \Nov)=QH^*(M;\Nov).$
\end{thm}
\begin{proof} 
Follows from the basic properties of the PSS maps discussed at the end of Section \ref{ss-ham-review} along with the diagram  \eqref{diag-limits} and the fact that a direct limit of quasi-isomorphisms is a quasi-isomorphism.
\end{proof}

Before we state the second property, we note the following important statement from Hamiltonian Floer theory.

Let $H: S^1 \times M \to \R$ a non-degenerate Hamiltonian and $J$ an $S^1$-dependent almost complex structure compatible with $\omega$. Assume that $(H,J)$ is regular and fix $\Delta\geq 0$. 
\begin{itemize}
\item The Floer data $(H_s: = H+\Psi(s)\Delta,J_s:=J)$, where $\Psi:\mathbb{R}\to \mathbb{R}$ is a smooth function that is equal to $0$ for $s<-1$ and to $1$ for $s>1$, is regular. This is a standard fact in Floer theory noting that adding $\Psi(s)\Delta$ does not change the Floer equation.
\item The resulting continuation map $$c_\Psi: CF^*(M, H) \to CF^*(M,H + \Delta)$$ is the naive map which sends each capped orbit to itself. Yet, note that the action of the capped orbit for $H+\Delta$ is $\Delta$ more than its action for $H$.
\end{itemize}
Let us fix a non-decreasing $\Psi$ for the proof below. Let us denote the continuation map above for any $H$ and $\Delta$ by $c_\Psi$ by abuse of notation.

\begin{thm}
\label{analog_thesis2}
If $K$ is stably displaceable, then $SH_M^*(K; \Nov)=0$.
\end{thm}
\begin{proof}[Sketch of proof]The proof is identical to that in the Section 4.2 of \cite{VarolgunesThesis} up to minor modifications. We provide an overview of the proof for completeness.

Let us first prove the result when $K$ is displaceable. Let $(H_\tau, J_\tau)$ be a choice of acceleration data for $K$ and $H:[0,1]\times M\to \mathbb{R}$ be a function whose time-$1$ Hamiltonian flow $\phi:M\to M$ displaces $K$. In fact, $\phi$ displaces a domain neighborhood $D$ of $K$. Assume that $H_\tau$'s are so that $\partial D$ is a level set of $H_1$ for all $t\in S^1$, and $H_\tau=H_1+\tau-1$ on $M-int(D)$ for all $\tau$. 

We recall an elementary construction for reparametrizing Hamiltonian flows . Let $I=[0,T]$ and $I'=[0,T']$ be closed intervals, and $\psi: I'\to I$ be a smooth map which sends $0$ to $0$ and $T'$ to $T$. Then, the time $T$-flow of the time dependent Hamiltonian vector field $X_t$, $t\in I$ of $h: I\times M\to \mathbb{R}$ and the time-$T'$ flow of $X'_t$, $t\in I'$ of $(h\circ(\psi\times id))\cdot\frac{d\psi}{dt} : I'\times M\to \mathbb{R}$ are the same map $M\to M$.\footnote{We warn the reader that there is a typo in the relevant formula in  \cite{VarolgunesThesis}.} 

Let us fix a non-decreasing function $\psi: [0,1/2]\to [0,1],$ which is locally constant in a neighborhood of the endpoints of $[0,1/2]$.

Using the reparametrization construction with $\psi$, starting with $H_{L},H_{R}:M\times [0,1]\to \mathbb{R}$ we can cook up a new Hamiltonian $H_{L}\upphi H_R:M\times \mathbb{R}/\mathbb{Z}\to \mathbb{R}$, such that the $H_L$ and $H_R$ parts are supported in $(1/2,1)$ and $(0,1/2)$ respectively. The Hamiltonian flow of $H_{L}\upphi H_R$ is tangent to $X_{H_R}$ first. After not moving for a short period, it arrives at $\phi^1_{H_R}$ in less than 1/2-time, and stops for a while. At some point after time $1/2$, it starts moving again, this time being tangent to $X_{H_L}$, and reaches to $\phi^1_{H_L}\circ\phi^1_{H_R}$ before time-1. It then stops for a little until time 1, after which it repeats this flow.

We define $SH_M^*(K,H;\Nov)$ via the family $H\upphi H_s$ in the same way we defined $SH_M^*(K;\Nov)$. Note that this construction does not use that $H$ displaces $K$. In particular, we can define $SH_M^*(K,0;\Nov)$, and it follows from Lemma 4.2.1 of \cite{VarolgunesThesis} that $SH_M^*(K,0;\Nov)$ is isomorphic (as a graded $\Nov-$module) to $SH_M^*(K;\Nov)$. Here and in the future, by abuse of notation, we denote the constant function $M\times [0,1]\to \mathbb{R}$, sending everything to $\Delta\in\mathbb{R}$ by $\Delta.$

The next step is to show that $SH_M^*(K,H;\Nov)$ is isomorphic to  $SH_M^*(K,0;\Nov),$  which is true for arbitrary $H$. We can find a $\Delta\geq 0$ such that $$-\Delta\leq H(x,t)\leq \Delta, $$ for all $(x,t)\in M\times [0,1]$. This implies that for any $G: M\times [0,1]\to \mathbb{R},$ we have \begin{align*}
-c\Delta+0\upphi G\leq H\upphi G\leq c\Delta+0\upphi G\leq 2c\Delta+H\upphi G,
\end{align*} where $c>0$ is a constant that depends on our choice of $\psi$.

Hence we obtain filtered chain maps $$tel(\mathcal{C}(-c\Delta+0\upphi H_s))\to tel(\mathcal{C}(H\upphi H_s))\to tel(\mathcal{C}(c\Delta+0\upphi H_s))\to tel(\mathcal{C}(2c\Delta+H\upphi H_s)).$$The composition of the first two maps is filtered chain homotopic to the map obtained from $c_\Psi's$ as explained right before the theorem using a filling in $3$-slits argument. The same result is true for the composition of last two maps.

Using Lemma \ref{lemstrongtel}'s last statement and the second bullet point of Lemma \ref{lemcomplete}, we obtain that there is a chain of maps $$SH_M^*(K,0;\Nov)\to SH_M^*(K,H;\Nov)\to SH_M^*(K,0;\Nov)\to SH_M^*(K,H;\Nov),$$ where the composition of the first two and the last two maps are isomorphisms. This implies the result.

The main point of the proof is to show that $SH_M^*(K,H;\Nov)=0$ for the displacing Hamiltonian $H$ from the beginning of the argument. This uses Lemma \ref{unboundedcofinal}. The more detailed claim is that a slightly modified version of the family $H\upphi H_s$ gives rise to a $1$-ray that satisfies the conditions of Lemma \ref{unboundedcofinal}. The actual proof of this is too long to include here (see Section 4.2.3 of \cite{VarolgunesThesis}). Let us instead explain the intuition behind the proof. Let $\gamma:S^1\to M$ be a $1$-periodic orbit of $H\upphi H_s$ for some $s$. Because $\phi$ displaces $D$, either $\gamma(0)$ or $\gamma(1/2)=\phi^{-1}(\gamma(0))$ needs to lie outside of $D$. Then conservation of energy and that $\partial D$ is a level set of $H_s$ for all times shows that in fact we have $\gamma([0,1/2])\subset M\setminus D$. Now if we could use parametrized moduli spaces and cascades instead of continuation maps, we would have our proof. This relies in the fact that $\gamma([0,1/2])\subset M\setminus D$ holds for all $1$-periodic orbits of all $H\upphi H_s$ and that $\frac{\partial H_s}{\partial s}=1$ in $M\setminus D$: the actions increase with a constant rate as we follow the orbits and accidental solutions can only further increase the action. There are technical difficulties in making this work, so we refer the reader to \cite{VarolgunesThesis} for the actual proof. 

We move on to the case when $K$ is only stably displaceable. Let $T^2$ be a symplectic torus such that $$\tilde{K}:=K\times \gamma\subset M\times T^2$$ is displaceable inside  $M\times T^2$, where $\gamma$ is a meridian in $T^2.$ Note that $M\times T^2$ also satisfies the conditions of our construction of relative symplectic cohomology over $\Nov$ as $T^2$ is aspherical.

We will prove that $SH^*_M(K;  \Nov)$ naturally injects into $SH^*_{M\times T^2}(\tilde{K};  \Nov)$, which finishes the proof. It is easy to see that acceleration data can be chosen for $\gamma\subset T^2$ where each Hamiltonian in the cofinal family has exactly $4$ contractible orbits, and the differentials on each of the corresponding Hamiltonian Floer groups vanish. Using the the chain level K\"{u}nneth isomorphism for Hamiltonian Floer theory and that completion commutes with tensor product with a finite dimensional $\Nov$-module, we easily prove the desired claim.
\end{proof}

We come to the third and final property of relative symplectic cohomology that we will discuss in this section. Recall from the introduction that a compact set $K\subset M$ is called \emph{$SH$-invisible} if $SH_M^*(K;\Nov)= 0.$

\begin{thm}\label{analog_thesis3}
If a compact subset $K\subset M$ is $SH$-invisible, then any compact subset $K'\subset K$ is also $SH$-invisible.
\end{thm}
\begin{proof}
The proof is identical to that of Theorem 1.2 (4) in \cite{Tonkonog2020}. The key point (Proposition 2.5 of \cite{Tonkonog2020}) is that there is a distinguished element $1_K\in SH_M(K,\Lambda)$, called the unit, with the following properties.
\begin{itemize}	
\item  $SH_M(K,\Lambda)=0$ if and only if $1_K=0$. 
\item Restriction maps send units to units.
\end{itemize} The element $1_K$ is constructed so that it is the unit of a pair-of-pants type product structure on $SH_M(K,\Lambda)$. The details are in Section 5 of \cite{Tonkonog2020}.
\end{proof}

\subsection{Towards the symplectic cohomology of the divisor complement}\label{subsec:SH}

We return to the geometric setup of Section \ref{subsec:setup}: $(M, \w)$ will be a closed symplectic manifold that is monotone
$$
	2\kappa c_1^M = [\w] \in H^{2}(M; \R) \quad\mbox{with}\quad \kappa > 0,
$$
$D = \cup_{i=1}^N D_i \subset (M, \w)$ will be a simple crossings divisor and $\lambda_1,\ldots,\lambda_n\in \mathbb{Q}_{>0}$ will be the weights. We will denote $X = M \setminus D$, $\bflambda \in H^2(M,X;\R)$ will be the associated lift of $2 c_1^M$, and $\theta \in \Omega^1(X)$ will be a primitive of $\omega|_X$ such that the relative de Rham cohomology class of $(\omega,\theta)$ is $\kappa\bflambda$. 


First we recall the action and index of orbits in the exact symplectic manifold $(X,\theta)$. 
Let $F: S^1 \times X \to \R$ be a Hamiltonian, and $\gamma:S^1 \to X$ a non-degenerate orbit of $F$.  
Its action is defined to be
\[ \cA_F(\gamma) := \int_{S^1} F(t, \gamma(t))\,dt + \int_{S^1} \gamma^*\theta.\]

In order to associate an index to orbits, we require an additional piece of data: a homotopy class of trivializations $\eta$ of $\Lambda^{top}_\C(TX)^{\otimes 2N}$, for some integer $N > 0$. 
To define the index $i_\eta(\gamma)$ of an orbit $\gamma$, we first choose a trivialization $\Phi$ of $\gamma^*TX$; we denote the Conley--Zehnder index with respect to this trivialization by $\CZ(\gamma,\Phi)$. 
The trivialization $\Phi$ induces a trivialization of $\Lambda^{top}_\C(\gamma^*TX)^{\otimes 2N}$, and we define $w(\Phi,\eta) \in \Z$ to be the winding number of
\[ \eta^{-1} \circ \Lambda^{top}_\C(\Phi)^{\otimes 2N}: S^1 \to \C^*.\]
We then define
\[ i_\eta(\gamma) = \CZ(\gamma,\Phi) + \frac{dim(X)}{2} - \frac{w(\Phi,\eta)}{N}. \]
One easily checks that the index is independent of the trivialization $\Phi$. 
Note that it is fractional: $i_\eta(\gamma) \in \frac{1}{N} \Z$. 

In our setting, the relevant choice of trivialization $\eta$ is determined by $\bflambda$. 
Let $N$ be an integer such that $N\lambda_i \in \Z$ for all $i$. 
Then $\sum_i N\lambda_i [D_i]$ is Poincar\'e dual to $c_1(\Lambda^{top}_\C(TM)^{\otimes 2N})$ by definition, so we may choose a section of $\Lambda^{top}_\C(TM)^{\otimes 2N}$ which is non-vanishing over $X$, and vanishes with multiplicity $N\lambda_i$ along $D_i$. Restricting this section to $X$ defines a homotopy class of trivializations of $\Lambda^{top}_\C(TX)^{\otimes 2N}$, which we denote by $\eta_{\bflambda}$. 
We will write $i(\gamma)$ for $i_{\eta_{\bflambda}}(\gamma)$.

Now let $F:S^1 \times M \to \R$ be a Hamiltonian, and $\gamma:S^1 \to X$ a non-degenerate orbit of $F$ which is contractible in $M$, and contained inside $X$. 
We define a canonical fractional cap $u_{in}$ for $\gamma$, by setting $u_{in}:= u - u\cdot \bflambda$ for an arbitrary cap $u$; the result is clearly independent of $u$. 
One should think of $u_{in}$ as a `cap inside $X$': indeed, if $u$ were a cap contained inside $X$, we would have $u_{in} = u$. 

\begin{lem}\label{lem:same_action_index}
We have
\[ i(\gamma) = i(\gamma,u_{in}) \qquad \text{and} \qquad \cA_F(\gamma) = \cA_F(\gamma,u_{in}).\]
\end{lem}
\begin{proof}
Let us choose an arbitrary $u: \mathbb{D}\to M$ capping $\gamma$. We start with the action. Directly from  the definitions: $$\cA_F(\gamma) = \int_{S^1} F(t, \gamma(t))\,dt + \int_{S^1} \gamma^*\theta \qquad \text{and} \qquad\cA_F(\gamma, u_{in}) = \int_{S^1} F(t, \gamma(t))\,dt + \int_\bD u^*\w-\kappa u\cdot \bflambda.$$ Therefore, the result follows from the assumption that the relative de Rham cohomology class of $(\omega,\theta)$ is $\kappa\bflambda$.

Recalling definitions for indices: $$i(\gamma) = \CZ(\gamma,\Phi) + \frac{\dim(X)}{2} - \frac{w(\Phi,\eta_{\bflambda})}{N},$$ where we choose $\Phi$ to be the trivialization of $\gamma^*TX$ induced by the cap $u$, and $$i(\gamma, u) = \CZ(\gamma, u) + \frac{\dim(M)}{2}-u\cdot\bflambda.$$ Therefore, we need to show that $$w(\Phi,\eta_{\bflambda})=N u\cdot\bflambda.$$ This follows because $\eta_{\bflambda}$ actually induces a section of $\Lambda^{top}_\C(u^*TM)^{\otimes 2N}.$ Using any trivialization of  $\Lambda^{top}_\C(u^*TM)^{\otimes 2N}$, we can think of this section as a map $\mathbb{D}\to \mathbb{C}$, which does not vanish along the boundary. The degree of this map at $0\in \mathbb{C}$ is easily computed to be $Nu\cdot\lambda$ using that $\eta_{\bflambda}$ vanishes with multiplicity $N\lambda_i$ along $D_i$. It is an elementary fact that the same degree is also equal to the winding number that we are interested in, so the result follows.
\end{proof}

\subsection{Positivity of intersection}

In this section we prove a result based on Abouzaid--Seidel's `integrated maximum principle'. 
We will later use it to prove Proposition \ref{prop-pos-int}, although the result is more broadly applicable.

Let $(W,\omega)$ be a symplectic manifold with a concave boundary modelled on the contact manifold $(Y,\theta)$. 
This means that $\partial W = Y$, and there is a symplectic embedding of the symplectization $(Y \times [c,c+\eps), d(\rho \cdot \theta))$ onto a neighbourhood of the boundary, where $\rho \in [c,c+\eps)$ is the Liouville coordinate. 
Note that as $\omega|_Y = cd\theta$, we have a relative de Rham cohomology class $[\omega;c\theta] \in H^2(W,Y)$. 
We will consider $u:(\Sigma,\partial \Sigma) \to (W,Y)$ satisfying the pseudoholomorphic curve equation for a certain class of almost-complex structures and Hamiltonian perturbations, and give a criterion guaranteeing that $[\omega;c\theta](u) \ge 0,$ with equality if and only if $u \subset Y$.

In order to define our pseudoholomorphic curve equation, we choose a complex structure $j$ on $\Sigma$, a family $\cJ$ of $\omega$-compatible almost-complex structures $J_z$ parametrized by $z \in \Sigma$, and a Hamiltonian-valued one-form $\mathcal{K} \in \Omega^1(\Sigma;C^\infty(W))$. 
Note that differential forms on $\Sigma \times W$ decompose into types: 
$$\Omega^\bullet(\Sigma \times W) = \bigoplus_{j+k = \bullet} \Omega^j(\Sigma,\Omega^k(W)),$$
so we may interpret $\cK$ as a one-form on $\Sigma \times W$. 
The de Rham differential decomposes as $d = d_\Sigma + d_W$, where 
$$d_\Sigma: \Omega^j(\Sigma,\Omega^k(W)) \to \Omega^{j+1}(\Sigma,\Omega^k(W)) \qquad \text{and}\qquad d_W:\Omega^j(\Sigma,\Omega^k(W)) \to \Omega^j(\Sigma,\Omega^{k+1}(W)).$$

The isomorphism $C^\infty(TW) \to \Omega^1(W)$ sending $v \mapsto \omega(v,-)$ allows us to turn $d_W \cK$ into a Hamiltonian-vector-field-valued one-form $X_\cK \in \Omega^1(\Sigma;C^\infty(TW))$. 
We will consider the pseudoholomorphic curve equation $$(du-X_\cK)^{0,1} = 0.$$ Note that the $(0,1)$-projection of $v\in \Omega^1(\Sigma;C^\infty(TW))$ is given by $\frac{1}{2}\left(v+\cJ\circ v \circ j\right)$.

We introduce the geometric energy of a pseudoholomorphic curve $u$:
$$E_{geom}(u) = \frac{1}{2}\int_\Sigma \left\| du - X_\cK \right\|^2.$$
It is manifestly non-negative. 
Let $\tilde{u}: \Sigma \to \Sigma \times W$ denote the graph of $u$.
We have the standard computation (e.g. Equation (8.12) of \cite{Seidel}):
$$E_{geom}(u) = \int_\Sigma u^* \omega + \tilde{u}^*\left( d_W \cK + \{\cK,\cK\}\right),$$
where the final term lives in $\Omega^2(\Sigma,C^\infty(W))$ and is defined by $\{\cK,\cK\}(v,w) := \{\cK(v),\cK(w)\}$, where $\{-,-\}$ is the Poisson bracket. 

We also introduce the topological energy
$$E_{top}(u) := \int_\Sigma u^*\omega + \tilde{u}^* d \cK. $$
Note that
$$E_{top}(u) = E_{geom}(u) + \int_\Sigma \tilde{u}^* \left(d_\Sigma \cK -\{\cK,\cK\}\right) .$$

\begin{prop}\label{prop-pre-pos-int}
Suppose that 
\begin{enumerate}
\item \label{it:cont-typ} $J_z$ is of contact type along $Y$, for all $z \in \partial \Sigma$: $$d\rho\circ J_z=-\rho\theta.$$
\item \label{it:neck-cond}There exist one-forms $\alpha, \beta \in \Omega^1(\Sigma)$ such that $\cK = \alpha \cdot \rho + \beta$ in a neighbourhood of $Y$
\item \label{it:pos-en} We have $d_\Sigma \cK - \{\cK,\cK\} - d\beta \ge 0$.\footnote{Given  $\xi \in \Omega^2(\Sigma,C^\infty(W))$, we say that $\xi \ge 0$ if for all $z \in \Sigma$, $v \in T_z\Sigma$, and $w \in W$, we have $\xi(v,jv)(w) \ge 0$.}
\end{enumerate}
Then any smooth map $u: (\Sigma,\partial \Sigma) \to (W,Y)$ satisfying $(du - X_\cK)^{0,1} = 0$,  with $\partial \Sigma \neq \emptyset$, will satisfy $[\omega;c\theta](u) \ge 0$, with equality if and only if $u \subset Y$.
\end{prop}
\begin{proof}
We have
\begin{align*} 
[\omega;c\theta](u) &=\int_\Sigma u^*\omega - \int_{\partial \Sigma} u^*c\theta \\
&= E_{geom}(u) - \int_{\Sigma} \tilde{u}^* \left(d_W\cK +\{\cK,\cK\}\right) - c\int_{\partial \Sigma} u^* \theta \\
&\ge \int_\Sigma \tilde{u}^* \left(-d\cK + d_\Sigma \cK - \{\cK,\cK\}\right) - c\int_{\partial \Sigma} u^* \theta \qquad\text{ as $E_{geom}(u) \ge 0$}\\
&\ge \int_\Sigma \tilde{u}^* \left(-d\cK + d\beta\right) - c\int_{\partial \Sigma} u^* \theta \qquad \text{ by hypothesis \eqref{it:pos-en}}\\ 
&= \int_{\partial \Sigma} -\tilde{u}^* \cK + \beta -c\cdot u^* \theta.
\end{align*}
By hypothesis \eqref{it:neck-cond}, the first and second terms combine to give
$$\int_{\partial \Sigma} -\tilde{u}^* (\alpha \cdot \rho + \beta) + \beta = - \int_{\partial \Sigma} c \cdot \alpha,$$
as $\rho = c$ along $Y$.

We can analyse the remaining term using the argument in \cite[Lemma 7.2]{Abouzaid2007}. Let $v \in T_z \partial\Sigma$ be a positively-oriented boundary vector. Using the Floer equation $$(du-X_\cK)^{0,1}=0,$$ we obtain $$u_*(v)=-Ju_*j(v)+X_\cK\left(v\right)+JX_\cK j(v),$$
so
$$u^*\theta(v) = -\theta\left(Ju_*j(v)\right)+\theta\left(X_\cK\left(v\right)\right)+\theta\left(JX_\cK j(v)\right).$$

We analyse each term on the RHS. For the first, we note that $j(v)$ points into $\Sigma$. Therefore $u_*(j(v))$ points into $W$. Such vectors can be written as the sum of a non-negative multiple of the Liouville vector and a vector that is tangent to $Y$. Because $J$ is of contact type, this implies that $$\theta\left(Ju_*j(v)\right)\ge 0.$$
For the second, we note that hypothesis \eqref{it:neck-cond} ensures that $X_\cK(v) = -\alpha(v) \cdot \mathcal{R}$, where $\mathcal{R}$ is the Reeb vector field on $Y$. 
Thus $\theta(X_\cK(v)) = -\alpha(v)$.
For the third, hypothesis \eqref{it:neck-cond} again ensures that $X_\cK(j(v))$ is a multiple of the Reeb vector field; because $J$ is of contact type, $\theta(JX_\cK j(v)) = 0$.
Putting it all together, we have
$$u^*\theta(v) \le -\alpha(v).$$

Combining, we finally obtain 
$$[\omega;c\theta](u) \ge \int_{\partial \Sigma} -c \cdot \alpha +c \cdot \alpha = 0$$
as required. 

If equality holds then we have $E_{geom}(u) = 0$, which implies that $du = X_\cK$. 
Hypothesis \eqref{it:neck-cond} then implies that $u_*(v) = X_\cK(v)$ is a multiple of the Reeb vector field $\mathcal{R}$ in a neighbourhood of $Y$, for all $v$; as $\mathcal{R}$ is tangent to $Y$, this implies that $u$ is contained in $Y$.
\end{proof}

\begin{rmk}
Note that if $\cK' = \cK+\xi$ where $\xi \in \Omega^1(\Sigma)$, then $X_\cK = X_{\cK'}$, so the associated pseudoholomorphic curve equations are identical. Thus we would expect that if the hypotheses of Proposition \ref{prop-pre-pos-int} hold for $\cK$, then they should also hold for $\cK'$. 
Indeed, Hypothesis \eqref{it:neck-cond} holds, as $\cK' = \alpha \cdot \rho + \beta'$, where $\beta' = \beta + \xi$; and Hypothesis \eqref{it:pos-en} also holds, because $\cK' - \beta' = \cK - \beta$.
\end{rmk}

Proposition \ref{prop-pre-pos-int} is designed to prove Proposition \ref{prop-pos-int} (= Proposition \ref{prop-pos-inta}), but there are other natural situations where Hypotheses \eqref{it:neck-cond} and \eqref{it:pos-en} can be made to hold. 
The simplest, of course, is if $\cK$ vanishes in a neighbourhood of $Y$. 
Alternatively, similarly to \cite{Abouzaid2007}, we may have $\cK = H \cdot \gamma$ where $H$ is independent of $z \in \Sigma$, $H = a \rho + b$ in a neighbourhood of $Y$, $H \ge b$ over $W$, and $d \gamma \ge 0$.


\section{Special Hamiltonian}\label{s-H}

Our goal in this section is to construct the special functions $\rho^R: M \to \R$, defined for $R>0$ sufficiently small, as mentioned Section \ref{ss-proofs}. 
Recall their key properties: 
\begin{itemize}
\item $\rho^R$ is continuous on $M$, and smooth on the complement of the skeleton $\mathbb{L}$; 
\item $\rho^R|_{\mathbb{L}} = 0$ and $\rho^R|_D \approx 1$;
\item we have $Z(\rho^R) = \rho^R$ on $X \setminus \mathbb{L}$, where $Z$ is the Liouville vector field on $(X,\theta)$;
\item $\rho^R \to \rho^0$ as $R \to 0$.
\end{itemize}
Having constructed the functions $\rho^R$, we use them to construct the Hamiltonians on $M$ which we use in our main arguments; and we compute the action and index of the orbits of these Hamiltonians. 
The results are expressed in Lemmas \ref{lem-act-final} and \ref{lem-ind-final}.

We use the geometric setup of Section \ref{subsec:setup} with slight modifications in light of Section \ref{s-basics-of-SC}. Let us spell this out fully. We have  a closed symplectic manifold $(M, \w)$ which is monotone
$$
	2\kappa c_1^M = [\w] \in H^{2}(M; \R) \quad\mbox{with}\quad \kappa > 0,
$$
$D = \cup_{i=1}^N D_i \subset (M, \w)$ is an orthogonal simple crossings divisor and $\lambda_1,\ldots,\lambda_N\in \mathbb{Q}_{>0}$ is a choice of weights. We denote $X = M \setminus D$ and $\bflambda \in \R^N \cong H^2(M,X;\R)$ is the associated lift of $2 c_1^M$. We also choose an admissible system of commuting Hamiltonians $\{r_i:UD_i \to [0,R_0)\}$ near $D$ and  a primitive $\theta \in \Omega^1(X)$ of $\omega|_X$ such that the relative de Rham cohomology class of $(\omega,\theta)$ is $\kappa\bflambda$. 
We assume that $\theta$ is adapted to $\{r_i:UD_i \to [0,R_0)\}$ and that \begin{equation}\label{eq-small-rad} R_0<\kappa\lambda_i,\text{ for all } i.\end{equation} The last condition can be achieved by shrinking the ascH  (as explained in Section \ref{ss:scH}).

In fact we will consider the $(0, R_0)$-family of such data obtained by shrinking the ascH to radius $R\in (0,R_0)$, while keeping all else fixed. The parameter $R$ will also\footnote{We note that this is for notational convenience only.} be used as the `smoothing parameter' for $\rho^R$. In Section \ref{s-proofs}, we will want $R$ to be sufficiently small for certain arguments to work. The approximations in this section  (such as $\rho^R|_D \approx 1$) will be more and more accurate as $R$ tends to $0$. The dependence on $R$ of our constructions below should be understood in this light.


\subsection{Overview of the construction of $\rho^R$}\label{s:HamOverview}

The first step in the construction is to enlarge the sets $UD_i$ via the Liouville flow. 
This gives us open sets $UD_i^{max}$, together with toric moment maps $r_i^{max}: UD^{max}_i \to [0,\kappa\lambda_i)$, such that $\cup_i UD_i^{max} = M \setminus \mathbb{L}$. 
For $I \subset [N]$ we define $UD_I^{max} = \cap_{i \in I} UD_i^{max}$, and we have toric moment maps $r_I^{max} : UD_I^{max} \to \prod_{i \in I} [0,\kappa\lambda_i)$.

Now for each non-empty $I \subset [N]$, we define open subsets $\mathring{U}D^{max}_I \subset UD^{max}_I$ so that $\cup_I \mathring{U}D^{max}_I = M \setminus \mathbb{L}$. 
We will define $\rho^R|_{\mathring{U} D^{max}_I} = \tilde{\rho}^R_I \circ r^{max}_I$, for smooth functions $$\tilde{\rho}^R_I: \prod_{i \in I} [0,\kappa\lambda_i) \to \R$$ carefully chosen so that the definition agrees on the overlaps and $\rho^R$ satisfies the desired key properties. In fact, $\tilde{\rho}^R_I$ will be well defined on the larger region
$$V_I := \R^I \setminus \prod_{i \in I} [\kappa \lambda_i,\infty).$$

Let us briefly discuss how we will ensure that $\rho^R$ thus defined satisfies $Z(\rho^R) = \rho^R$. 
We translate this into a property of the functions $\tilde{\rho}^R_I$. 
We denote the standard projection by $\mathrm{pr}_I:\R^N \to \R^I$, and set $\bflambda_I := \mathrm{pr}_I(\bflambda)$. 
We consider the (Euler-type) vector field $\tilde{Z}_I$ on $\R^I$ defined by
$$\left(\tilde{Z}_I\right)\!_r : = \sum_{i \in I} (r_i - \kappa \lambda_i) \frac{\partial}{\partial r_i}.$$

\begin{lem}\label{lem:LiouvV}
For all $x \in UD_I \setminus D$,
	\[ (r_I)_* Z_x = \left(\tilde{Z}_I\right)_{r_I(x)}.\]
\end{lem}
\begin{proof}
Follows from the fact that $Z(r_i) = r_i - \kappa\lambda_i$, as $\theta$ is adapted to the scH.
\end{proof}

In fact, $UD_I^{max}$ and $r_I^{max}$ are constructed so that Lemma \ref{lem:LiouvV} also holds if we put $max$ superscripts on the $r_I$ and $UD_I$ (Lemma \ref{lem:maxversion}). 
This gives us:

\begin{cor}\label{cor:rhoLiouv}
The function $\rho^R:= \tilde{\rho}^R_I \circ r^{max}_I$ satisfies $Z(\rho^R) = \rho^R$ if and only if $\tilde{Z}_I(\tilde{\rho}^R_I) = \tilde{\rho}^R_I$.
\end{cor}

Note that a function $f:V_I \to \R$ satisfies $\tilde{Z}_I(f) = f$ if and only if it is linear along the rays emanating from $\kappa\bflambda_I$, converging to $0$ at that point.

The functions $\tilde{\rho}^R_I$ will be constructed roughly  as follows. We will choose a hypersurface $\tilde{Y}^R_I\subset V_I\cap \R^I_{\geq 0}$ which is a smoothing of $\tilde{Y}^0_I := \partial \R^I_{\ge 0}$, satisfying certain properties (see Lemma \ref{lem:construct-Y}). Then, we will define $\tilde{\rho}^R_I$ as the function that is linear along the rays emanating from $\kappa\bflambda_I$, converging to zero at that point, and takes the value $1$ on $\tilde{Y}^R_I$.

For the other key properties of $\rho^R$ let us mention the following slightly sketchy point to orient the reader. Recall that $\rho^R$ is supposed to be a smoothing of the continuous function $\rho^0: M \to \R$ introduced in Section \ref{ss-rigid-sk}, which has all of the properties we need (e.g., it satisfies $\rho^0|_{\mathbb{L}} = 0$, $\rho^0|_{D} = 1$ and $Z(\rho^0) = \rho^0$), except it is not smooth. 
We now give an alternative description of the function $\rho^0$, which is parallel with the construction of $\rho^R$. 
We extend the function $\frac{\kappa\lambda_i - r_i^{max}}{\kappa\lambda_i}: UD_i^{max}\to \R$ to $M$ by defining it to be $0$ everywhere outside of its original domain of definition. Let us momentarily denote this extension with the same notation. Then we have $$\rho^0 = \max_i \frac{\kappa\lambda_i - r_i^{max}}{\kappa\lambda_i}.$$ 
In particular, on $UD_I^{max}$, we have $\rho^0 = \tilde{\rho}^0_I \circ r_I^{max}$, where $\tilde{\rho}^0_I(r) = \max_{i \in I} \frac{\kappa\lambda_i - r_i}{\kappa\lambda_i}$. 
Note that $\tilde{\rho}^0_I$ is equal to 0 at $\kappa\bflambda_I$, linear along the rays emanating from this point, and equal to $1$ along $\tilde{Y}^0_I$. The functions $\tilde{\rho}_I^R$ mentioned above will be consistently-chosen smoothings of the functions $\tilde{\rho}^0_I$.

\begin{rmk}
We would like to warn the reader of an abuse of notation we already committed a couple of times above and will continue with below. We will use $r_i$ both as the function $r_i: UD_i\to [0,R)$ and also the $i^{th}$ coordinate function on $\mathbb{R}^I$ with $i\in I.$ We believe that this will not cause too much confusion, partly because often we will actually need to be using $r_i^{max}: UD_i^{max}\to [0,\kappa\lambda_i)$  in place of the former anyway.
\end{rmk}

\subsection{Construction of $UD_I^{max}$, $r_I^{max}$, $\mathring{U}D_I^{max}$}\label{s:UDImax}
 
Note that $UD_i \setminus D_i$ is closed under the positive Liouville flow as long as the flow is defined, by Lemma \ref{lem:LiouvV} and Equation \eqref{eq-small-rad}. Let us define $UD_i^{max}\subset M$ as the union of $UD_i$ with the set of points in $X$ that enter into $UD_i$ under the positive Liouville flow in finite time. Of course we have $UD_i\subset UD_i^{max}$. Note that $UD_i^{max}$ depends on $R$ just as $UD_i$ does (unless $D=D_i$ is smooth); nevertheless we suppress $R$ from the notation.

We  extend $r_i$ to $$r_i^{max}: UD_i^{max}\to \mathbb{R}_{\geq 0}$$ by first flowing into $UD_i$ with the Liouville flow in some time $T\geq 0$, applying $r_i$, and then flowing with $\tilde{Z}_{\{i\}}$ for time $-T$. This is well-defined and smooth by Lemma \ref{lem:LiouvV}. 

Let us also define $UD_I^{max} := \cap_{i \in I} UD_i^{max}$ and $$r_I^{max}:UD_I^{max} \to \R_{\ge 0}^I.$$
Then the following is true by construction:

\begin{lem}\label{lem:maxversion}
 Lemma \ref{lem:LiouvV} holds if we put $max$ superscripts on the $r_I$ and $UD_I$.
\end{lem}

Now, recall that $\tilde{Y}^0_I := \partial \R_{\ge 0}^I$. 
Define the projection-from-$\kappa\bflambda_I$ map
$$P_I: V_I \to \tilde{Y}^0_I,$$
which flows a point along $\tilde{Z}_I$ until it intersects $\tilde{Y}^0_I$. 

Now, let us define $UD_i^{1/2} := \{r_i \le R/2\} \subset UD_i = \{r_i < R\}$. 
Define $UD_i^{1/2,max}$ to be the union of $UD_i^{1/2}$ with the set of points in $X$ that enter into $UD^{1/2}_i$ under the positive Liouville flow in finite time. 

\begin{defn}
For $I \subset [N]$, define 
$$\mathring{U}D_I^{max} := UD_I^{max} \setminus \bigcup_{j \notin I} UD_j^{1/2,max}.$$
\end{defn}

(Figure \ref{fig:Y} may help the reader visualize these sets.)

\begin{lem}\label{lem:Uringcover}
The sets $\left\{\mathring{U}D_I^{max}\right\}_{\emptyset \neq I \subset [N]}$ form an open cover of $M \setminus \mathbb{L}$.
\end{lem}
\begin{proof}
Because $UD_i^{1/2,max}$ is closed in $M \setminus \mathbb{L}$, and contained in $UD_i^{max}$, the sets $UD_i^{max}$ and $\left(UD_i^{1/2,max}\right)^c$ form an open cover of $M \setminus \mathbb{L}$ for all $i$. 
Taking the intersection of these open covers over all $i$ gives us an open cover by the sets
$$ \bigcap_{i \in I} UD_i^{max} \cap \bigcap_{ i \notin I} \left(UD_i^{1/2,max}\right)^c = \mathring{U}D_I^{max}$$
for $I \subset [N]$. It remains to check that $\mathring{U}D_{\emptyset}^{max} = \emptyset$. 
This follows from the fact that $\cup_i UD_i^{1/2,max} = M \setminus \mathbb{L}$ (because every flowline of the Liouville vector field in $X \setminus \mathbb{L}$ ultimately enters $\cup_i UD_i^{1/2}$).
\end{proof}

The following is an easy consequence of Lemma \ref{lem:LiouvV} and the construction of $UD_I^{max}$ and $UD_i^{1/2,max}$:

\begin{lem}\label{lem:Umaxhalfim}
If $i \in I$, then
$$ UD_I^{max} \setminus UD_i^{1/2,max} = \left(P_I \circ r_I^{max}\right)^{-1}\left(\{r_i > R/2\}\right).$$
\end{lem}


\subsection{Construction of $\tilde{Y}^R_I$}\label{ss:tildeY}

\begin{figure}
\centering
\includegraphics[width=0.5\linewidth]{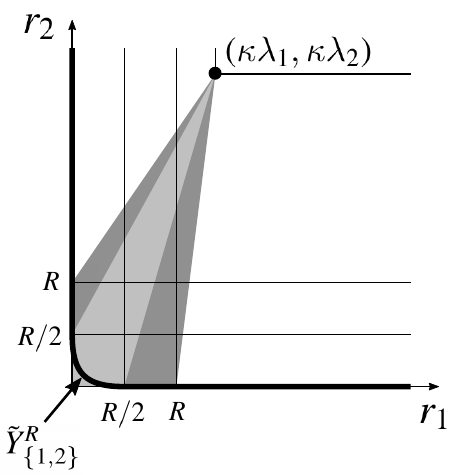}
\caption{The hypersurface $\tilde{Y}^R_{\{1,2\}}$. The image of $r_{\{1,2\}}^{max}$ is shaded. 
The images of the regions $\mathring{U}D_{\{1\}}^{max}\cap  \mathring{U}D_{\{1,2\}}^{max}$ and $\mathring{U}D_{\{2\}}^{max} \cap \mathring{U}D_{\{1,2\}}^{max}$ are shaded darker.}
\label{fig:Y}
\end{figure}

For any $R_0>R>0$ (as always in this section), let $q^R:\R \to \R$ be a function satisfying:
\begin{itemize}
\item $q^R(r) = 0$ for $r \ge R/2$;
\item $(q^R)'(r)<0$ for $r < R/2$;
\item $q^R(0) = 1$.
\end{itemize}
Consider 
\begin{align*}
Q^R_I: \R^I &\to \R\\
Q^R_I(r) &:= \sum_{i \in I} q^R(r_i).
\end{align*}

Now define $$\tilde{Y}^R_I := \{Q^R_I = 1\}$$
(see Figure \ref{fig:Y}).

\begin{lem}\label{lem:construct-Y}
The hypersurfaces $\tilde{Y}^R_I \subset \R^I$ have the following properties:
\begin{enumerate}
\item \label{it:Y-V} $\tilde{Y}^R_I$ is contained in the region $V_{I,\ge 0} := V_I \cap \R_{\ge 0}^I$.
\item \label{it:Y-Z} Every flowline of $\tilde{Z}_I$ in $V_I$ crosses $\tilde{Y}^R_I$ transversely at a unique point. 
\item \label{it:Y-patch1} If $\tilde{\nu}^R_I:\tilde{Y}^R_I \to \R^I$ is a normal vector field (pointing towards the component containing $\kappa\bflambda$), then $\tilde{\nu}^R_{I,i} \ge 0$ for all $i$. (Here $\tilde{\nu}^R_{I,i}$ is the $i$th component of $\tilde{\nu}^R_I$.)
\item \label{it:Y-patch2} For any $J \subset I$, $\tilde{Y}^R_I$ coincides with $\tilde{Y}^R_J \times \R^{I \setminus J}$ over the region $\cap_{i \in I \setminus J} \{\tilde{\nu}^R_{I,i} = 0\}$. 
\item \label{it:Y-patch3} The region $\{\tilde{\nu}^R_{I,i} = 0\}$ contains $P_I^{-1}\left(\{r_i>R/2\}\right)$.
\end{enumerate}
\end{lem}
\begin{proof}
Property \eqref{it:Y-V} follows from the fact that $Q^R_I \ge 1$ if any $r_i \le 0$ and $Q^R_I = 0$ if all $r_i\geq R/2$.

To prove property \eqref{it:Y-Z}, we first observe any flowline of $\tilde{Z}_I$ in $V_I$ starts at $\kappa\bflambda_I$, where $Q^R_I = 0$, and ends up outside $V_{I,\ge 0}$, where $Q^R_I \ge 1$, so it must cross $\tilde{Y}^R_I$ somewhere. 
Furthermore, we have that $\tilde{Z}_I(Q^R_I) \ge 0$ for any $i$: we have $$\tilde{Z}_I(q^R(r_i)) = (r_i-\kappa\lambda_i) \cdot (q^R)'(r_i),$$ 
where $(q^R)'(r_i) \le 0$, and $r_i - \kappa\lambda_i < R/2 - \kappa\lambda_i <0$ wherever $(q^R)'(r_i) \neq 0$. 
Finally, we have $\tilde{Z}_I(Q^R_I) > 0$ along $\tilde{Y}^R_I$, because at any point on $\tilde{Y}^R_I$ we have $q^R(r_i) > 0$ and hence $(q^R)'(r_i)<0$ for some $i$.

Property \eqref{it:Y-patch1} follows from the fact that $\partial Q^R_I/\partial r_i \le 0$ for all $i$.
Property \eqref{it:Y-patch2} follows from the fact that $(q^R)'(r_i) = 0$ if and only if $q^R(r_i) = 0$. 
Property \eqref{it:Y-patch3} follows from the fact that $$P_I^{-1}(\{r_i > R/2\}) \cap \tilde{Y}^R_I \quad  \subset  \quad P_I^{-1}(\{r_i > R/2\})  \cap V_{I,\ge 0} \quad \subset \quad \{r_i>R/2\},$$ and $(q^R)'(r_i) = 0$ for $r_i > R/2$. 
\end{proof}

\begin{rmk}
The hypersurface $\tilde{Y}^R_I$ has the additional property (which we will not use, but which may help the reader to visualize the construction) that it coincides with $\tilde{Y}^0_I$ away from a neighbourhood of the singular locus of the latter. We can also choose $q^a$ to be a convex function, which would imply that the component of $\R^I-\tilde{Y}^R_I$ that does not contain $0$ is convex (which would in turn imply that $\tilde{\rho}^R_I$ is convex). Again, we do not need this property.
\end{rmk}

\subsection{Construction of $\tilde{\rho}^R_I$}\label{ss:BasicRho}

By property \eqref{it:Y-Z} of $\tilde{Y}^R_I$, there is a unique smooth function $\tilde{\rho}^R_I:V_I \to \R$ satisfying
$$ \tilde{\rho}^R_I|_{\tilde{Y}^R_I} = 1 \qquad \text{and}\qquad \tilde{Z}_I(\tilde{\rho}^R_I) = \tilde{\rho}^R_I.$$
Recall that the second condition means that $\tilde{\rho}^R_I$ is linear along the rays emanating from $\kappa\bflambda_I$, converging to zero at $\kappa\bflambda_I$. 
In particular, the level sets of $\tilde{\rho}^R_I$ are scalings of $\tilde{Y}^R_I$ centred at $\kappa \bflambda_I$.

\begin{lem}\label{lem:rhoI-patch}
If $J \subset I$, then
$$\tilde{\rho}^R_I = \tilde{\rho}^R_J \circ \mathrm{pr}_{IJ} \qquad \text{over the region }\qquad \bigcap_{i \in I \setminus J} P_I^{-1}\left(\{r_i > R/2\}\right),$$
where $\mathrm{pr}_{IJ}:\R^I \to \R^J$ is the natural projection.
\end{lem}
\begin{proof}
Follows from the fact that $\tilde{Y}^R_I$ coincides with $\tilde{Y}^R_J \times \R^{I \setminus J}$ in the given region, by properties \eqref{it:Y-patch2} and \eqref{it:Y-patch3} of $\tilde{Y}^R_I$.
\end{proof}

\subsection{Construction of $\rho^R$}

\begin{lem}\label{lem:rho-patch}
For any $\emptyset \neq I,J \subset [N]$, we have
$$ \tilde{\rho}^R_I \circ r_I^{max} = \tilde{\rho}^R_J \circ r_J^{max} \qquad \text{over} \qquad \mathring{U}D_I^{max} \cap \mathring{U}D_J^{max}.$$
\end{lem}
\begin{proof}
First note that $\mathring{U}D_I^{max} \cap \mathring{U}D_J^{max} \subset UD_{I \cup J}^{max}$. 
We have
$$ \mathring{U}D_I^{max} \cap \mathring{U}D_J^{max} = \bigcap_{k \notin I \cap J} \left(P_{I \cup J} \circ r_{I \cup J}^{max}\right)^{-1} \left(\{r_k > R/2\}\right),$$
as an immediate consequence of Lemma \ref{lem:Umaxhalfim}. 
Over this set, we have
$$ \tilde{\rho}^R_I \circ r_I^{max} = \tilde{\rho}^R_I \circ \mathrm{pr}_{I \cup J,I} \circ r_{I \cup J}^{max} = \tilde{\rho}^R_{I \cup J} \circ r_{I \cup J}^{max} $$
by Lemma \ref{lem:rhoI-patch}. The result now follows by applying the same argument to $\tilde{\rho}^R_J \circ r_J^{max}$.
\end{proof}

Lemma \ref{lem:Uringcover} and Lemma \ref{lem:rho-patch} allow us to define:

\begin{defn}
We define $\rho^R:M \to \R$ to be equal to $\tilde{\rho}^R_I \circ r_I^{max}$ over each $\mathring{U}D_I^{max}$, and equal to $0$ over $\mathbb{L}$.
\end{defn}

To check that $\rho^R$ is continuous along $\mathbb{L}$, we use the fact that $Z(\rho^R) = \rho^R$ on $M \setminus \mathbb{L}$ by Corollary \ref{cor:rhoLiouv}, and the level sets of $\rho^R$ are compact submanifolds disjoint from $\mathbb{L}$.
It follows that $\rho^R \to 0$ as we go towards $\mathbb{L}$, so $\rho^R$ is continuous along $\mathbb{L}$.

\begin{defn}\label{def:KRs}
Because $Z(\rho^R) = \rho^R$, and $\rho^R|_D \ge 1$, the subset $K^R_\s := \{\rho^R \le \s\}$ is a Liouville subdomain of $X$ for any $\s \in (0,1)$. 
The contact manifold $Y^R_\s = \partial K^R_\s$ with contact form $\sigma^{-1}\iota_{Y^R_\s}^*\theta$ is independent of $\s$.
\end{defn}

For the remainder of this section, we will drop $R$ from the notation: so we write $\rho$ instead of $\rho^R$, etc.

\subsection{The Hamiltonian and its orbits}\label{ss-Horbs}

Let $h:\R \to \R$ be a smooth function which is constant on a neighbourhood of $0$. 
It is clear that the function $h \circ \rho$ is smooth on $M$. 
We denote its Hamiltonian flow by $\Phi^{h \circ \rho}_t$.
In order to describe the orbits of $h \circ \rho$, we first compute $d\rho$. 

\begin{lem}\label{lem-nu-exist}
There exist smooth functions $\nu_i:M \setminus \mathbb{L} \to \R_{\ge 0}$, supported in $UD_i^{max}$, such that
$$ (d\rho)_m = -\sum_i \nu_i(m) \cdot \left(dr_i^{max}\right)_m. $$
(Here the LHS denotes the value of the one-form $d\rho$ at the point $m$. The RHS is well-defined, even though $dr_i^{max}$ is only defined over $UD_i^{max}$, because $\nu_i$ vanishes outside $UD_i^{max}$.)
\end{lem}
\begin{proof}
For any $i,I$, we define the following function on $V_I$:
$$\tilde{\nu}_{I,i}:=\left\{ \begin{array}{ll}
												 -\partial \tilde{\rho}_I/r_i & \text{if $i \in I$} \\
												 0 & \text{else}
									 \end{array} \right.$$
Note that it is non-negative by property \eqref{it:Y-patch1} of $\tilde{Y}_I$. 
We claim that for any $i$, and any $J \subset I$, we have 
$$\tilde{\nu}_{I,i} = \tilde{\nu}_{J,i} \circ \mathrm{pr}_{IJ} \qquad \text{over the region }\qquad \bigcap_{i \in I \setminus J} P_I^{-1}\left(\{r_i > R/2\}\right).$$
If $i \in I$, this follows by Lemma \ref{lem:rhoI-patch} (there are two cases: $i \in J$ and $i \in I \setminus J$). 
If $i \notin I$, it is obvious as both functions are $0$. 
This allows us to mimic the construction of $\rho$: we set $\nu_i = \tilde{\nu}_{I,i} \circ r_I^{max}$ over $\mathring{U}D_I^{max}$. 
We finally observe that $d\tilde{\rho}_I = -\sum_i \tilde{\nu}_{I,i} dr_i$, which completes the proof.
\end{proof}

For any $m \in M-\mathbb{L}$, we define $I(m) := \{i: \nu_i(m) \neq 0\}$. 
We have $m \in UD_{I(m)}^{max}$. 

We define $\nu: M \setminus \mathbb{L} \to \R^N$ to be the smooth function with coordinates $(\nu_1,\ldots,\nu_N)$. 
We note that the function $h'(\rho) \cdot \nu: M \setminus \mathbb{L} \to \R^N$ extends smoothly to $M$, and we denote this extension by $\nu^h: M \to \R^N$. 
Note that $\nu^h$ is constant along orbits of $h \circ \rho$, so we have a well-defined $\nu^h(\gamma) \in \R^N$ associated to such an orbit $\gamma$. 
We can interpret $\nu^h_i(\gamma)$ as `the number of times $\gamma$ wraps around $D_i$' (it is an integer unless $\gamma$ is contained in $D_i$, see Lemma \ref{lem:HamflowH} below). 
We define  $$I(\gamma):= \{i:\nu^h_i(\gamma) \neq 0\} \subset [N].$$ 
Note that if $h'(\rho) \ge 0$, then $\nu^h(\gamma) \in \R_{\ge 0}^N$.

\begin{cor}\label{cor:PhiHact}
For any $m \in M \setminus \mathbb{L}$, we have $\Phi^{h \circ \rho}_1(m) = \nu^h(m) \cdot m$. 
To explain the notation, $\nu^h(m) \in \R^{I(m)}$ gets projected to $(\R/\Z)^{I(m)}$, which then acts on $m \in UD_{I(m)}^{max}$ by the Hamiltonian torus action.
\end{cor}

\begin{lem}\label{lem:HamflowH}
We have $\Phi^{h \circ \rho}_1(m) = m$ if and only if for all $i$, either $m \in D_i$ or $\nu^h_i(m) \in \Z$. 
\end{lem}
\begin{proof}
For $m \in \mathbb{L}$, the claim is obvious, as $h \circ \rho$ is constant and $\nu^h$ vanishes. 
For $m \notin \mathbb{L}$, the claim follows from Corollary \ref{cor:PhiHact}.
\end{proof}

\subsection{Perturbing to achieve nondegeneracy}\label{ss-pert}

Now let us suppose that for some $\eps>0$, we have that
\begin{itemize}
\item $h(\rho)$ is constant for $\rho \le \eps$;
\item $h(\rho)$ is linear for $\rho \ge 1-\eps$;
\item On any interval on which $h(\rho)$ is linear, except $(-\infty,\eps]$, the slope is not a Reeb period of $Y$.
\end{itemize} 
Then the orbits of $h \circ \rho$ come in families parametrized by manifolds with corners. 

The families are indexed by a set 
$$P = \coprod_{I \subset [N]} P_I,$$
where $P_I$ consists of families of orbits $\gamma$ with $I(\gamma) = I$. 
The two cases $I=\varnothing,I\neq\varnothing$ must be treated differently. 
In order to describe $P_\varnothing$, let us suppose that $\eps'$ is maximal so that $h$ is constant on $(-\infty,\eps']$. Then
$$P_\varnothing = \{0\} \cup \{\rho>\eps': h'(\rho) = 0\}.$$
Associated to $p \in P_\varnothing$ is a set of constant orbits $C_p$, which can be identified with a subset of $M$:
$$C_0 = \{\rho \le \eps'\},\qquad C_p = \{\rho=p\}\quad \text{for $p \in P_\varnothing \setminus \{0\}$}.$$
On the other hand, for $I \neq \varnothing$ we have
$$P_I = \left\{p \in \im(r_I^{max}):\text{ for each $i \in I$ we either have $p_i = 0$ or $h'(\tilde{\rho}_{I}(p)) \cdot \tilde{\nu}_{I,i}(p) \in \Z\setminus\{0\}$}\right\}.
$$

Associated to each $p\in P_I$, we define a subset of $M$:
$$C_p := \{m \in UD^{max}_{I}: r_{I}^{max}(m) = p, \nu_k(m)=0 \text{ for $k \notin I$}\}.$$
For each $p \in P$, $C_p$ is a manifold-with-corners on which the flow of $h \circ \rho$ is $1$-periodic, yielding a manifold-with-corners of orbits which is diffeomorphic to $C_p$.

We now perturb $h \circ \rho$, in such a way as to make the orbits nondegenerate.

\begin{lem}\label{lem:pertH}
Given $\eps>0$, there exists a perturbation $H$ of $h \circ \rho$ with nondegenerate orbits, such that for any capped orbit $(\gamma,u)$ of $H$, there exists a capped orbit $(\bar{\gamma},\bar{u})$ of $h \circ \rho$, such that:
\begin{enumerate}
\item \label{it:act-approx}$|\cA(\gamma,u) - \cA(\bar{\gamma},\bar{u})|<\eps$;
\item \label{it:ind-approx}$|\CZ(\gamma,u) - \CZ(\bar{\gamma},\bar{u})| \le k(\bar{\gamma})/2$, where $\CZ$ denotes the Conley--Zehnder index,\footnote{The definition is due to Robbin--Salamon in the case of the possibly-degenerate orbit $\gamma$.} and $k(\bar{\gamma}) := \dim \ker \left(D\Phi^1_{h \circ \rho}  - \id\right)_{\bar{\gamma}(0)}$.
\end{enumerate}
\end{lem}
\begin{proof}
Note that the subsets $C_p\subset M$, $p\in P$ are closed, disjoint, preserved by the flow of $h \circ \rho$, and the flow is one-periodic on them. 
We will choose disjoint neighbourhoods $N_p$ of $C_p$, and perturb in each $N_p$ separately: i.e., $H = h \circ \rho+\delta$ where $\delta = \sum_p \delta_p$ with $\delta_p$ supported in $N_p$. 

We fix a Riemannian metric on $M$ for the duration of this proof. 
In particular, whenever we say that a function is `$C^k$-small', we mean with respect to this metric.

Note that $d(\Phi_1^{h \circ \rho}(m), m) > \eta$ for some $\eta>0$ over the compact set $M \setminus \cup_p N_p$. 
By making $\delta$ $C^1$-small, we can make $\Phi_t^{h \circ \rho+\delta}$ $C^0$-close to $\Phi_t^{h \circ \rho}$ for all $t\in [0,1]$; in particular we can ensure that all fixed points of $\Phi_1^{h \circ \rho+\delta}$ lie in some $N_p$. 
By taking a generic such $\delta$, we can ensure that all orbits of $h \circ \rho+\delta$ are nondegenerate. 
By taking $N_p$ small, we may ensure that any orbit $\gamma$ of $h \circ \rho+\delta$ is $C^0$-close to an orbit $\bar{\gamma}$ of $h \circ \rho$. 
When the orbits are sufficiently $C^0$-close, we can construct a cylinder $v: S^1 \times [0,1] \to M$ stretching between $\gamma$ and $\bar{\gamma}$, so that $v(\cdot,t)$ is the unique geodesic from $\gamma(t)$ to $\bar{\gamma}(t)$; concatenating with this cylinder defines a natural bijection between caps for $\gamma$ and $\bar{\gamma}$.  
In order to arrange \eqref{it:act-approx} we must bound the symplectic area of the cylinder. 
This is achieved by observing that
$$\int_{S^1 \times [0,1]} v^* \omega = \int_{S^1 \times [0,1]} \omega\left(\frac{\partial v}{\partial s},\frac{\partial v}{\partial t} \right),$$
and $\partial v/\partial s$ can be made arbitrarily small while $\partial v / \partial t$ is bounded.

Now we arrange \eqref{it:ind-approx}. 
Recall that $\CZ(\bar{\gamma},\bar{u})$ is by definition that Conley--Zehnder index of the path of symplectic matrices $\Psi_t (D\Phi_t^{h \circ \rho}) \Psi_t^{-1}$, where $\Psi_t$ is a trivialization of $\bar{\gamma}^* TM$ induced by the cap $\bar{u}$ and $\CZ(\gamma,u)$ is the Conley--Zehnder index of the corresponding path of symplectic matrices. 
By making $\delta$ $C^2$-small, we can make $\Phi_t^{h \circ \rho+\delta}$ $C^1$-close to $\Phi_t^{h \circ \rho}$ for all $t \in [0,1]$;\footnote{This means that given $\eta>0$, we may choose $\delta$ so that for all $(m,v) \in TM$ with $|v| \le 1$, we have 
$$d\left(D\Phi_t^{h \circ \rho+\delta}(m,v),D\Phi_t^{h \circ \rho}(m,v)\right) < \eta$$ for an a priori fixed Riemannian metric on $TM$.} this implies that the aforementioned paths of symplectic matrices can be made $C^0$-close; the result now follows by \cite[Corollary 4.9]{McLean2016}.
\end{proof}

\begin{rmk}
Our approach to perturbing degenerate orbits follows \cite{McLean2016}.
With more effort one can prove a more precise result: one can find a Morse--Bott perturbation $H$, whose orbits are precisely the orbits of $h \circ \rho$ corresponding to critical points of a Morse function defined on the manifold with corners (and increasing at the boundary), and are nondegenerate.  
The technique for doing this goes back to \cite[Proposition 2.2]{CFHW}, see also \cite[Section 3.3]{Oancea2004} and \cite{KvK}. 
These references all deal with closed manifolds of orbits; the case of manifolds with corners is addressed in \cite{Ganatra2020}, in a setting closely related to ours.
\end{rmk}

\subsection{Action computation}

We start with a preliminary lemma which will be used in our action computation below. We state this lemma in a much more general setup than we need and after the proof make some comments to explain how we will specialize it.

\begin{lem}\label{lem:prelim-act}
Let $(M,\omega)$ be a symplectic manifold and $\pi: M\to \R^k$ be a smooth map. Let $f:\R^k\to \R$ be a smooth function. Let $\phi_t$ be the Hamiltonian flow of $\tilde{f}:=\pi^*f$. Consider a map $$u: [0,1]\times [0,1]\to M$$ such that for all $(t,s)\in [0,1]\times [0,1],$ $$u(t,s)=\phi_t( u(0,s)).$$ Moreover, we assume that $u([0,1]\times \{0\}) = \{A\}$ and $u([0,1]\times \{1\}) = \{B\}$, where $A$ and $B$ are points in $\R^k$. We orient $[0,1]\times [0,1]$ is so  that $\partial_t,\partial_s$ is a positive basis.

Then, the symplectic area of $u$ is equal to $f(B)-f(A)$. 
\end{lem}

\begin{proof}
This is an elementary computation.\begin{align*}
\int_{[0,1] \times [0,1]} u^*\omega&=\int^1_0\int^1_0\omega(u_*\partial_t,u_*\partial_s)dsdt\\
&=\int^1_0\int^1_0\omega(X_{\tilde{f}},u_*\partial_s)dsdt\\
&=\int^1_0\left(\int^1_0df(\pi_*u_*\partial_s)ds\right)dt\\
&=\int^1_0\left(\int_{\{t\}\times [0,1]}(\pi\circ u\circ\iota_t)^*df\right)dt\\
&=\int^1_0\left(f(B)-f(A)\right)dt\\
&= f(B) - f(A)
\end{align*}as required.
\end{proof}

Note that the assumption on the boundary of $u$ is automatic if $\pi$ is involutive; even more specifically, when $\pi$ is a moment map for a Hamiltonian torus action. Also note that if $f$ is an affine function, then $f(B)-f(A)$ is equal to the linear part of $f$ evaluated at the vector $\overrightarrow{AB}$ considered as an element of $\mathbb{R}^k$. If $\pi$ is a moment map for a Hamiltonian $(\R/\Z)^k$-action, and $f$ is integral affine, then $u$ as in the statement of the lemma satisfies $$u(0,s)=u(1,s),\text{ for all }s\in [0,1].$$We will only use this special case of the lemma below, where $u$ can also be thought of as a map $\R/\Z\times [0,1] \to M$. As a final remark that will be relevant, note that the blow down map $$\R/\Z\times [0,1]\to \mathbb{D}\subset \C\text{, where }(t,s)\mapsto se^{2\pi i t}$$ is orientation reversing, where we use the standard orientation of $\C$.

Let us now get back to the action computation that we wanted to undertake, continuing the notation used in the previous section.

There is a canonical cap $u_{out}$ associated to any orbit $\gamma$ of $h \circ \rho$, which we now describe. 
If $I(\gamma) = \emptyset$, then $\gamma$ is a constant orbit. 
We define $u_{out}$ to be the constant cap in this case. 
Otherwise, $\gamma$ is contained in $UD_{I(\gamma)}^{max}$. 
If $\gamma$ is contained in $UD_{I(\gamma)}$ then it is contained in an admissible standard chart, and we define $u_{out}$ to be the cap contained in that chart. 
Note that $u_{out}$ is well-defined by Lemma \ref{lem:adm-uniq}.

Note that if $\gamma$ is an orbit on $D$, it is contained in $UD_{I(\gamma)}$. 
For an orbit $\gamma$ not contained in $D$, we define $u_{out}$ to be the union of the cylinder swept by $\gamma$ along the Liouville flow taking it into $UD_{I(\gamma)}$, with the canonical cap in an admissible chart. 

At this point the reader might also benefit from looking at Remark \ref{rmk-weakly-admissible}, which gives a simpler version of admissibility and suffices for the purposes of this paper. It works because of the following Lemma.

\begin{lem}\label{lem:act}
The action of the $1$-periodic orbit $\gamma$ of $h \circ \rho$ with respect to the outer cap is given by
\[ \cA(\gamma,u_{out}) = h(\rho(\gamma)) + \sum_i \nu^h_i(\gamma)\cdot r_i^{max}(\gamma).\]
\end{lem}
\begin{proof}
The action is
\begin{align*}
\cA(\gamma,u_{out}) &= \int_{S^1} h (\rho(\gamma(t)) + \int_{u_{out}} \omega.
\end{align*}
The first term is $ h(\rho(\gamma))$, because $h(\rho(\gamma(t)) = h(\rho(\gamma))$ is constant along $\gamma$. 
We claim that the second term is
$$ \omega(u_{out}) = \sum_i \nu^h_i(\gamma)\cdot r_i^{max}(\gamma). $$

Consider the map 
\begin{align*}
f:\mathbb{R}^I&\to\mathbb{R}\\
f(r) &= \sum_{i\in I}-\nu^h_i(\gamma)\cdot r_i.
\end{align*} Notice that $\gamma$ is a one periodic orbit of the Hamiltonian vector field of $\tilde{f} :=  f\circ r^{max}_I$ (see Lemma \ref{lem-nu-exist}).

We break $u_{out}$ into two pieces: the piece $u_{out,1}$ lying in an admissible chart, and the piece $u_{out,2} = \cup_{t \in [0,T]} \varphi_t(\gamma)$ swept out by the Liouville flow. Assume that the boundary of $u_{out,1}$ is contained in $r_I^{-1}((a_i)_{i\in I})$.

Using the symplectic embedding of the admissible chart into $\C^{I}\times\C^{n-|I|}$, we see that $\int_{u_{out,1}} \omega$ is equal to the symplectic area of an arbitrary cap of a $1$-periodic orbit of  $X_{\tilde{f}}$ contained inside the fiber above $(a_i)_{i\in I}$ of the moment map $\C^{I}\times\C^{n-|I|}\to \mathbb{R}^I$. Choosing the cap obtained by radially scaling the loop to the origin inside the slice $\C^{I}\times\{c\}$ that it is contained in, we immediately obtain (e.g. using Lemma \ref{lem:prelim-act}): $$\int_{u_{out,1}} \omega = -\left(\sum_{i\in I} -\nu_i^h(\gamma) \cdot a_i\right).$$

For the area of the second piece, we use Lemma \ref{lem:prelim-act} for the map $r_I^{max}$, function $f$ and map $u_{out,2}$  to obtain:
$$\int_{u_{out,2}} \omega = \sum_{i\in I} \nu_i^h(\gamma) \cdot (r_i^{max}(\gamma)-a_i).$$
Note that here we used the $r_I^{max}$-relatedness of the Liouville vector field and the Euler vector field (i.e., Lemma \ref{lem:maxversion}).

Putting the computations together, we get the desired result.
\end{proof}

We define the fractional inner cap $u_{in} := u_{out} - \nu^h(\gamma) \cdot \bflambda$ as in Section \ref{subsec:SH}. Strictly speaking we do not need the following result for our argument, but we thought it was informative. 
Note that it is a slight generalization of the well-known formula in \cite[Section 1.2]{Vit}, which gives the result for SH-type orbits.

\begin{lem}\label{lem-act-final}
The action of the orbit $\gamma$ of $h \circ \rho$ with respect to the inner cap is given by
\[ \cA(\gamma,u_{in}) = h(\rho(\gamma)) - h'(\rho(\gamma))\cdot \rho(\gamma).\]
\end{lem}
\begin{proof}
By Lemma \ref{lem:act}, setting $\rho = \rho(\gamma)$, we have
\begin{align*}
\cA(\gamma,u_{in}) & = h(\rho) + \sum_{i} \nu^h_i(\gamma)\cdot r_i^{max}(\gamma) - \nu^h(\gamma) \cdot \bflambda\\
&= h(\rho) - h'(\rho) \sum_{i} \nu_i(\gamma) \cdot (r_i(\gamma) - \lambda_i) \\
&= h(\rho) - h'(\rho) \cdot \tilde{Z}_{I}\left(\tilde{\rho}_{I} \right)_{r_{I}^{max}(\gamma)}\\
&= h(\rho) - h'(\rho) \cdot \rho 
\end{align*}
where the last step follows as $\tilde{Z}_I(\tilde{\rho}_I) = \tilde{\rho}_I$ and $\tilde{\rho}_I \circ r_I^{max} = \rho$.
\end{proof}

\subsection{Index computation}

\begin{lem}\label{lem:ind-deg}
Let $\gamma$ be an orbit of $h \circ \rho$, with $J:= \{j \in I(\gamma): r_j^{max}(\gamma) \neq 0\}$. 
Define the $|J| \times |J|$ matrix
$$Hess_\gamma := \left(\frac{\partial^2 (h \circ \tilde{\rho}_I)}{\partial r_i \partial r_j} (r_I(\gamma))\right)_{i,j \in J}.$$
Then the Conley--Zehnder index of the orbit $\gamma$ of $h \circ \rho$ with respect to the outer cap is given by
$$ \CZ(\gamma,u_{out}) = 2\sum_i \left\lceil \nu^h_i(\gamma) \right\rceil + \frac{1}{2} \mathrm{sign}\left(Hess_\gamma\right).$$
\end{lem}
\begin{proof}
For constant orbits the result is easy, so we assume that $\gamma$ is nonconstant. 
We may assume that $\gamma$ and $u_{out}$ lie in an admissible chart $\C^{I(\gamma)} \times \C^{n-|I(\gamma)|}$, as the index does not change as we flow along the Liouville flow. 
The flow of $h \circ \rho$ in the admissible chart decomposes as a product of the flow
\[ \varphi_t(r,\theta) = (r,\theta +2\pi t\tilde{\nu}^h(r))\]
on $\C^{I(\gamma)}$ (written in action-angle coordinates) with the trivial flow on $\C^{n-|I(\gamma)|}$. 
Thus $\CZ(\gamma,u^{out}) = \CZ(D\varphi_t)$.  
We have
\begin{align*}
\CZ(D\varphi_t) &= \CZ\left(diag\left(e^{2 \pi i t \cdot \nu^h(z)}\right) \cdot  \left( \mathbf{1}  +2\pi i t \cdot \left[\frac{\partial \nu^h_j(z)}{\partial z_i}\right] \right)\right) \\
&= \CZ\left(diag\left(e^{2 \pi i t \cdot \nu^h(z)}\right)\right) + \CZ\left(diag\left(e^{2 \pi i  \cdot \nu^h(z)}\right) \cdot  \left( \mathbf{1}  +2\pi i t \cdot \left[\frac{\partial \nu^h_j(z)}{\partial z_i}\right] \right)\right)
\end{align*}
by a standard argument (c.f. \cite[Section 3.3]{Oancea2004}). 
The first term is equal to $2 \sum_i \left\lceil \nu^h_i(\gamma) \right\rceil$ (see \cite[Section 3.2]{Oancea2004}). 
For the second, we decompose $\C^{I(\gamma)} = \C^J \oplus \C^{I(\gamma) \setminus J}$. 
Note that $\partial \nu^h_j/\partial z_i = 0$ for $i \notin J$, because $r_i$ has vanishing derivative along $\{z_i=0\}$, where our orbit is contained. 
Also note that $e^{2\pi i \cdot \nu^h_i(z)} = 1$ for $i \in J$. 
Putting these together, one finds that the second term is equal to the Conley--Zehnder index of the path $\mathbf{1}_{J} + 2\pi i t \cdot [\partial \nu^h_j/\partial z_i]_{i,j \in J}$. 
Writing this in the basis given by action-angle coordinates (i.e., $\left(r_i\partial/\partial r_i,\partial/\partial \theta_i\right)_{i \in J}$), we see that it takes the form of a symplectic shear, whose Conley--Zehnder index is equal to 
\[ \CZ\left( \begin{array}{cc}
				\mathbf{1} & -2\pi t \cdot Hess_\gamma \\
				0 & \mathbf{1} 
		\end{array} \right) = \frac{1}{2} \mathrm{sign}\left(Hess_\gamma\right)
		\]
by the `normalization' property of the Conley--Zehnder index, see \cite[Theorem 4.1]{Robbin1993}.\footnote{The signature of a symmetric matrix is the number of positive eigenvalues minus the number of negative eigenvalues.}
\end{proof}

\begin{lem}\label{lem:ind}
Let $\gamma$ be an orbit of $H$ which corresponds to an orbit $\bar{\gamma}$ of $h \circ \rho$ as in Lemma \ref{lem:pertH}. 
Then we have
$$ i(\gamma,u_{out}) = 2\sum_i \left\lceil \nu^h_i(\bar{\gamma}) \right\rceil + \delta(\gamma),$$
where $0 \le \delta(\gamma) \le 2n$.
\end{lem}
\begin{proof}
We apply Lemmas \ref{lem:pertH} and \ref{lem:ind-deg}. 
Continuing the notation from the proof of the latter, we have 
$$\ker \left( D\varphi_1 - \id\right) = \C^{n - |I(\gamma)|} \oplus \C^{I(\gamma) \setminus J} \oplus \langle \partial/\partial \theta_j \rangle_{j \in J} \oplus \ker \left(Hess_{\bar{\gamma}}\right).$$
Recall that $k(\bar{\gamma})$ is, by definition, the dimension of this space. 
Thus we have
$$k(\bar{\gamma}) + \left|\mathrm{sign}\left(Hess_{\bar{\gamma}} \right)\right| \le 2n.$$
Combining the stated Lemmas, we have
\begin{align*}
\left| \CZ(\gamma,u^{out}) - 2\sum_i \left\lceil \nu^h_i(\bar{\gamma})\right\rceil -  \frac{1}{2} \mathrm{sign}\left(Hess_{\bar{\gamma}}\right)\right| & \le \frac{k(\bar{\gamma})}{2} \\
\Rightarrow \left| \CZ(\gamma,u^{out}) - 2\sum_i \left\lceil \nu^h_i(\bar{\gamma}) \right\rceil  \right| &\le \frac{2n}{2} = n.
\end{align*}
Recalling that $i(\gamma,u^{out}) := n + \CZ(\gamma,u^{out})$, the result is immediate.
\end{proof}

\begin{lem}\label{lem-ind-final}
Let $\gamma$ be an orbit of $H$ which corresponds to an orbit $\bar{\gamma}$ of $h \circ \rho$ as in Lemma \ref{lem:pertH}, and suppose that $h'(\rho) \ge 0$ everywhere, so that $\nu^h_i(\bar{\gamma}) \ge 0$ for all $i$. 
Then we have
$$i(\gamma,u_{in}) \ge \sum_i (2-\lambda_i) \cdot \nu^h_i(\bar{\gamma}).$$
In particular, when Hypothesis \ref{hyp:gr} is satisfied, we have $i(\gamma,u_{in}) \ge 0$.
\end{lem}
\begin{proof}
By Lemma \ref{lem:ind}, we have
\begin{align*}
i(\gamma,u_{in}) &\ge \sum_i 2\left\lceil \nu^h_i(\bar{\gamma})\right\rceil - \lambda_i \cdot \nu^h_i(\bar{\gamma}) \\
&\ge \sum_i (2-\lambda_i) \cdot \nu^h_i(\bar{\gamma})
\end{align*}
as required.
\end{proof}

\begin{lem}\label{lem:mindex}
Let $\gamma$ be an orbit of $H$ which corresponds to an orbit $\bar{\gamma}$ of $h \circ \rho$ as in Lemma \ref{lem:pertH}. 
Then we have
$$ i_{mix}(\gamma) = \sum_i (2- \kappa^{-1} r_i^{max}(\bar{\gamma})) \cdot \nu^h_i(\bar{\gamma}) - \kappa^{-1} h(\rho(\bar{\gamma}))+D(\gamma),$$
where $D(\gamma)$ is bounded: in particular, the lower bound is $D(\gamma) \ge -\kappa^{-1} \eps(\gamma)$, where $\eps(\gamma)$ is as in Lemma \ref{lem:pertH}.
\end{lem}
\begin{proof}
The equality follows by using the outer cap to compute the mixed index, via Lemmas \ref{lem:pertH}, \ref{lem:act}, and \ref{lem:ind}.
\end{proof}


\section{Proofs}\label{s-proofs}

In this section we prove Theorems \ref{thm:aa}, \ref{thm:specseqa}, and \ref{thm-superheavy}. 
We will assume throughout that the divisor $D$ is orthogonal, although that is not a hypothesis of Theorems \ref{thm:aa} and \ref{thm:specseqa}; the general results follow using Remark \ref{rmk:Dorth}.

Because $D$ is orthogonal (and in particular admits an admissible system of commuting Hamiltonians), we can make all of the constructions from the previous section, whose notation and assumptions (e.g. Equation \ref{eq-small-rad}) we continue. Right before Section \ref{ss-Horbs} we had started omitting the dependence on $R\in (0,R_0)$ from the notation for brevity, now we bring it back.

\subsection{Properties of $\tilde{\rho}^R_I$}\label{ss-rhoprop}

When we talk about a property ($n$) of $\tilde{Y}^R_I$ below, we mean the properties from Lemma \ref{lem:construct-Y}.

\begin{lem}\label{lem:rho-D-est}
There is a continuous function $\eps_1:[0,, R_0) \to \R_{\geq 0}$, with $\eps_1(0) = 0$, such that for all $R \in (0,R_0)$, all $I$, and all $r \in \tilde{Y}^0_I$, we have
$$1 \le \tilde{\rho}^R_I(r) \le 1+\eps_1(R).$$
\end{lem}
\begin{proof}
Note that $\tilde{Y}^R_I$ is sandwiched between $\tilde{Y}^0_I$ and $(R/2,\ldots,R/2) + \tilde{Y}^0_I$; hence it is also sandwiched between $\tilde{Y}^0_I$ and $\alpha \cdot \kappa \bflambda + \tilde{Y}^0_I$, where $\alpha = R/ (2\kappa\min \lambda_i)$. 
It follows that 
$$1 \le \tilde{\rho}^R_I(r) \le \frac{1}{1-\alpha}$$
for $r \in \tilde{Y}^0_I$, which gives the desired result.
\end{proof}

\begin{lem}
There is a continuous function $\eps_2:[0,R_0) \to \R_{\geq 0}$, with $\eps_2(0) = 0$, such that for all $R \in (0,R_0)$, all $I$, and all $r \in \tilde{Y}^0_I$, we have
\begin{equation}
\label{eq:bound-nu-1}
1 \le \sum_i \kappa\lambda_i \cdot \tilde{\nu}^R_{I,i}(r) \le 1+\eps_2(R).
\end{equation}
\end{lem}
\begin{proof}
Because $\tilde{Z}_I\left(\tilde{\rho}^R_I\right) = \tilde{\rho}^R_I$ by construction, we have
\begin{equation}
\label{eq:Zrho-gives}
\sum_i (\kappa \lambda_i - r_i) \cdot \tilde{\nu}^R_{I,i}(r) = \tilde{\rho}^R_I(r).
 \end{equation}
Thus Lemma \ref{lem:rho-D-est} gives
$$1 \le \tilde{\rho}^R_I(r) = \sum_i (\kappa \lambda_i - r_i) \cdot \tilde{\nu}^R_{I,i}(r) \le \sum_i \kappa \lambda_i \cdot \tilde{\nu}^R_{I,i}(r),$$
where the last step uses the fact that $\tilde{\nu}^R_{I,i}(r) \ge 0$ by property \eqref{it:Y-patch1} of $\tilde{Y}^R_I$, and $r_i \ge 0$ for all $i$.

For the right-hand bound, observe that $\tilde{\nu}^R_{I,i}(r) = 0$ whenever $r_i > R/2$, by property \eqref{it:Y-patch3} of $\tilde{Y}^R_I$; as $\tilde{\nu}^R_{I,i} \ge 0$ this implies that 
$$\sum_i (\kappa \lambda_i - R/2) \cdot  \tilde{\nu}^R_{I,i}(r)\le \sum_i (\kappa \lambda_i - r_i) \cdot \tilde{\nu}^R_{I,i}(r)  = \tilde{\rho}^R_I(r) \le 1+\eps_1(R).$$
Thus we have
$$ \sum_i \kappa\lambda_i \cdot \tilde{\nu}^R_{I,i}(r) \le \left(\max_i \frac{\kappa \lambda_i}{\kappa \lambda_i - R/2}\right) \cdot (1+\eps_1(R)) $$
where the RHS converges to $1$ as $R \to 0$, as required.
\end{proof}

The following Lemma will be used in the proof of Theorem \ref{thm:aa}:

\begin{lem}\label{lem-B}
There exists a continuous function $\s^B_{crit}:[0,R_0) \to \R_{\ge 0}$, with $\s^B_{crit}(0) = \s_{crit}$ (recall Definition \ref{def-crit}), such that for all $R \in (0,R_0)$, all $I$, and all $r \in \tilde{Y}^0_I$, we have
$$\sum_i \left(2-\kappa^{-1}r_i\right) \cdot \tilde{\nu}^R_{I,i}(r) - \kappa^{-1} \left(\tilde{\rho}^R_I(r)  - \s^B_{crit}(R)\right) > 0.$$
\end{lem}
\begin{proof}
Note that by property \eqref{it:Y-patch2} of $\tilde{Y}^R_I$, if $\tilde{\nu}^R_{I,i}(r) \neq 0$ and $r \in \tilde{Y}^0_I$ then $r_i\le R/2$. 
Combining this observation with Lemma \ref{lem:rho-D-est}, we have
$$\sum_i \left(2-\kappa^{-1}r_i\right) \cdot \tilde{\nu}^R_{I,i}(r) - \kappa^{-1} \tilde{\rho}^R_I(r) \ge \sum_i \left(2-\frac{R}{2\kappa}\right) \cdot \tilde{\nu}^R_{I,i}(r) - \kappa^{-1} \cdot (1+\eps_1(R)).$$
Dividing the left-hand bound in \eqref{eq:bound-nu-1} by $\max_i \kappa \lambda_i$ immediately gives
$$ \sum_i \left(2-\frac{R}{2\kappa}\right) \cdot \tilde{\nu}^R_{I,i}(r) \ge \frac{2-\frac{R}{2\kappa}}{\max_i \kappa \lambda_i}.$$
Thus we may take
$$\tilde{\s}^B_{crit}(R) = 1+\eps_1(R) - \frac{2-\frac{R}{2\kappa}}{\max_i \lambda_i},\qquad \s^B_{crit}(R) = \max\left(0,\tilde{\s}^B_{crit}(R)\right),$$
which clearly has the desired properties.
\end{proof}

The following Lemma will be used in the proof of Theorem \ref{thm-superheavy}:

\begin{lem}\label{lem-D}
There exists a continuous function $\s^D_{crit}:[0,R_0) \to \R_{\ge 0}$, with $\s^D_{crit}(0) = \s_{crit}$, and a positive function $\eta: (0,R_0) \to \R_{>0}$, such that for all $R \in (0,R_0)$, all $I$, all $i \in I$, and all $r \in V_I$ satisfying $\tilde{\rho}^R_I(r) > \s^D_{crit}(R)$ and $\tilde{\nu}^R_{I,i}(r) \neq 0$, we have 
$$2-\kappa^{-1}r_i \ge\color{black} \eta(R) \cdot \left(\tilde{\rho}^R_I(r) - \s^D_{crit}(R)\right).$$
\end{lem}
\begin{proof}
Suppose that the flowline of $\tilde{Z}_I$ passing through $r$ exits $\tilde{Y}_I^0$ at $r'$. 
Because both $\tilde{\rho}_I^R$ and $r_i - \kappa \lambda_i$ vary linearly along flowlines of $\tilde{Z}_I$, we have
$$ \frac{\tilde{\rho}_I^R(r)}{\tilde{\rho}^R_I(r')} = \frac{r_i - \kappa \lambda_i}{r'_i - \kappa \lambda_i},$$ 
and therefore
$$ 2 -\kappa^{-1} r_i = 2 - \lambda_i + \frac{\tilde{\rho}^R_I(r)}{\tilde{\rho}^R_I(r')} \cdot \left(\lambda_i - \frac{r'_i}{\kappa}\right).$$
Now by property \eqref{it:Y-patch3} of $\tilde{Y}^I_R$, if $\tilde{\nu}^R_{I,i}(r) \neq 0$ then $r$ lies in the region $P_I^{-1}(\{r_i \le R/2\})$, and therefore $r'_i \le R/2$.  
We also have $\tilde{\rho}_I^R(r') \le 1+\eps_1(R)$ by Lemma \ref{lem:rho-D-est}. 
It follows that 
$$2-\kappa^{-1}r_i \ge\color{black} 2-\lambda_i + \frac{\tilde{\rho}^R_I(r) \cdot \left(\lambda_i - \frac{R}{2\kappa} \right)}{1+\eps_1(R)}.$$
Now let us set
$$\tilde{\s}^D_{crit}(R) = \max_i \frac{(\lambda_i - 2)\cdot (1+\eps_1(R))}{\lambda_i - \frac{R}{2\kappa}};$$
then we find that the functions
\begin{align*}
\s^D_{crit}(R) &= \max\left(0,\tilde{\s}^D_{crit}(R)\right),\\
\eta(R) &= \min_i \frac{\lambda_i - \frac{R}{2\kappa}}{1+\eps_1(R)}
\end{align*} 
have the desired properties.
\end{proof}

\subsection{Proof of Theorem \ref{thm:aa}}\label{ss-thmaproof}

Let $R\in (0,R_0)$ be sufficiently small that $\s_{crit}^B(R)<1$.
Let $\s = \s_{crit}^B(R) + 2\delta < 1$, for some $\delta>0$.  
The proof will rely on a special choice of acceleration data for $K^R_\s$ (see Definition \ref{def:KRs}) which we now describe. 
Fix $0<\ell_1< \ell_2<\ldots$ such that the Reeb flow on $Y^R_\s=\partial K^R_\s$  has no $\ell_n$-periodic orbits for all $n$, and $\ell_n \to \infty$ as $n \to \infty$. (Here we take the contact form from Definition \ref{def:KRs}.)

\begin{figure}
\begin{center}
\includegraphics[width=0.7\textwidth]{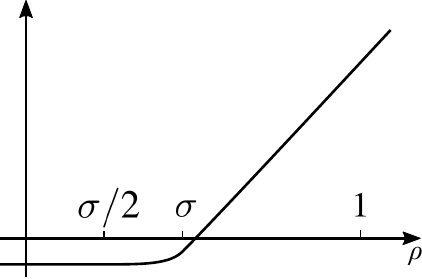}
\end{center}
\caption{The function $h_n$.}
\label{fig:B}
\end{figure}

We now choose smooth functions $h_n : \R \to \R$ approximating $\max(0,\ell_n(\rho - \s))$. 
We require that they each satisfy the conditions from Section \ref{ss-pert}, and furthermore:
\begin{itemize}
\item $h_1<h_2<\ldots$ (pointwise);
\item $h_n'(\rho) \ge 0$;
\item $h_n(\rho) = 2\delta_n$ for $\rho \le \s/2$;
\item $h_n(\rho) = \ell_n (\rho-\s) + \delta_n$ for $\rho \ge \s$,
\end{itemize}
where $\delta_n < 0$ converges monotonically to $0$ as $n \to \infty$, and furthermore $\ell_n \s - \delta_n < \ell_{n+1}\s - \delta_{n+1}$ for all $n$ (the latter condition will be used in the proof of Proposition \ref{prop-pos-inta}). See Figure \ref{fig:B}. 
Note that $h_n$ converges monotonically to $0$ on $(-\infty, \s]$ and $+\infty$ outside it. 
We extend $(h_n)_{n \in \Z_{\ge 1}}$ to $(h_\tau)_{\tau \in [1,\infty)}$ by convex interpolation: $h_\tau = (n+1-\tau)h_n + (\tau-n)h_{n+1}$ for $ \tau \in [n,n+1]$.  
We choose our acceleration data $(H_\tau,J_\tau)$ for $K^R_\s \subset M$, where $H_n$ is a perturbation of $h_n \circ \rho^R$ as in Lemma \ref{lem:pertH}, where the parameter $\eps$ in the Lemma is chosen smaller than $\ell_n \delta$, and $H_\tau$ is a corresponding perturbation of $h_\tau \circ \rho^R$. 
We further require that, over a `neck region' $\{\s \le \rho^R \le \s+\eps\}$ (where $\s<\s+\eps < 1$), we have $H_\tau = h_\tau \circ \rho^R$ and $J_\tau$ is of contact type. 

We denote by $\cC = \cC(H_\tau,J_\tau)$ the corresponding Floer 1-ray
$$CF^*(M,H_1;\Nov) \to CF^*(M,H_2;\Nov) \to \ldots,$$
so that $SC^*_M(K^R_\s;\Nov) = \widehat{tel}(\cC)$. 
By restricting $(H_\tau,J_\tau)$ to $K^R_\s$, we obtain acceleration data appropriate for defining the symplectic cochain complex of $K^R_\s$.
We denote by $\cC_{SH} := \cC(H_\tau|_{K^R_\s},J_\tau|_{K^R_\s})$ the corresponding Floer 1-ray
$$CF^*(K^R_\s,H_1|_{K^R_\s};\Bbbk) \to CF^*(K^R_\s,H_2|_{K^R_\s};\Bbbk) \to \ldots,$$
so that $SC^*(K^R_\s;\Bbbk) = tel(\cC_{SH})$.

By construction, the orbits of $H_n$ are either contained in $K^R_\s$ (in which case we say they are of $SH$-type), or contained in $M \setminus K^R_\s$ (in which case we say they are of $D$-type). 
We have a corresponding direct sum decomposition of $\Nov$-modules:
$$tel(\cC) = tel(\cC)_{SH} \oplus tel(\cC)_D.$$

Let us denote $SC_\Nov := (SC^*(K^R_\s;\Bbbk) \otimes_\Bbbk \Nov, d \otimes id_\Nov)$. 
We have the isomorphism 
\begin{align*}
\iota: SC_\Nov & \to tel(\mathcal{C})_{SH},\\ \iota(\gamma\otimes e^a) & = (\gamma,a\cdot u_{in}).
\end{align*}
By Lemma \ref{lem:same_action_index}, this map respects action and index.

\begin{lem}\label{lem:Dtypebound}
If $\gamma$ is a $D$-type orbit of $H_n$, then $i_{mix}(\gamma) \ge \kappa^{-1} \delta  \ell_n$.
\end{lem}
\begin{proof}
By Lemma \ref{lem:mindex}, we have
$$ i_{mix}(\gamma) \ge \sum_i \left(2- \kappa^{-1} r_i^{max}(\bar{\gamma})\right) \cdot \nu^h_i(\bar{\gamma}) - \kappa^{-1} \left(h\left(\rho^R(\bar{\gamma})\right) + \eps(\gamma)\right).$$
Note that as $\gamma$ is a $D$-type orbit, we have $\rho^R(\bar{\gamma}) \ge 1$ by Lemma \ref{lem:rho-D-est}, and therefore 
\begin{align*}
h\left(\rho^R(\bar{\gamma})\right) &\le \ell_n\left(\rho^R(\bar{\gamma}) - \s\right) \qquad \text{and}\\
h'\left(\rho^R(\bar{\gamma})\right) & = \ell_n.
\end{align*}
Thus we have $\nu^h_i(\bar{\gamma}) = \ell_n \cdot \tilde{\nu}^R_{I,i}(r_I^{max}(\bar{\gamma}))$. 
Setting $r=r_I^{max}(\bar{\gamma})$ (which lies in $\tilde{Y}^0_I$ because $\bar{\gamma}$ is a $D$-type orbit), and recalling that we chose $\eps(\gamma)<\ell_n\delta$, we obtain
\begin{align*}
 i_{mix}(\gamma) &\ge \ell_n \cdot \sum_i (2-\kappa^{-1} r_i) \cdot \tilde{\nu}^R_{I,i}(r) - \kappa^{-1} \ell_n (\tilde{\rho}_I^R(r) - \s)- \kappa^{-1} \ell_n \delta \\
 &\ge \kappa^{-1} \ell_n \cdot \left(\s - \s_{crit}^B(R) - \delta\right) =  \kappa^{-1}\delta\ell_n,
\end{align*} 
where the second inequality follows from Lemma \ref{lem-B}.
\end{proof}

We now consider the filtration on $tel(\mathcal{C})$ associated to the filtration map 
$$\cF'(\gamma,u):=\frac{\cA({\gamma},u)+\delta\ell_n}{\kappa}, $$ 
if $\gamma$ is a $1$-periodic orbit of $H_n$. 
It is clear that this is a filtration map, because the differential increases action, and it also increases $n$ and hence $\ell_n$ by the definition of the telescope complex. 
We define the corresponding filtration on $SC_\Nov$, associated to the filtration map
$$\cF(\gamma \otimes e^a) := \frac{\cA(\gamma) + \kappa a + \delta\ell_n}{\kappa}.$$
(Note that because $\iota$ respects index and action, we have $\cF' \circ \iota = \cF$.) 
For any cochain complex $C^*$, we define the quotient complex $\sigma_{<p} C^* := \oplus_{*<p} C^*$ with the induced differential. 

\begin{lem}\label{lem:cleanDorbits}
For any $p$, $\iota$ induces an isomorphism of graded $\cQ_{\ge 0} \Nov$-modules 
$$\iota: \sigma_{<p} \cF_{\ge p}SC_\Nov \xrightarrow{\sim} \sigma_{<p} \cF'_{\ge p}tel(\cC).$$
\end{lem}
\begin{proof}
It suffices to show that $\sigma_{<p} \cF'_{\ge p} tel(\cC)$ does not include any $D$-type orbits. 
Indeed, for any generator $(\gamma,u)$ of this complex, we have
$$i(\gamma,u) < p \le \frac{\cA({\gamma},u)+\delta\ell_n}{\kappa},$$
which means $\gamma$ cannot be of $D$-type, by Lemma \ref{lem:Dtypebound}.
\end{proof}

We now recall that $SC_\Nov$ comes equipped with the $\cQ$-filtration, induced by the $\cQ$-filtration on $\Nov$ (c.f. Equation \eqref{eqdef-SCNov}).

\begin{lem}\label{lem:Fpqiso}
For any $p$, the inclusions of the following subcomplexes are quasi-isomorphisms:
$$ \cF'_{\ge p} tel(\cC) \subset tel(\cC), \qquad \text{and} \qquad \cF_{\ge p} \cQ_{\ge q}SC_\Nov \subset \cQ_{\ge q} SC_\Nov$$
(the latter for any $q \in \Z \cup \{-\infty\}$).
\end{lem}
\begin{proof}
Observe that $\cF'_{\ge p} tel(\cC)$ is the telescope complex of the sub-1-ray of Floer groups 
$$\cA_{\ge \kappa p - \delta\ell_n} CF^*(M,H_n;\Nov) \subset CF^*(M,H_n;\Nov).$$
As the action filtration on each $CF^*(M,H_n;\Nov)$ is exhaustive, continuation maps increase action, and $\kappa p - \delta\ell_n \to - \infty$ as $n \to \infty$, the result follows by Lemma \ref{lem:limit_filt_qiso}. 
The argument for $\cF_{\ge p} \left(tel(\cC_{SH}) \otimes \cQ_{\ge q} \Nov\right) \subset tel(\cC_{SH}) \otimes \cQ_{\ge q} \Nov$ is identical.
\end{proof}

\begin{lem}\label{lem:trunc-PSS}
For any $p$, we have
$$H^j\left(\sigma_{<p}\cF'_{\ge p} tel(\cC)\right) \cong QH^j(M;\Nov) \qquad \text{ for $j<p-1$.}$$
\end{lem}
\begin{proof}
Applying Lemma \ref{lem:Fpqiso} and the PSS isomorphism, we have
$$H^j(\cF_{\ge p}tel(\cC)) = H^j(tel(\cC)) = \varinjlim_n HF^j(M,H_n;\Nov) = \varinjlim_n QH^j(M;\Nov) = QH^j(M;\Nov).$$
The result now follows as the degree truncation $\sigma_{<p}$ does not affect cohomology in degrees $< p-1$.
\end{proof}

We now denote
$$\left(SC^{(p)}_\Nov,d^{(p)}\right) := \sigma_{<p}\cF_{\ge p} (SC_\Nov,d \otimes id_\Nov).$$
For any $p>q$, we have a natural chain map $SC^{(p)}_\Nov \to SC^{(q)}_\Nov$, induced by the inclusion $\cF_{\ge p} \subset \cF_{\ge q}$ and the projection $\sigma_{<p} \twoheadrightarrow \sigma_{<q}$.
In particular we obtain an inverse system $\mathcal{SC_\Nov}$ of graded filtered $\cQ_{\ge  0}\Nov$-modules. 
We consider the `homotopy inverse limit' 
$$\wt{SC}_\Nov := \underset{\leftarrow}{tel}(\mathcal{SC_\Nov}),$$
(see Section \ref{ss-invtel} for the notation). 
We denote the differential by $\wt{d}$, and equip it with the filtration $\wt{\cQ}$ induced by $\cQ_{\ge \bullet}$ (see Remark \ref{rmk:filtinvlim}). 

We now make precise the notion of `filtered quasi-isomorphism' appearing in Theorem \ref{thm:aa} \eqref{it:1a}. 
We consider the category of $\Q$-graded filtered $\cQ_{\ge 0}\Nov$-cochain complexes $(M,d,\cQ_{\ge \bullet})$, where multiplication by $e^a$ increases the degree and the filtration level by $a$. 
Morphisms are $\cQ_{\ge 0}\Nov$-linear filtered chain maps. 
A morphism in this category is called a filtered quasi-isomorphism if it induces a quasi-isomorphism on each associated graded. 
Objects $M$ and $N$ are said to be filtered quasi-isomorphic if there exists a zigzag of filtered quasi-isomorphisms between them. 
This implies, in particular, that we have isomorphisms $H^j(\mathrm{Gr}_k^\cQ M) \cong H^j(\mathrm{Gr}_k^\cQ N)$ for all $j,k$.

\begin{lem}\label{lem:SCRfiltqi}
The filtered complex $(\wt{SC}_\Nov,\wt{d},\wt{\cQ}_{\ge \bullet})$ is filtered quasi-isomorphic to $(SC_\Nov,d \otimes id_\Nov,\cQ_{\ge \bullet})$ in the above sense.
\end{lem}
\begin{proof}
We have maps of inverse systems 
$$ \xymatrix{ SC_\Nov & SC_\Nov \ar[l]^-{id}& SC_\Nov \ar[l]^-{id} & \ldots \ar[l]\\
\cF_{\ge 0} SC_\Nov \ar[u]\ar[d] & \cF_{\ge 1}SC_\Nov \ar[u] \ar[d]\ar[l] & \cF_{\ge 2} SC_\Nov \ar[u] \ar[d]\ar[l] & \ldots \ar[l] \\
\sigma_{<0}\cF_{\ge 0} SC_\Nov  & \sigma_{<1}\cF_{\ge 1}SC_\Nov  \ar[l] & \sigma_{<2}\cF_{\ge 2} SC_\Nov \ar[l]&\ldots, \ar[l]
}$$
both of which induce a filtered quasi-isomorphism on the corresponding inverse telescope complex. For the upper map, this follows from Lemma \ref{lem:Fpqiso}. The lower map requires a little more argument. We first observe that $H^j(\mathrm{Gr}_k \sigma_{<p}\cF_{\ge p} SC_\Nov) \cong H^j(\mathrm{Gr}_k SC_\Nov)$ for $j<p-1$. It follows easily that for each $j$, the inverse system $H^j(\mathrm{Gr}_k \sigma_{<p}\cF_{\ge p}SC_\Nov)$ satisfies the Mittag-Leffler condition, so its $\varprojlim^1$ vanishes. 
Therefore, the cohomology of the $k$th associated graded of the inverse telescope of the bottom inverse system is 
$$ \varprojlim H^j(\mathrm{Gr}_k \sigma_{<p}\cF_{\ge p} SC_\Nov) = H^j(\mathrm{Gr}_k SC_\Nov),$$
by Lemma \ref{lem-invtelses}. This completes the argument. 

Finally, we observe that there is a filtered quasi-isomorphism from the inverse telescope of the top inverse system to $SC_\Nov$. 
Indeed, we take the composition
$$\underset{\leftarrow}{tel}\left(SC_\Nov \xleftarrow{id} SC_\Nov \xleftarrow{id} \ldots\right) \to \prod_{p \in \mathbb{N}} SC_\Nov \to SC_\Nov$$
where the first map is the natural one (i.e., the one appearing in the proof of Lemma \ref{lem-invtelses}), and the second map is given by projecting to any of the identical factors. 
Because this inverse system clearly satisfies the Mittag-Leffler condition, the proof of Lemma \ref{lem-invtelses} shows that the induced map on cohomology is the obvious isomorphism
$$\varprojlim_p H^*(SC_\Nov) \cong H^*(SC_\Nov).$$ 
Therefore the chain map is a quasi-isomorphism, and applying the same argument to the associated graded pieces shows that it is a filtered quasi-isomorphism.
This completes the necessary zig-zag of filtered quasi-isomorphisms.
\end{proof}

\begin{prop}[= Proposition \ref{prop-pos-int}]\label{prop-pos-inta}
For any Floer solution $u$ that contributes to $\mathcal{C}(H_\tau,J_\tau)$ with both ends asymptotic to $SH$-type orbits, we have $u\cdot\bflambda\geq 0.$ In case of equality, $u$ is contained in $K^R_\s$.
\end{prop}
\begin{proof}
Let $u:\R \times S^1 \to M$ be a pseudoholomorphic curve contributing to $\cC(H_\tau,J_\tau)$, with both ends asymptotic to $SH$-type orbits. 
We choose $\eps>0$ so that $u$ is transverse to $\partial K^R_{\s + \eps}$, and in a neighbourhood of $\partial K^R_{\s+\eps}$ we have that $H_\tau = h_\tau \circ \rho^R$ and $J_\tau$ is of contact type. 
We will apply Proposition \ref{prop-pre-pos-int} to the part of $u$ that lies in $\{\rho^R \ge \s+\eps\}$, to show that $u$ is contained in $K^R_{\s+\eps}$; applying the same argument to a sequence of such $\eps$ converging to $0$ will show that $u \subset K^R_\s$ as required.

We check the hypotheses of Proposition \ref{prop-pre-pos-int} one by one. 
First recall that we chose $J_\tau$ to be of contact type along $\partial K^R_{\s+\eps}$, so hypothesis \eqref{it:cont-typ} is satisfied. 

Now we check hypothesis \eqref{it:neck-cond}. 
We have $H_\tau = h_\tau \circ \rho^R$ in a neighbourhood of $\partial K^R_{\s+\eps}$. 
Thus $\cK = \left(h_{\psi(s)} \circ \rho^R\right) dt$ in this region, where $\psi(s)$ is either constant in the case of a Floer differential, or $\psi(s) = n$ for $s \ll 0$ and $\psi(s) = n+1$ for $s \gg 0$, in the case of a continuation map. 
Now observe that $h_n(\rho)$ is a linear function of $\rho$ for $\rho \ge \s$, and $h_\tau$ is obtained by linear interpolation from the $h_n$, hence is also linear in $\rho$; this establishes hypothesis \eqref{it:neck-cond}.

Finally we check hypothesis \eqref{it:pos-en}. 
We have $\cK = H_{\psi(s)}(t) dt$, so $d_\Sigma \cK = \partial_s H_{\psi(s)}(t)dt \ge 0$ as $H_\tau$ is increasing. 
Furthermore we have $\{\cK,\cK\}(\partial_s,\partial_t) = \{0,H(s,t)\} = 0$. 
It remains to address the term $d\beta$ appearing in the hypothesis. 
Observe that $h_n(\rho) = \ell_n(\rho - \s) + \delta_n = \alpha_n \rho+ \beta_n$, where we have arranged that the `constant terms' $\beta_n = -\ell_n \s + \delta_n$ are decreasing. 
We can extend $\beta_n$ to $\beta_\tau$ by linear interpolation, just as we did for $h_\tau$; this will clearly be a decreasing function of $\tau$. We then have $\beta = \beta_{\psi(s)}dt$, and it is clear that $d\beta \le 0$. Putting the three terms together,
$$d_\Sigma \cK - \{\cK,\cK\} - d\beta \ge 0,$$
verifying hyothesis \eqref{it:pos-en}.  
The result now follows by Proposition \ref{prop-pre-pos-int}.
\end{proof}

\begin{proof}[Proof of Theorem \ref{thm:aa}]
Item \eqref{it:1a} holds by Lemma \ref{lem:SCRfiltqi}. 
For item \eqref{it:2a}, we observe that $SC_\Nov^{(p)}$ comes equipped with another differential, namely the pullback of the differential on $\sigma_{<p} \cF'_{\ge p} tel(\cC)$ under the isomorphism of Lemma \ref{lem:cleanDorbits}, which we denote by $\partial^{(p)}$. 
The difference $d^{(p)} - \partial^{(p)}$ does not decrease the $\cQ$-filtration, by Proposition \ref{prop-pos-inta}. 
Any Floer solution $u$ contributing to the part of $\partial^{(p)}$ which preserves the $\cQ$-filtration must satisfy $u \cdot \bflambda = 0$, and hence be contained in $K^R_\s$ by Proposition \ref{prop-pos-inta}. These are precisely the Floer solutions contributing to $d^{(p)}$, so in fact $d^{(p)}-\partial^{(p)}$ strictly increases the $\cQ$-filtration. 
The maps in the inverse system are clearly chain maps for the differentials $\partial^{(p)}$, so $\wt{SC}_\Nov$ admits a corresponding differential, which we denote by $\partial$; and $\wt{d} - \partial$ strictly increases the $\wt{\cQ}$-filtration, by the corresponding property of $d^{(p)} - \partial^{(p)}$. 

For item \eqref{it:3a}, we observe that Lemma \ref{lem:trunc-PSS} implies that the inverse system $H^j(\sigma_{<p}\cF'_{\ge p} tel(\cC))$ has the Mittag-Leffler property, and in particular has $\varprojlim^1 = 0$. 
Therefore we have
\begin{align*}
H^j(\wt{SC}_\Nov,\partial) & \cong \varprojlim H^j(\sigma_{<p}\cF'_{\ge p} tel(\cC)) \qquad\text{by Lemma \ref{lem-invtelses}}\\
&\cong QH^j(M;\Nov)\qquad \text{by Lemma \ref{lem:trunc-PSS} again.}
\end{align*}
\end{proof}

\subsection{Proof of Theorem \ref{thm:specseqa}}

In order to fit with the standard terminology for spectral sequences, in which filtrations are assumed to be increasing (see \cite[Chapter 5]{Weibel}), we turn the decreasing filtrations $\wt{\cQ}_{\ge \bullet}$ into increasing ones by setting $\bar{\cQ}_{j} = \wt{\cQ}_{\ge -j}$.

\begin{lem}\label{lem-Q-bb}
Suppose that Hypothesis \ref{hyp:gr} holds. 
Then the $\bar{\cQ}$-filtration on $\wt{SC}_\Nov$ is bounded below. (Recall that this means that for each $i$, there exists $q(i)$ such that $\bar{\cQ}_{q(i)} \wt{SC}^i_\Nov = 0$.)\end{lem} 
\begin{proof}
If $i(\gamma \otimes e^a) = i$, then
\[ a_0\wt{\cQ}(\gamma \otimes e^a) = a = i(\gamma \otimes e^a) - i(\gamma) \le i\] 
by Lemma \ref{lem-ind-final}. Thus we may take $q(i) = \lfloor -i/a_0\rfloor -1$. 
\end{proof}

\begin{proof}[Proof of Theorem \ref{thm:specseqa}]
We start by establishing that the inclusion
$$\left(\bigcup_q \bar{\cQ}_{q} \wt{SC}_\Nov,\partial\right) \subset \left(\wt{SC}_\Nov,\partial\right)$$
is a quasi-isomorphism. 
This follows as
\begin{align*}
H^*\left(\bigcup_q \bar{\cQ}_{q} \wt{SC}_\Nov,\partial\right) & = \varinjlim_q H^*\left(\bar{\cQ}_{q} \wt{SC}_\Nov,\partial\right) \qquad \text{as direct limit commutes with cohomology} 
\\
&= H^*\left(\wt{SC}_\Nov,\partial\right) \qquad \text{by Lemma \ref{lem:Q-exh} below.}
\end{align*}
 
The spectral sequences induced by these filtered complexes are identical (this follows immediately from the construction). 
The $\bar{\cQ}$-filtration on $\bigcup_q \bar{\cQ}_{q} \wt{SC}_\Nov$ is exhaustive by construction, and bounded below by Lemma \ref{lem-Q-bb}. 
Therefore the corresponding spectral sequence converges to $H^*\left(\wt{SC}_\Nov,\partial\right)$ by \cite[Theorem 5.5.1]{Weibel}; and this is isomorphic to $QH^*(M;\Nov)$ by Theorem \ref{thm:aa} \eqref{it:3a}. 

Now we identify the $E_1$ page. 
By definition we have $E_0^{j,k} = \mathrm{Gr}_{j}^{\bar{\cQ}} \wt{SC}^{j+k}_\Nov$, and $d_0^{j,k}$ is the differential induced by $\partial$. 
The latter is equal to the differential induced on the associated graded by $\widetilde{d}$, by Theorem \ref{thm:aa} \eqref{it:2a} (combined with the fact that any cylinder $u$ satisfying $u \cdot \bflambda > 0$ satisfies $u \cdot \bflambda \ge a_0$). 
Therefore $\left(E_0^{j,k},d_0^{j,k}\right) = \left(\mathrm{Gr}_j^{\bar{\cQ}} \wt{SC}^{j+k}_\Nov,\widetilde{d}\right)$, which is quasi-isomorphic to $\left(\mathrm{Gr}_j^{\bar{\cQ}} SC^{j+k}_\Nov,d \otimes id_\Nov\right) $ by Lemma \ref{lem:SCRfiltqi}. 
Observe that $\mathrm{Gr}_j^{\bar{\cQ}} \Nov$ is spanned by $q^{-j}$, and hence is concentrated in degree $-ja_0$. 
It follows that $E_1^{j,k} = SH^{j(1+a_0)+k}(K^R_\s;\Bbbk) \otimes_\Bbbk \Bbbk \cdot q^{-j}$ as claimed.
\end{proof}

\begin{lem}\label{lem:Q-stab}
The map
$$H^j\left (\cQ_{\ge q} SC_\Nov^{(p)},\partial^{(p)}\right) \to H^j\left (\cQ_{\ge q} SC_\Nov^{(r)},\partial^{(r)}\right)$$
is an isomorphism, for all $p \ge r > j+2$.
\end{lem}
\begin{proof}
Let $(C,\partial)$ be the cone of the chain map $\left(\cQ_{\ge q} SC_\Nov^{(p)},\partial^{(p)}\right) \to \left (\cQ_{\ge q} SC_\Nov^{(r)},\partial^{(r)}\right)$. 
The $\cQ$-filtration on $C$ is bounded below by Lemma \ref{lem-Q-bb}, and it is clearly bounded above by $q$. 
Therefore the corresponding spectral sequence converges to the cohomology of $(C,\partial)$ by \cite[Theorem 5.5.1]{Weibel}. 

The $E_1$ page is the cohomology of the cone of the chain map $\left(\cQ_{\ge q} SC_\Nov^{(p)},d^{(p)}\right) \to \left (\cQ_{\ge q} SC_\Nov^{(r)},d^{(r)}\right)$. 
This cone coincides with the cone of the chain map $\cQ_{\ge q} \cF_{\ge p}SC_\Nov \to \cQ_{\ge q}\cF_{\ge r}SC_\Nov$ in degrees $<r-1$. 
The latter cone is acyclic, by Lemma \ref{lem:Fpqiso}. 
Therefore $E_1^{j,k} = 0$ for $j+k < r-2$. 
Because the spectral sequence converges, this means $H^j(C,\partial) = 0$ for $j<r-2$.
This implies the result.
\end{proof}

\begin{lem}\label{lem:Q-exh}
The natural map
$$\varinjlim_q H^*\left(\bar{\cQ}_q \wt{SC}_\Nov,\partial\right) \to H^*\left(\wt{SC}_\Nov,\partial\right)$$ 
is an isomorphism.
\end{lem}
\begin{proof}
Note that 
$$\left(\cQ_{\ge q} \wt{SC}_\Nov,\partial\right) = \underset{\leftarrow}{tel} \left(\cQ_{\ge q} SC^{(p)},\partial^{(p)}\right).$$
The inverse system $H^j(\cQ_{\ge q} SC^{(p)},\partial^{(p)})$ has the Mittag-Leffler property for all $j,q$, by Lemma \ref{lem:Q-stab}, so 
\begin{align*}
H^j\left(\cQ_{\ge q} \wt{SC}_\Nov,\partial\right) &\cong \varprojlim_p H^j \left(\cQ_{\ge q} SC^{(p)},\partial^{(p)}\right) \qquad \text{by Lemma \ref{lem-invtelses}}\\
&\cong H^j \left(\cQ_{\ge q} SC^{(p)},\partial^{(p)}\right) \qquad \text{for any $p>j+2$, by Lemma \ref{lem:Q-stab}.}
\end{align*}
A similar argument, using Lemma \ref{lem:trunc-PSS}, shows that
$$H^j\left(\wt{SC}_\Nov,\partial\right) = H^j\left(SC^{(p)}_\Nov,\partial^{(p)}\right) \qquad \text{for any $p>j+1$. }$$

Therefore we have an identification
$$\xymatrix{\varinjlim H^j\left(\bar{\cQ}_q \wt{SC}_\Nov,\partial\right) \ar[r] \ar@{=}[d] & H^j\left(\wt{SC}_\Nov,\partial\right) \ar@{=}[d]\\
\varinjlim H^j\left(\cQ_{\ge q} SC^{(p)}_\Nov,\partial^{(p)}\right) \ar[r] & H^j\left(SC^{(p)}_\Nov,\partial^{(p)}\right),}$$
for any $p>j+2$. 
The bottom map is an isomorphism, because the $\cQ$-filtration on $SC^{(p)}$ is exhaustive.
\end{proof}

\subsection{Proof of Theorem \ref{thm-superheavy}}\label{ss-proof-superheavy}

The key to the proof is the following:

\begin{prop}
	\label{prop:index_bd_collar_invt}
	Let $\s_{crit}^D(R)<\s_1<\s_2<1$. Then there exists an isomorphism
	$$
	SH^*_M\left(\overline{M\setminus K^R_{\s_1}};\Nov\right)\cong SH^*_M\left(\overline{M \setminus K^R_{\s_2}};\Nov\right).
	$$
\end{prop}

The proof relies on the Contact Fukaya Trick of \cite[Section 4]{Tonkonog2020}, with which we assume some familiarity. 

We first describe a special choice of acceleration data for $\overline{M\setminus K^R_{\s_2}} \subset M$. 
Let $\eps>0$ be such that $0<\s_2-2\eps$, $\s_2+3\eps < 1$, and $(\s_1/\s_2) \cdot (\s_2 - 2\eps) > \s_{crit}^D(R)$. 
Let $0<\ell_1<\ell_2<\ldots$ and $\delta_1<\delta_2<\ldots <0$ be reals such that the Reeb flow on $Y^R_{\s_2}=\partial K^R_{\s_2}$ has no $\ell_n$-periodic orbits or $\delta_n$-periodic orbits for all $n$, and $\ell_n \to \infty$, $\delta_n \to 0$ as $n \to \infty$.

\begin{figure}[ht]
\centering
\begin{minipage}[b]{0.45\linewidth}
\includegraphics[width=0.9\textwidth]{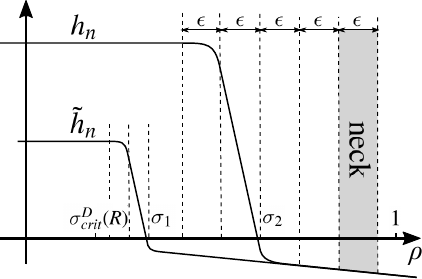}
\caption{The functions $h_n$ and $\tilde{h}_n$.}
\label{fig:D}
\end{minipage}
\begin{minipage}[b]{0.45\linewidth}
\includegraphics[width=0.9\textwidth]{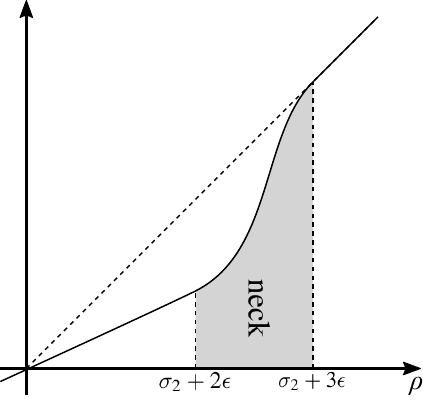}
\caption{The function $f$.}
\label{fig:D2}
\end{minipage}
\end{figure}

We choose smooth functions $h_n:\R \to \R$ satisfying:
\begin{itemize}
\item $h_1<h_2<\ldots$ (pointwise);
\item $h_n'(\rho) \le 0$;
\item $h_n(\rho) = \ell_n \eps$ for $\rho\le \s_2-2\eps$;
\item $h_n(\rho) = -\ell_n(\rho - \s_2) + \delta_n$ for $\s_2-\eps \le \rho \le \s_2$;
\item $h_n(\rho) = \delta_n \rho$ for $\rho \ge \s_2+\eps$.
\end{itemize}
Note that $h_n$ converges monotonically to $0$ on $[\s_2,\infty)$ and to $+\infty$ outside it.
We define $\tilde{h}_n(\rho) = \frac{\s_1}{\s_2}h_n\left(\frac{\s_2}{\s_1}\rho\right)$, and observe that $\tilde{h}_n$ converges monotonically to $0$ on $[\s_1,\infty)$ and to $+\infty$ outside it. See Figure \ref{fig:D}.

We extend $h_n$ to $h_\tau$ by linear interpolation as before, and make a choice of acceleration data $(H_\tau,J_\tau)$ for $\overline{M\setminus K^R_{\s_2}}$ such that $H_n$ is a perturbation of $h_n \circ \rho^R$ in accordance with Lemma \ref{lem:pertH}. 
We assume that $H_n>0$ over the region $\{\rho^R \le \s_2-\eps\}$ (we can arrange this so long as $h_n \circ \rho^R > 0$ over this region, which is true so long as the $\delta_n$ are chosen sufficiently small). 
We need to make some special assumptions over the `neck' region $\{\s_2+2\eps \le \rho^R \le \s_2+3\eps\}$, 
which make the contact Fukaya trick work: first we assume that $J_\tau$ is of contact type over the neck (this includes the assumption that $J_\tau$ is invariant under translation by the Liouville vector field); second we assume that the perturbation term $H_\tau - h_\tau \circ \rho^R$ vanishes over the neck, which is possible as $h_n \circ \rho^R = \delta_n \rho^R$ has no periodic orbits over this region. 

We now choose a smooth function $\ginv: \R \to \R$ satisfying:\footnote{Our function $f$ corresponds to the function $g^{-1}$ from \cite{Tonkonog2020}.}
\begin{itemize}
\item $\ginv'(\rho)>0$;
\item $\ginv(\rho) \le \rho$;
\item $\ginv(\rho) = \frac{\sigma_1}{\s_2} \cdot \rho$ for $\rho \le \s_2+2\eps$;
\item $\ginv(\rho) = \rho$ for $\rho \ge \s_2+3\eps$.
\end{itemize}
See Figure \ref{fig:D2}. 
We then define a diffeomorphism $\phi:M \to M$ by:
$$ \phi(m) = \left\{ \begin{array}{ll}
									\varphi_{\log \left(\frac{\ginv\left(\rho^R(m)\right)}{\rho^R(m)}\right)}(m) & \text{for $m \in X$;}\\
									m & \text{for $\rho^R(m) > \s_2+3\eps$},
										\end{array}\right.$$
where $\varphi_t:X \to X$ denotes the time-$t$ Liouville flow. 
The definition is chosen so that $\rho^R(\phi(m)) = \ginv(\rho^R(m))$. 
Note that $\phi$ sends $K^R_{\s_2}$ to $K^R_{\s_1}$ via the time-$\log(\s_1/\s_2)$ Liouville flow. 

We now define acceleration data $(\tilde{H}_\tau,\tilde{J}_\tau)$ for $\overline{M \setminus K^R_{\s_1}}$ by taking
\begin{align*}
\tilde{J}_\tau &= \phi_*J_\tau;\\
\tilde{H}_\tau &= \left\{\begin{array}{ll}
										\frac{\s_1}{\s_2} \phi_* H_\tau & \text{on $\phi\left(\{\rho^R \le \s_2+3\eps\}\right)$};\\
										\phi_* H_\tau & \text{on $\phi\left(\{\rho^R \ge \s_2+2\eps\}\right)$.}
											\end{array}\right.
\end{align*}
Note that the definition of $H_\tau$ agrees on the overlaps, using the fact that $H_\tau = \delta_\tau \rho^R$ and $\phi = \varphi_{\s_1/\s_2}$ over this region.
Furthermore, we observe that $\phi_*X_{H_\tau} = X_{\tilde{H}_\tau}$. 
(This is relatively easy to check on the complement of the image of the neck region $\phi(\{\s_2+2\eps \le \rho^R \le \s_2+3\eps\})$; over the neck region it uses the fact that both $H_\tau$ and $\tilde{H}_\tau$ are equal to $\delta_\tau\rho^R$.) 
The fact that $J_\tau$ is of contact type over the neck ensures that $\phi_* J_\tau$ is $\omega$-compatible.

Thus we have constructed acceleration data for  $\overline{M\setminus K^R_{\s_2}}$ and $\overline{M\setminus K^R_{\s_1}}$ leading to Floer $1$-rays 
$$\cC_{\s_2} := \cC(H_\tau,J_\tau) \qquad \text{and} \qquad \cC_{\s_1} := \cC(\tilde{H}_\tau,\tilde{J}_\tau),$$
such that the map $(\gamma,u) \mapsto \phi(\gamma,u) :=(\phi \circ \gamma, \phi \circ u)$ defines an isomorphism $\cC_{\s_2} \xrightarrow{\sim} \cC_{\s_1}$, which however need not respect the action filtrations. 
We want to prove that this map of $1$-rays induces an isomorphism of the completed telescopes $$\widehat{tel}(\mathcal{C}_{\s_2})\to \widehat{tel}(\mathcal{C}_{\s_1}).$$ 

\begin{lem}\label{lem:keybound}
There exist constants $B$, $\eta>0$, $C$ such that
$$-i_{mix}(\gamma) -B \ge -i_{mix}(\phi(\gamma)) \ge \eta \cdot (-i_{mix}(\gamma)) + C$$
for any orbit $\gamma$ of $H_n$, and the corresponding orbit $\phi(\gamma)$ of $\tilde{H}_n$.
\end{lem}
\begin{proof}

First we show that $-i_{mix}(\gamma) -B \ge -i_{mix}(\phi(\gamma))$, for some $B>0$ that we specify below.
Note that $i(\gamma,u_{out}) = i(\phi(\gamma),\phi(u)_{out})$, so it suffices to show $\cA(\gamma,u_{out}) -B \ge \cA(\phi(\gamma),\phi(u)_{out})$. 
Let $(\bar{\gamma},\bar{u}_{out})$ be a capped orbit of $h_n \circ \rho^R$ corresponding to $(\gamma,u_{out})$ under Lemma \ref{lem:pertH}. 
Then we have
\begin{equation}\label{eq:act-form} \cA(\gamma,u_{out}) = h_n(\rho^R(\bar{\gamma})) + \sum_i \nu^h_i(\bar{\gamma})\cdot r_i^{max}(\bar{\gamma}) + \eps(\gamma)\end{equation}
by Lemma \ref{lem:act}, where $\eps(\gamma)$ is bounded, and similarly for $\phi(\gamma,u_{out})$. 
We consider the first term on the RHS. 
Note that orbits occur either in the region $\{\rho^R \le \s_2-\eps\}$, in which case $\tilde{h}_n(\rho^R(\phi(\bar{\gamma}))) = \frac{\s_1}{\s_2}h_n(\rho^R(\bar{\gamma}))< h_n(\rho^R(\bar{\gamma})$, because $h_n(\rho^R(\bar{\gamma}))>0$ (we ensured this positivity when choosing our perturbation); or in the region $\{\rho^R \ge \s_2\}$, where both $h_n(\rho^R(\bar{\gamma})) = \delta_n \rho^R(\bar{\gamma})$ and $\tilde{h}_n(\rho^R(\phi(\bar{\gamma}))) = \delta_n \rho^R(\phi(\bar{\gamma}))$ lie in the bounded interval $(\delta_1 \cdot (1+\eps_1(R)),0)$. 
In either case, we have $h_n(\rho^R(\bar{\gamma})) \ge \tilde{h}_n(\rho^R(\phi(\bar{\gamma}))) + B'$ for some fixed $B'$. 
For the second term on the RHS of \eqref{eq:act-form}, note that $h'(\rho) \le 0$, so $\nu^h_i(\bar{\gamma}) \le 0$. 
We have $\nu^h_i(\bar{\gamma}) = \nu^h_i(\phi(\bar{\gamma}))$, and $r_i^{max}(\bar{\gamma}) \le r_i^{max}(\phi(\bar{\gamma}))$ (here we use our assumption that $\ginv(\rho) \le \rho$, as well as the fact that $Z(r_i^{max})<0$). 
Together this yields $$\sum_i \nu^h_i(\bar{\gamma}) \cdot r_i^{max}(\bar{\gamma}) \ge \sum_i \nu^h_i(\phi(\bar{\gamma})) \cdot r_i^{max}(\phi(\bar{\gamma})).$$ 
Adding the bounds together, and taking $B>B'+2|\eps(\gamma)|$ for all $\gamma$, gives the result.

Now we consider the $-i_{mix}(\phi(\gamma)) \ge \eta \cdot (-i_{mix}(\gamma)) + C$ part of the statement.  
By Lemma \ref{lem:mindex}, we have
\begin{equation}\label{eq:mixindform}
 -i_{mix}(\gamma) = \sum_i -\nu^h_i(\bar{\gamma}) \cdot (2- \kappa^{-1} r_i^{max}(\bar{\gamma})) - \kappa^{-1}h_n(\rho^R(\bar{\gamma})) +D(\gamma)
 \end{equation}
where $|D(\gamma)|$ is bounded. 
We focus on the first term on the RHS. 
We start by recalling that $-\nu^h_i(\phi(\bar{\gamma})) = -\nu^h_i(\bar{\gamma})  \ge 0$. 
Note that if $\bar{\gamma}$ is non-constant, then $\rho^R(\bar{\gamma}) > \s_2-2\eps$, so $\rho^R(\phi(\bar{\gamma})) > \frac{\s_1}{\s_2}(\s_2 - 2\eps)>\s_{crit}^D(R)$. 
Therefore, by Lemma \ref{lem-D}, whenever $\nu^h_i(\bar{\gamma}) \neq 0$ we have 
$$2-\kappa^{-1} r_i^{max}(\phi(\bar{\gamma})) > 2\eta,$$
where $2\eta = \eta(R) \cdot (\frac{\s_1}{\s_2}(\s_2 - 2\eps)-\s_{crit}^D(R))>0$. 
As a result we have
$$\sum_i -\nu^h_i (\phi(\bar{\gamma})) \cdot (2- \kappa^{-1} r_i^{max}(\phi(\bar{\gamma}))) \ge \eta \cdot \sum_i -\nu^h_i(\bar{\gamma}) \cdot (2- \kappa^{-1} r_i^{max}(\bar{\gamma})).$$
Note that this inequality also holds for the constant orbits, as then we have $\nu^h_i(\bar{\gamma}) = \nu^h_i(\phi(\bar{\gamma})) = 0$.

Now we focus on the second term on the RHS of \eqref{eq:mixindform}. 
We saw in the first part of the proof that $\tilde{h}_n(\rho^R(\phi(\bar{\gamma}))) = \frac{\s_1}{\s_2} \cdot h_n(\rho^R(\bar{\gamma})) > 0$ for orbits with $\rho^R(\bar{\gamma})<\s_2-\eps$. 
Decreasing $\eta$ if necessary so that it is less than $\s_1/\s_2$, and recalling that $D(\gamma)$ is bounded, we obtain the desired bound for such orbits.  
For the remaining orbits we recall from the first part of the proof that both $\tilde{h}_n(\rho^R(\phi(\bar{\gamma})))$ and $h_n(\rho^R(\bar{\gamma}))$ are bounded. 
Therefore, decreasing $C$ if necessary, we obtain the desired bound for the remaining orbits.
\end{proof}

\begin{lem}\label{lem:A-equiv}
If $(\gamma_j,u_j)$ is a sequence of capped orbits of $H_{n_j}$ such that $i(\gamma_j,u_j) = i$ is constant, then
$$ \cA(\gamma_j,u_j) \to +\infty \qquad \iff \qquad \cA(\phi(\gamma_j,u_j)) \to + \infty.$$
\end{lem}
\begin{proof}
Lemma \ref{lem:keybound} gives
\begin{align*}
 \kappa^{-1}\cA(\gamma_j,u_j) - i - B &\ge  \kappa^{-1} \cA(\phi(\gamma_j,u_j)) - i \ge \eta \cdot (\kappa^{-1}\cA(\gamma_j,u_j) - i) + C \\
\Rightarrow \qquad \cA(\gamma_j,u_j) -\kappa B& \ge \cA(\phi(\gamma_j,u_j)) \ge \eta \cdot \cA(\gamma_j,u_j) + \kappa C + (1-\eta)\kappa i
 \end{align*}
where $\eta>0$, from which the result follows.
\end{proof}

\begin{rmk} Notice that the contact Fukaya trick that we presented here is simpler than the one in \cite{Tonkonog2020} (compare Figure \ref{fig:D} above with Figure 2 in \cite{Tonkonog2020}). We would like to stress that it is possible to use this simpler version because we are in a different situation. 
\end{rmk}

\begin{proof}[Proof of Proposition \ref{prop:index_bd_collar_invt}]
By Lemma \ref{lem:A-equiv}, the isomorphism $\cC_{\s_1} \cong \cC_{\s_2}$ induces an isomorphism of the corresponding degreewise-action-completed telescope complexes; so $$SC^*_M\left(\overline{M\setminus K^R_{\s_1}};\Nov\right) \cong SC^*_M\left(\overline{M\setminus K^R_{\s_2}};\Nov\right),$$ and the result follows by taking cohomology.
\end{proof}

We continue with the following observation of McLean:

\begin{prop}[see Proposition 6.20 of \cite{McLean2020}] 
\label{prop:D_Stably_disp}	
	Let $D$ be an SC divisor in a symplectic manifold $M$. Then $D$ is stably displaceable. \qed
\end{prop}

\begin{proof}[Proof of Theorem \ref{thm-superheavy}]
It follows from Proposition \ref{prop:D_Stably_disp} that a neighbourhood of our divisor $D$ is stably displaceable. 
Suppose that $R$ is sufficiently small that the domains $UD_i$ of our system of commuting Hamiltonians $\{r_i:UD_i \to [0,R)\}$ are contained in this stably displaceable neighbourhood, for all $i$. 
This ensures that $\overline{M \setminus K^R_{\s}}$ is contained in this neighbourhood for $\s$ sufficiently close to $1$. In particular, $$SH^*_M\left(\overline{M\setminus K^R_{\s}};\Nov\right)=0$$ for such $\s$, by Theorem \ref{analog_thesis2}. By Proposition \ref{prop:index_bd_collar_invt}, we see that in fact we have the same result for any $\s_{crit}^D(R)<\s<1$. 
This completes the proof using Theorem \ref{analog_thesis3}, as the sets $\left\{\overline{M\setminus K^R_{\s}}\right\}_{R>0,\s > \s_{crit}^D(R)}$ exhaust $M \setminus K_{crit}$ (this follows from the fact that $\rho^R \to \rho^0$ and $\s_{crit}^D(R) \to \s_{crit}$ as $R \to 0$). 
\end{proof}

\appendix

\section{Algebraic background}

\subsection{Filtration maps}\label{ss-filtration-maps}

In this section, we present an elementary framework to better deal with the type of filtrations that we encounter in this paper, which are in particular indexed by real numbers. 

A filtration map on an abelian group $A$ is a map $\rho: A\to \mathbb{R}\cup\{\infty\}$ satisfying the inequality $$\rho(x+y)\geq \min{(\rho(x),\rho(y))},$$  equality $\rho(x)=\rho(-x)$, and sending $0$ to $\infty$. A filtration map defines a filtration by the subgroups $$F_{\geq \rho_0}A:=\{\a\in A\mid \rho(a)\geq \rho_0\}.$$ 

 Note that if $(V_{\alpha}, \rho_{\alpha})$ are abelian groups equipped with filtration maps indexed by a set $\alpha\in I$, then $\bigoplus_{\alpha\in I} V_{\alpha}$ is equipped with a filtration map given by $$\rho\left(\sum v_i\right):=\min{\left(\rho_i(v_i)\right)}.$$ Let us call this the $min$ construction.

We can define a pseudometric on an abelian group $A$ with a filtration map $\rho$ by $d(a,a'):=e^{-\rho(a-a')}$. The completion $\widehat{A}$ of $A$ is defined by taking the abelian group of Cauchy sequences in $A$ and modding out by the subgroup of sequences which converge to $0$. $\widehat{A}$ is equipped with a canonical filtration map: $$\rho((a_i)_{i\in \mathbb{N}})=\varinjlim \rho(a_i).$$ We call $A$ complete,  if the natural map $A\to \widehat{A}$ is bijective. 

We define filtration maps $\cQ: \Nov\to \mathbb{R}$ by setting $\cQ(q^a)= a$ and using the $min$ construction. 

A filtration map on a $\mathbb{Z}$-graded $\Nov$-module $A$ is a filtration map for each $A^i$ which in addition is  additive for the module action by homogenous elements of $\Nov$. A filtration map on a $\Nov$-cochain complex $C$ is a filtration map on the underlying $\mathbb{Z}$-graded $\Nov$-module, which satisfies the condition that the differential does not decrease the filtration map. Let $F_{\geq \rho_0}C:=\bigoplus_{i\in\mathbb{Z}}F_{\geq \rho_0}C^i, $ which is of course nothing but the filtration associated to the filtration map on $C$ constructed by the min construction.

Filtered chain maps between $\Nov$-cochain complexes equipped with filtration maps are defined to be chain maps that do not decrease the values of the filtration maps. Filtered chain homotopies between filtered chain maps are defined in the same fashion.

\subsection{Quasi-isomorphic subcomplexes of the telescope}\label{ss-cofinal}

Let \begin{align}
\xymatrix{ 
\mathcal{C}:=\mathcal{C}_1\ar[r]^-{f_1}& \mathcal{C}_2\ar[r]^-{f_2} &\mathcal{C}_3\ar[r]^-{f_3}& \ldots }
\end{align}be a $1$-ray of $\mathbb{Q}$-graded chain complexes.

The telescope $tel(\mathcal{C})$ of $\mathcal{C}$ is defined to be the cone of the chain map $$id-f: \bigoplus_{i=1}^{\infty} \mathcal{C}_i\to \bigoplus_{i=1}^{\infty} \mathcal{C}_i.$$

Assume that we have a commutative diagram \begin{align*}
\xymatrix{ 
\mathcal{C}_1'\ar[r]\ar[d]& \mathcal{C}_2'\ar[d]\ar[r] &\mathcal{C}_3'\ar[r]\ar[d]& \ldots\\ \mathcal{C}_1\ar[r] &\mathcal{C}_2 \ar[r]&\mathcal{C}_3\ar[r]&\ldots}, 
\end{align*} where the vertical maps are inclusions of subcomplexes. We call the top $1$-ray $\mathcal{C}'$.

We obtain the commutative diagram of $\mathbb{Q}$-graded abelian groups \begin{align}\label{diag-limits}
\xymatrix{
H(tel(\mathcal{C}')) \ar[d]\ar[r]  &H(tel(\mathcal{C}))\ar[d] \\ H\left(\varinjlim(\mathcal{C}_i')\right) \ar[r] &H\left(\varinjlim  (\mathcal{C}_i)\right)}
\end{align} where the vertical maps are isomorphisms (see \cite[Lemma 2.2.2]{Varolgunes2018} for the proof).

\begin{lem}\label{lem:limit_filt_qiso} Assume that every element $\gamma$ of $\mathcal{C}_i$ lands inside $\mathcal{C}_{i+N(\gamma)}'$ for some $N(\gamma)>0.$ Then, $$tel(\mathcal{C}')\to tel(\mathcal{C})$$ is a quasi-isomorphism.
\end{lem}
\begin{proof}
Because direct limits commute with quotients, we have $$\varinjlim \mathcal{C}_i/\varinjlim \mathcal{C}_i'\simeq \varinjlim \mathcal{C}_i/\mathcal{C}_i',$$ as $\mathbb{Q}$-graded chain complexes. A basic property of filtered direct limits is that any element in $\varinjlim \mathcal{C}_i/\mathcal{C}_i'$ is in the image of the canonical map $\mathcal{C}_i/\mathcal{C}_i'\to \varinjlim \mathcal{C}_i/\mathcal{C}_i'$ for some $i>0$. This and the given condition implies that the direct limit on the RHS is zero. In particular, the lower horizontal map in Diagram \eqref{diag-limits} is also an isomorphism. This finishes the proof.
\end{proof}

\subsection{Completed telescopes}\label{sscomplete}

Let $FiltCh_{\Nov}$ be the category of free $\Q$-graded $\Nov$-cochain complexes equipped with a filtration map and morphisms given by filtered chain maps. 

Let \begin{align}
\xymatrix{ 
\mathcal{C}:=\mathcal{C}_1\ar[r]^-{f_1}& \mathcal{C}_2\ar[r]^-{f_2} &\mathcal{C}_3\ar[r]^-{f_3}& \ldots }
\end{align} be a $1$-ray in $FiltCh_{\Nov}$.  Let us equip $tel(\mathcal{C})$ with the filtration map obtained from the min construction. If we define $$\mathcal{C}_{A_0}=F_{\geq A_0} \mathcal{C}_1\to F_{\geq A_0} \mathcal{C}_2\to ...,$$
then by construction \begin{align}\label{eqnfiltrationtelescope} F_{\geq A_0}tel(\mathcal{C})=tel(\mathcal{C}_{A_0}).\end{align} 

We define maps between two $1$-rays in $FiltCh_{\Nov}$ as diagrams \begin{align}
\xymatrix{ 
\mathcal{C}_1\ar[r]\ar[d]& \mathcal{C}_2\ar[d]\ar[r] &\mathcal{C}_3\ar[r]\ar[d]& \ldots \\ \mathcal{C}_{1}'\ar[r] &\mathcal{C}_2'\ar[r]&\mathcal{C}_3'\ar[r]&\ldots}
\end{align}where the horizontal arrows are filtered chain maps and each square is equipped with a map $\mathcal{C}_i\to\mathcal{C}'_{i+1}$, which is a filtered chain homotopy between the two filtered chain maps $\mathcal{C}_i\to\mathcal{C}'_{i+1}$ obtained by composing the arrows at the edges of the square. The resulting category we call $1\text{-$ray$-}Ch_{\Nov}$. 

 In $1\text{-$ray$-}Ch_{\Nov}$ we also have a notion of two morphisms being {\it equivalent}, defined by the existence of a homotopy of maps of $1$-rays. The definition is identical to \cite{Varolgunes2018} except that here we require all the homotopy maps to not decrease the filtration values, instead of requiring them to be $\Lambda_{\geq 0}$-module maps.
 
The following is a direct analogue of the second bullet point of Lemma 2.1.9 in \cite{Varolgunes2018} for $n=1$. The proof is omitted. 

\begin{lem}\label{lemstrongtel} Let us start with a morphism in $1\text{-$ray$-}Ch_{\Nov}$  \begin{align}
\xymatrix{ 
\mathcal{C}_1\ar[r]\ar[d]& \mathcal{C}_2\ar[d]\ar[r] &\mathcal{C}_3\ar[r]\ar[d]& \ldots \\ \mathcal{C}_{1}'\ar[r] &\mathcal{C}_2'\ar[r]&\mathcal{C}_3'\ar[r]&\ldots}
\end{align}  Then there is an induced filtered chain map  $tel(\mathcal{C})\to tel(\mathcal{C}')$. Hence the telescope construction is a functor $$tel:1\text{-$ray$-}Ch_{\Nov}\to FiltCh_{\Nov}.$$

Moreover, equivalent morphisms in $1\text{-$ray$-}Ch_{\Nov}$ gives rise to filtered homotopy equivalent chain maps.
\end{lem}

Degreewise completion defines a functor $$\widehat{\cdot}: FiltCh_{\Nov}\to FiltCh_{\Nov}.$$

Let us call a chain map $C\to C'$ between $\Nov$-cochain complexes equipped with filtration maps a strong filtered quasi-isomorphism if it induces a quasi-isomorphism $$F_{\geq \rho_0}C\to F_{\geq \rho_0}C',$$ for every $\rho_0\in\mathbb{R}$. Because the filtrations are exhaustive, a strong filtered quasi-isomorphism is a quasi-isomorphism.

\begin{lem}\label{lemcomplete}Under the degreewise completion functor\begin{itemize}
\item a strong filtered quasi-isomorphism is sent to a strong filtered quasi-isomorphism.
\item a filtered chain homotopy is sent to a filtered chain homotopy.
\end{itemize}  
\end{lem}

\begin{proof}
The first bullet point follows from a spectral sequence comparison theorem. Precisely, we must show that the chain map $F_{\ge \rho_0}\widehat{C} \to F_{\ge \rho_0}\widehat{C}'$ is a quasi-isomorphism, for all $\rho_0 \in \R$. We consider the spectral sequences associated to these filtered complexes, and the map of spectral sequences between them associated to the strong filtered quasi-isomorphism. We observe that this map is an isomorphism on the $E_1$ page. To see this, we first observe that the map $\mathrm{Gr}_i^F C \to \mathrm{Gr}_i^F C'$ is a quasi-isomorphism for all $i$, using the long exact sequence associated to a short exact sequence of chain complexes. We have $\mathrm{Gr}_i^F \widehat{C} = \mathrm{Gr}_i^F C$, and similarly for $C'$, so the map $\mathrm{Gr}_i^F \widehat{C} \to \mathrm{Gr}_i^F \widehat{C}'$ is also a quasi-isomorphism; it then follows by construction that the map is an isomorphism on the $E_1$ page. The filtrations are both complete and exhaustive by construction, so the Eilenberg--Moore Comparison Theorem \cite[Theorem 5.5.11]{Weibel} gives the result.

The second bullet point follows from the fact that the completion is an additive functor. 
\end{proof}

\begin{rmk}
The first bullet point of Lemma \ref{lemcomplete} is not explicitly used in the present paper. It would be an input in the proof of the well-definedness of relative symplectic cohomology (Proposition \ref{prop:SHrelwelldef}), which we omitted.
\end{rmk}

\begin{lem}\label{unboundedcofinal}
Let $\mathcal{C}=\mathcal{C}_1\to \mathcal{C}_2\to \mathcal{C}_3\to\ldots$ be a $1$-ray in $FiltCh_{\Nov}$ with the following property: for any integer $i\geq 1$ and real number $r$ there exists a positive integer $N$ such that the composition of the $\Nov$-chain maps from the $1$-ray $C_i\to C_{i+N}$ increases the filtration map by at least $r$ for any element of $C_i$. Then, $\widehat{tel}(\mathcal{C})$ is acyclic.
\end{lem}
\begin{proof}
As in the proof of Lemma \ref{lemcomplete}, it suffices by the Eilenberg--Moore Comparison Theorem to show that $\mathrm{Gr}_i tel(\cC)$ is acyclic for all $i$. 
Using Equation \ref{eqnfiltrationtelescope} and elementary homological algebra, we obtain that $\mathrm{Gr}_i tel(\mathcal{C})$ is quasi-isomorphic to the telescope of 
$$ \mathrm{Gr}_i \cC_1 \to \mathrm{Gr}_i \cC_2 \to \ldots .$$
Because the homology of the telescope is isomorphic to the homology of the direct limit, it suffices to show the acyclicity of the direct limit of this diagram. It is easy to see that the direct limit is in fact trivial on the nose (i.e. at the chain level).
\end{proof}

\subsection{Homotopy inverse limit}\label{ss-invtel}

\begin{defn}
Let $\cC$ be an inverse system of cochain complexes and cochain maps:
$$  C^*_0 \xleftarrow{i_{01}} C^*_1 \xleftarrow{i_{12}} \ldots .$$
We define the cochain complex $\prod_p C^*_p$ to be the degreewise direct product of the $C^*_p$. 
There is a natural chain map $id - i: \prod_p C^*_p \to \prod_p C^*_p$, sending $(c_p) \mapsto (c_p - i_{p,p+1}(c_{p+1}))$. 
We define the inverse telescope complex
$$ \underset{\leftarrow}{tel}(\cC) := Cone\left(\prod_p C^*_p \xrightarrow{id - i} \prod_pC^*_p\right)[-1].$$
\end{defn}

The following recovers the Milnor exact sequence if $\cC$ satisfies the Mittag-Leffler condition. We believe that it is standard, but we could not locate it in the literature.

\begin{lem}\label{lem-invtelses}
There is a short exact sequence
$$ 0 \to \varprojlim{}\!^1 H^{j-1}\left(C^*_p\right) \to H^j\left(\underset{\leftarrow}{tel}(\cC)\right) \to \varprojlim H^j\left(C^*_p\right) \to 0.$$
\end{lem}
\begin{proof}
The long exact sequence associated to the short exact sequence of cochain complexes
$$ 0 \to \prod_p C^*_p[-1] \to \underset{\leftarrow}{tel}(\cC) \to \prod_p C^*_p  \to 0$$
gives an exact sequence
$$H^{j-1}\left(\prod_p C^*_p\right) \xrightarrow{[id - i]} H^{j-1}\left(\prod_p C^*_p\right) \to H^j\left(\underset{\leftarrow}{tel}(\cC)\right) \to H^j\left(\prod_p C^*_p\right) \xrightarrow{[id - i]} H^j\left(\prod_p C^*_p\right).$$
This gives the desired short exact sequence, as the $\varprojlim$ is defined to be the kernel of $[id - i]$, and $\varprojlim^1$ is defined to be the cokernel.
\end{proof}

\begin{rmk}\label{rmk:filtinvlim}
If $\cC$ is an inverse system of \emph{filtered} cochain complexes with \emph{filtered} cochain maps, then the inverse telescope complex acquires a filtration by
$$F_{\ge p}\left( \underset{\leftarrow}{tel}(\cC) \right) := \underset{\leftarrow}{tel}(F_{\ge p}\cC)$$
(it is clear how to regard the RHS as a subcomplex of $\underset{\leftarrow}{tel}(\cC)$).
\end{rmk}


\bibliographystyle{alpha}
\bibliography{./QHisSHD}


\end{document}